%% file: FinePartitions.tex
\documentclass[12pt,letterpaper]{amsart}

\usepackage{amscd}
\usepackage{amsmath,amsfonts,amsthm,epsfig,latexsym,graphicx,amssymb}
\usepackage[english]{babel}
\usepackage{comment}
\usepackage{epsfig}
\usepackage{tikz, graphics} 
\usepackage{graphics,color}
\usepackage{psfrag}
\usepackage{hyperref}
\usepackage{float}
\hypersetup{colorlinks=true, citecolor=green, linkcolor=blue, hypertexnames=false, urlcolor=cyan}

\newtheorem{coro}{Corollary}
\newtheorem{defi}{Definition}
\newtheorem{prop}{Proposition}

\newtheorem{theo}{Theorem}
\newtheorem{lemm}{Lemma}

\newtheorem{prob}{Problem}
\newtheorem{rema}{Remark}

\newtheorem{conv}{Convention}
\newtheorem{cons}{Construction}

\newtheorem*{coro*}{Corollary}
\newtheorem*{defi*}{Definition}
\newtheorem*{prop*}{Proposition}
\newtheorem*{conj*}{Conjecture}
\newtheorem*{theo*}{Theorem}
\newtheorem*{lemm*}{Lemma}
\newtheorem*{ques*}{Question}
\newtheorem*{exer*}{Exercise}
\newtheorem*{exem*}{Example}
\newtheorem*{prob*}{Problem}
\newtheorem*{rema*}{Remark}
\newtheorem*{conv*}{Convention}
\newtheorem*{clai*}{Claim}
\newtheorem*{affi*}{Affirmation}
\newtheorem*{claim*}{Claim}
\newtheorem*{exes*}{Exercise}
\newtheorem*{pbm*}{Problem}

\def\II{{\mathbb I}}  

 \def\NN{{\mathbb N}} 

 \def\RR{{\mathbb R}} \def\SS{{\mathbb S}}
\def\TT{{\mathbb T}}

 \def\ZZ{{\mathbb Z}}

  \def\cG{{\mathcal G}} \def\cM{{\mathcal M}} \def\cS{{\mathcal S}} 
\def\cB{{\mathcal B}}  \def\cH{{\mathcal H}}  \def\cT{{\mathcal T}} 
\def\cC{{\mathcal C}}  \def\cI{{\mathcal I}} \def\cO{{\mathcal O}} \def\cU{{\mathcal U}}
  \def\cJ{{\mathcal J}}  \def\cV{{\mathcal V}}
    \def\cW{{\mathcal W}}
\def\cF{{\mathcal F}}   \def\cR{{\mathcal R}}

\title[Markov partitions with prescribed combinatorics]{Geometric Markov partitions for pseudo-Anosov homeomorphisms with prescribed combinatorics}

\author{Inti Cruz Diaz}
\date{\today}
\email{incruzd@im.unam.mx}

\begin{document}

 \begin{abstract}
In this paper, we focus on the construction and refinement of geometric Markov partitions for pseudo-Anosov homeomorphisms that may contain spines. We introduce a systematic approach to constructing \emph{adapted Markov partitions} for these homeomorphisms. Throughout this process, we define key concepts such as geometric Markov partitions and their geometric types. Our primary result is an algorithmic construction of \emph{adapted Markov partitions} for every generalized pseudo-Anosov map. The algorithm constructs a Markov partition by starting from a single point. We then apply this algorithm to a distinguished class of points, the so-called \emph{first intersection points} of the homeomorphism, which produces \emph{primitive Markov partitions} that are well-behaved under iterations. Additionally, we manage to prove that the set of \emph{primitive geometric types}  of a given order is finite, providing a canonical and distinguished tool for the classification of pseudo-Anosov homeomorphisms, as will be further discussed in \cite{CruzDiaz2024}.\\
Then, we construct new geometric Markov partitions from existing ones while maintaining control over their combinatorial properties and without losing track of their geometric types. The first type of geometric Markov partition constructed is one with a binary incidence matrix, which permits the introduction of the sub-shift of finite type associated with the incidence matrix of any Markov partition—this is known as the \emph{binary refinement}. After that, we describe a process that allows us to cut any Markov partition  along a family of stable and unstable segments prescribed by a finite set of periodic codes, which we call the $s$ and $u$-boundary refinements. Finally, we present an algorithmic construction of a Markov partition where all periodic boundary points are located at the corners of the rectangles in the partition, referred to as the corner refinement. Each of these Markov partitions and their intrinsic combinatorial properties plays a crucial role in our algorithmic classification of pseudo-Anosov homeomorphisms up to topological conjugacy, as developed by the author in his thesis \cite{IntiThesis}, and can be found in an updated version in the article \cite{CruzDiaz2024}.
\end{abstract}

 \maketitle

\tableofcontents

\input{Intro/Intromain}

\input{Preliminares/Preliminaresmain}

\input{Existencia/Existencia}

\input{Refinamientos/MainRefinamientos}

\input{Refinamientos/SecBinaryref}

\input{Refinamientos/Seccoderef}

\input{Corner/Corner}

\bibliographystyle{alpha}
\bibliography{BibPseudoA}

\end{document}

%% file: Intro/Intromain.tex
\section{Introduction}\label{Section: Introduccion}
\subsection{\textbf{p-A} homeomorphisms and their classifications.}

For us, a \emph{surface} is a two-dimensional \emph{smooth manifold} (or at least $C^1$). It is closed if it is compact and has no boundary, and we shall restrict our study to the setting where $S$ is a closed and orientable surface. Let us not forget that $S$ is also a topological space, and therefore, apart from the set of diffeomorphisms of $S$, Diff$(S)$, we can consider its set of homeomorphisms, Hom$(S)$, and their analogues that preserve the orientation of $S$, Diff$_+(S)$ and Hom$_+(S)$, respectively. Clearly, Diff$(S) \subset \text{Hom}(S)$ and Diff$_+(S) \subset \text{Hom}_+(S)$. We are going to study a certain type of surface homeomorphisms, the \emph{generalized pseudo-Anosov homeomorphisms} (see \ref{Defi: pseudo-Anosov homeomorphism}), which, despite being only homeomorphisms, preserve a notion of \emph{uniform hyperbolicity} outside a finite family of points on the surface.

Pseudo-Anosov homeomorphisms are carefully introduced in Subsection \ref{Subsec: Foliaciones singulares y homeomorfismos pA}. Roughly speaking, $f \in \text{Hom}_+$ is a generalized pseudo-Anosov homeomorphism if it preserves a pair of transverse \emph{singular measured foliations} (see \ref{Defi: Foliacion medible}),  denoted by $(\cF^s, \mu^s)$ and $(\cF^u, \mu^u)$. Additionally, there exists a \emph{stretch factor} $\lambda > 1$ such that the $\mu^s$-measure of every compact interval $J$ contained in a leaf of the \emph{unstable foliation} $\cF^u$, but without singularities in its interior, is \emph{uniformly expanded} by $f$; i.e., $\mu^s(f(J)) = \lambda \mu^s(J)$. Simultaneously, any compact interval $I \subset F^s \in \cF^s$ contained in a leaf of the \emph{stable foliation} $\cF^s$, without singularities in its interior, is \emph{uniformly contracted} by a factor $\lambda^{-1}$; i.e., $\mu^u(f(I)) = \lambda^{-1} \mu^u(I)$. This uniform contraction and expansion is where the analogy with uniform hyperbolicity begins. 

In \cite{lanneau2017tell}, E. Lanneau provides a beautiful introduction through examples of (linear) pseudo-Anosov homeomorphisms, but let us compile some of the most well-known constructions of pseudo-Anosov homeomorphisms in the literature:
\begin{enumerate}\label{Prob: compute examples general}
	\item Anosov diffeomorphisms on $\mathbb{T}^2$ and their \emph{branched coverings}, examples constructed in \cite{fathi2021thurston}.
	\item The examples given by Arnoux and Yoccoz in \cite{ArYoccoz1981}, constructed for any surface of genus greater than or equal to $2$.
	\item In certain cases, it is possible to obtain a pseudo-Anosov homeomorphism by composing \emph{Dehn twists}. Some examples are presented in \cite{fathi2021thurston}. R. C. Penner has generalized such constructions in \cite{penner1988construction} to produce a broader family of homeomorphisms isotopic to pseudo-Anosov using Dehn twists. 
	\item All expansive homeomorphisms on compact surfaces are pseudo-Anosov homeomorphisms without spines (\cite{hiraide1987expansive} \cite{lewowicz1989expansive}).
\end{enumerate}

Each of these examples has a unique flavor; some have a topological nature, such as Penner's examples, while others are based on dynamical hypotheses, as seen in the works of Hiraide and Lewowicz. In each construction, the authors establish the existence of a pair of invariant transverse foliations using different techniques. Additional examples by Y. Verberne were constructed in \cite{verberne2023construction}, while T. Hall and A. Carvalho introduced some generalizations of the pseudo-Anosov homeomorphisms in \cite{Andre2005extensions} and \cite{Hall2004unimodal}. However, here we must restrict ourselves to the so-called \emph{canonical models} of generalized pseudo-Anosov homeomorphisms, though it is important to acknowledge the existence of such generalizations.

But where did pseudo-Anosov homeomorphisms first appear? They are central to the Nielsen-Thurston classification of orientation-preserving diffeomorphisms of surfaces up to isotopy. One of the basic concepts in algebraic topology is \emph{isotopy} between two homeomorphisms. It is well known that isotopy is an equivalence relation on $\text{Hom}+(S)$, and the set of isotopy classes of $\text{Hom}+(S)$, when equipped with the natural operation of multiplication on the left, forms the \emph{modular group} of the surface, $\cM\cC\cG(S)$. The Nielsen-Thurston classification theorem \cite{thurston1988geometry} states that any surface diffeomorphism is \emph{isotopic} to a unique homeomorphism: \emph{periodic}, pseudo-Anosov, or \emph{reducible}. In the reducible case, the surface can be decomposed into finite pieces along a family of closed (non-null-homotopic) curves, and in each piece, the homeomorphism is isotopic to one of the previously mentioned types. Thus, these homeomorphisms have played a central role in classification problems since the beginning of their study. Furthermore, M. Bestvina and M. Handel \cite{bestvina1995train} provided an algorithmic proof of the Nielsen-Thurston theorem by introducing the machinery of \emph{train tracks} and the\emph{ incidence matrix associated} with the action of a mapping class on such train tracks as the main combinatorial tool to to achieve such classification. Before continuing, let us clarify what we mean by a classification. The classification of homeomorphisms up to isotopy requires solving three main obstacles:

\begin{enumerate} 
	\item \textbf{The combinatorial presentation.} Associate to every homeomorphism a finite and combinatorial object in such a way that two elements are isotopic if and only if they are associated with the same combinatorial object, even if this association is not unique. For example, this is achieved through a certain incidence matrix for Bestvina and Handel. 
	\item \textbf{The realization problem:} Determine which of these combinatorial objects is realized by a homeomorphism: 
	\item \textbf{The conjugacy problem}: Determine when two different combinatorial objects are related to the same isotopy class. \end{enumerate}

The objective of Bestvina and Handel in the article \cite{bestvina1995train} is to provide an algorithmic proof of the Handel-Thurston classification theorem for homeomorphisms in $\text{Hom}_+(S, P)$. In other words, given a homeomorphism $f$, their goal is to determine whether $f$ belongs to the isotopy class of a reducible, periodic, or pseudo-Anosov homeomorphism. Almost simultaneously, Jerome Los exploits the intimate relation between the mapping class group of the $n$-punctured disk and the \emph{braid groups} in \cite{los1993pseudo} to provide a finite algorithm for determining whether an isotopy class (of the punctured disk) contains a pseudo-Anosov element. Their approach involves train tracks and a certain incidence matrix, whose eigenvalues they attempt to minimize after a series of operations on a subjacent train track. 

It is not necessarily clear from the work of Bestvina and Handel that they were able to solve the conjugacy problem. However, L. Mosher, in his 1983 thesis \cite{mosher1983pseudo} and the 1986 article \cite{mosher1986classification}, classifies pseudo-Anosov homeomorphisms on a given closed and oriented surface $S$ equipped with a finite set of marked points $P$ (where the spines are contained). To achieve this, Mosher considers the action of $\mathcal{M}(S, P)$ on a directed graph $\tilde{\mathcal{G}}(S, P)$, where each vertex represents a cellular decomposition of the surface with a single vertex and a distinguished sector around that vertex. The directed edges of $\tilde{\mathcal{G}}(S, P)$ correspond to elementary moves that transform one decomposition into another. This cellular decomposition, in some sense, serves as a simplified version of Thurston’s train tracks. Importantly, Mosher defines certain \emph{pseudo-Anosov invariants} that address the combinatorial presentation problem and provide partial answers to the realization and conjugacy problems. For example, Mosher’s classification requires considering certain powers of the homeomorphism and knowing the number and types of each singularity of the pseudo-Anosov homeomorphism.
Subsequently, M. Bell \cite{bell2015recognising} developed an algorithm and a computational program, the \emph{Flipper} \cite{flipperBell}, which, among other things, determines the isotopy class type of a given homeomorphism. In this case, in addition to \emph{train tracks}, another tool used by M. Bell is a particular type of ideal triangulations introduced by I. Agol in \cite{agol2011ideal} to study the $3$-manifold obtained by suspending a pseudo-Anosov homeomorphism, the so called \emph{mapping torus}.

In this article, we are interested in classifying \textbf{p-A} homeomorphisms not up to isotopy but up to an equivalence relation that is more natural in the setting of dynamical systems: topological conjugacy. Two homeomorphisms $f: S \to S'$ and $g: S \to S'$ are \emph{topologically conjugate} if there exists a homeomorphism $h: S \to S'$ such that $f \circ h = h \circ g$. In this case, any orbit of $g$ is mapped by $h$ to a homeomorphic orbit of $f$. Since the seminal paper of S. Smale \cite{smale1967differentiable}, we know this is the equivalence relation from which we can understand the qualitative dynamical behavior of a map or a flow. Thus, this paper is the first step toward solving the following general problem:

\begin{prob}\label{Prob: Conjugacion pseudo-Anosov}
	Classify the set of \emph{p-A} homeomorphisms on surfaces up to topological conjugacy.
\end{prob}

Since two pseudo-Anosov homeomorphisms that are isotopic are topologically conjugate (\cite{fathi2021thurston}), Mosher's  classification is a partial answer to our classification problem.  However, we will take a completely different approach from Bestvina, Handel, Los, and Mosher. Instead of working with train tracks, which are fairly stable objects under isotopy, and the incidence matrix of the isotopy class action on such a train track, we will use geometrized Markov partitions and, along with them, a novel combinatorial object—its geometric type—which serves as a generalization of the classical incidence matrix of a Markov partition. So let us discuss such objects before we start presenting our results.

\subsection{Markov partition for maps and flows}

One of the more useful tools to understand the dynamics of a map (or  flow) are the Markov partitions . Recall that a Markov partition (\ref{Defi: Geo Markov partition}) for a \textbf{p-A} homeomorphism $f: S \to S$ is a finite family of \emph{immersed rectangles} in $S$, $\cR=\{R_i\}_{i=1}^{n}$, with disjoint interiors, such that their union $\cup \mathcal{R} = \cup_{i=1}^n R_i$ is equal to $S$ and such that: if $\overset{o}{R_i}$ and $\overset{o}{R_j}$ denote the interiors of two rectangles in $\cR$, the closure of any connected component of $f^{-1}(\overset{o}{R_j}) \cap \overset{o}{R_i}$ is a \emph{horizontal sub-rectangle} of $R_i$, and similarly, the closure of any connected component of $f(\overset{o}{R_i}) \cap \overset{o}{R_j}$ is a \emph{vertical sub-rectangle} of $R_j$. It is well known (see \cite{fathi2021thurston}) that every pseudo-Anosov homeomorphism without spines has a Markov partition, and a synthetic construction for the generalized case can be found here. A \emph{geometric Markov partition} of $f$ is a Markov partition for which we have chosen a vertical orientation in each of the rectangles in the partition. To indicate that $\mathcal{R} = \{R_i\}_{i=1}^n$ is a geometric Markov partition for $f$, we will write $(f,\mathcal{R})$.

Markov partitions were defined and constructed for Anosov systems by Y. Sinai in \cite{sinai2020markov} and \cite{sinai1968construction}. Later, they were constructed for the basic sets of Axiom A homeomorphisms by R. Bowen in \cite{bowen1970markov}. Their ideas exploit the \emph{pseudo orbit tracing property} to obtain rectangles with sufficiently small diameters, making them the rectangles of a Markov partition for the system. The same approach was used by Hiraide in \cite{hiraide1985Markovhomeomorphisms} to construct Markov partitions for expansive surface homeomorphisms. 

Nowadays, there are beautiful results that can be proved using the methods of symbolic dynamics. In the words of R. Bowen in \cite{bowen1970markov}, once you have a Markov partition for an Axiom $A$ diffeomorphism, “(the map) $f$ is the quotient of a subshift of finite type. One then studies $f$ by studying the structure of this quotient; this is what we mean by symbolic dynamics."
 This approach, using Markov partitions and their associated subshifts of finite type, has been fruitful, allowing us to prove, for example, that every pseudo-Anosov homeomorphism has positive topological entropy, exhibits a dense set of periodic points, is topologically transitive, topologically mixing, and \emph{Bernoulli}. Starting with this manuscript and finishing with \cite{CruzDiaz2024}, following the ideas of C. Bonatti, R. Langevin, E. Jeandenans, and F. Béguin to classify invariant neighborhoods of $C^1$ structurally stable diffeomorphisms on surfaces, called \emph{Smale diffeomorphisms}, we add another beautiful application of Markov partitions: classifying, up to topological conjugacy, all pseudo-Anosov homeomorphisms. Let me add that recently, geometric Markov partitions were used in \cite{iakovoglou2022invariantcaracterizing} as a new invariant of conjugacy for transitive Anosov flows by I. Iakovoglou, in order to achieve a complete classification of such flows. Additionally, he has recently constructed Markov partitions for non-transitive expansive flows in \cite{iakovoglou2024markovpartitionsnontransitiveexpansive}, generalizing the work of Ratner for Anosov flows in \cite{ratner1973markov} and Brunella for expansive and transitive flows in \cite{brunella1992expansive}. Therefore, there is still a lot of research being conducted using geometric Markov partitions to classify other kinds of systems up to topological conjugacy.

In order to achieve an algorithmic classification of the conjugacy classes of surface homeomorphisms that contain pseudo-Anosov homeomorphisms in \cite{CruzDiaz2024}, we must rely on the existence of Markov partitions for pseudo-Anosov homeomorphisms that may contain a \emph{spine} and endow them with vertical orientation in each of their rectangles. This gives rise to the notion of \emph{geometric Markov partitions} and their corresponding \emph{geometric types} \ref{Defi: Geometric type of (f,cR)}, as introduced by C. Bonatti and R. Langevin in \cite{bonatti1998diffeomorphismes} for Smale diffeomorphism.  The geometric type of a geometric Markov partition is a combinatorial object associated with the pair $(f, \mathcal{R})$, which generalizes the \emph{incidence matrix} of a Markov partition (see \cite{BowenAnosovbook}), by considering not only the number of times the rectangles $f(\cR):=\{f(R_i)\}_{i=1}^n$ intersect transversally with the rectangles in $\mathcal{R} := \{R_j\}_{j=1}^n$, but also the order in which they intersect and the change in their relative vertical orientation.This information is contained in a 4-tuple:
$$
T(f,\mathcal{R}) := (n, \{h_i, v_i\}_{i=1}^n, \rho, \epsilon).
$$
Here, $n$ is the number of rectangles in the partition $\mathcal{R}$, $h_i$ and $v_i$ are the number of horizontal and vertical rectangles contained in $R_i$, respectively, and the map $\rho:\mathcal{H}(f,\mathcal{R}) \to \mathcal{V}(f,\mathcal{R})$ is a bijection between the finite sets of horizontal rectangles (see \cite{bonatti1998diffeomorphismes}). In the last chapters of \cite{bonatti1998diffeomorphismes}, C. Bonatti and E. Jeandenans show how to collapse the invariant manifolds of a basic piece of a Smale diffeomorphism to obtain a pseudo-Anosov homeomorphism, possibly with spines. This suggested that it was possible to adapt the ideas and constructions used by Bonatti, Langevin, and Béguin to classify pseudo-Anosov homeomorphisms under topological conjugacy. During my doctoral studies, I carried out such a classification, where, starting with a geometric Markov partition whose geometric type is known, it was necessary to construct new geometric Markov partitions for such \textbf{p-A} homeomorphisms with some prescribed combinatorial properties and without losing track of the geometric type obtained. This obeyed the necessity to define certain symbolic spaces or to simplify many technical arguments. Thus, we devote this article to such constructions. In the next subsection of this introduction, we will present the results found in this article and some comments about their application in \cite{CruzDiaz2024}, where the classification is carried out without concern for the intermediate constructions that we describe here, which mark the beginning of a fine algorithmic classification of pseudo-Anosov homeomorphisms.

\subsection{Results and organization of the article}

\subsubsection{Existence of geometric Markov partition}
After introducing our main definitions of \textbf{p-A} homeomorphisms, geometric types, and geometric Markov partitions, our first result is an algorithm to construct Markov partitions for pseudo-Anosov homeomorphisms that may contain spines (generalizing the arguments in \cite{fathi2021thurston}). For this purpose, we introduce the notions of adapted stable and unstable graphs and their compatible counterparts. To be adapted (\ref{Defi: Adapted graph}) means that the stable (or unstable) graph has edges corresponding to stable intervals exclusively contained in the stable separatrices of a singular point of $f$. To be compatible (\ref{Defi: Compatible graphs}) means that the endpoints of one graph are contained in the other and vice versa. Such criteria are usually easy to determine, leading to a decomposition of the surface into rectangles with the \emph{Markovian properties} between the intersections of the family and their images.

\begin{prop*}[\ref{Prop:Compatibles implies Markov partition}]
	If $\delta^u$ and $\delta^s$ are graphs adapted to $f$ and compatible, then the closure in $S$ of every connected component of
	$$
	\overset{o}{S} := S \setminus \left[ \cup \textbf{Ex}^s(\delta^u) \bigcup \cup \textbf{Ex}^u(\delta^s) \right]
	$$
	is a rectangle adapted to $f$ whose stable and unstable boundaries are contained in $\cup \delta^s$ and $\cup \delta^u$, respectively. Furthermore, the family $\mathcal{R}(\delta^s, \delta^u)$ given by such rectangles is an adapted Markov partition of $f$.
\end{prop*}

We must apply such criteria to construct a family of distinguished Markov partitions, and this is done in \ref{Cons: Recipe for Markov partitions}, where we start with points in the intersections of the stable and unstable separatrices of a pair of singularities and produce an \emph{adapted Markov} partition of $f$ (Corollary \ref{Coro: Existence adapted Markov partitions}). An adapted Markov partition (\ref{Defi: Partion wellsuited/corner/adapted partition}) is a Markov partition where the only periodic points in the boundary of the rectangles of the partition are singularities, and all such periodic points are in the corners of every rectangle of the partition they are contained in. This is a very special kind of Markov partition, but when we apply our construction to the family of \emph{first intersection points} (\ref{Defi: first intersection points}) of $f$, we discover a family of Markov partitions whose geometric types are, in some sense, canonical.

A geometric Markov partition for the homeomorphism $f$ obtained by the procedure indicated in \ref{Cons: Recipe for Markov partitions} is denoted by $\cR(z,n)$ and is called a \emph{primitive geometric Markov partition} of order $n$, where $z$ is the first intersection point of $f$ that was used in the construction \ref{Cons: Recipe for Markov partitions}, and $n \geq n(f)\in \NN_+$ is the parameter of the "complexity" of the partition (see \ref{Coro: n(f) number}). The parameter $n$ is referred to as the \emph{order} of $\cR(z,n)$. Let $\cM(f,n)$ be the set of geometric Markov partitions of $f$ of order $n$ (see \ref{Coro: constant type in orbit}). The main result in this direction is the following theorem, which states that the set of primitive geometric types of order $n$ is a finite.

\begin{theo}[\ref{Theo: finite geometric types}]
	For every $n > n(f)$, the set $\cT(f,n) := \{T(\cR) : \cR \in \cM(f,n)\}$ is finite. These geometric types are referred to as the \emph{primitive geometric types} of $f$ of order $n$.
\end{theo}

In \cite{IntiThesis}, it was shown that the geometric type is, in fact, a total invariant of conjugation. Therefore, the previous theorem provides us, for each $n$, with a canonical and finite invariant of conjugation. An ongoing project aims to determine new families of first intersection points that lead us to a canonical and distinguished primitive geometric Markov partition with the main property that any other primitive geometric type can be computed from this canonical geometric type. However, this is a topic to be explored in another paper.

\subsection{Geometric refinements.} Once the existence of a geometric Markov partition has been established, we move on to the main topic of this paper. Given a pair $(f,\cR)$—a homeomorphism and a geometric Markov partition—and the geometric type of the pair (see \ref{Defi: Geometric type of (f,cR)}), we use this information to construct another geometric Markov partition $\cR'$ of $f$ and compute the geometric type of $(f,\cR')$. Such Markov partitions are referred to as \emph{refinements} of $(f,\cR)$, and they are carefully introduced in Subsection \ref{Subsec: Generales refinamientos}. In this paper, we study three main refinements of a geometric type: the binary refinement, the boundary refinement along a family of periodic codes, and the corner refinement of a geometric Markov partition. We provide algorithms and formulas for each case, highlighting, case by case, the main properties that we try to induce in our refinements.

\subsubsection{The binary refinement}

The incidence matrix $A(f,\cR)$ of a Markov partition $\cR$ for the \textbf{p-A} homeomorphism $f:S \to S$ is one of the most common combinatorial objects associated with a Markov partition. It turns out that the geometric type of the pair $T(f,\cR)$ completely determines the incidence matrix. A matrix whose coefficients are exclusively $0$ and $1$ is called \emph{binary}, and it is a classic result that whenever $A(f,\cR)$ is binary, we can define a sub-shift of finite type $(\Sigma_{A(f,\cR)})$ associated with the matrix and a projection $\pi_{(f, \cR)}:\Sigma_{A(f,\cR)} \to S$ that semiconjugates $\sigma_{A(f,\cR)}$ with the \textbf{p-A} homeomorphism $f$ (see \ref{Prop:proyecion semiconjugacion}). We must use the sub-shift and the projection to establish in \cite{CruzDiaz2024} (as was done in \cite{IntiThesis}) that the geometric type is a total invariant of conjugation. However, to achieve this, we must first describe a method to: Given a geometric Markov partition $(f,\cR)$ and its geometric type, obtain a new geometric Markov partition of $f$ whose incidence matrix is binary. This is the content of the next theorem.

\begin{theo*}[\ref{Theo: Refinamiento binario}]
	Given a geometric partition $\mathcal{R}$ of $f$ and the geometric type $T$ of the pair $(f, \mathcal{R})$, there exists a finite algorithm to construct a geometric refinement of the pair $(f, \mathcal{R})$, known as the \emph{binary refinement} of $(f, \mathcal{R})$, denoted by $(f, \textbf{Bin}(\mathcal{R}))$ with the following properties:
	\begin{itemize}
		\item The  rectangles in $\textbf{Bin}(\cR)$ that are the horizontal sub-rectangles of the Markov partition $(f, \mathcal{R})$.
		\item The incidence matrix of $(f, \textbf{Bin}(\cR) )$ is binary.
		\item  there are explicit formulas in terms of the geometric type $T$ of $(f, \mathcal{R})$ to compute the geometric type  of  $(f, \textbf{Bin}(\mathcal{R}))$.
	\end{itemize} 
\end{theo*}

A geometric type $T$ is called symbolically modelable, denoted as $\cG\cT(\textbf{p-A})^{SIM}$, if there exists a pair $(f,\cR)$, a homeomorphism and geometric Markov partition, such that the geometric type of $(f,\cR)$ is $T$ and the incidence matrix of $T$ is binary.

\subsubsection{Refinements along a family of periodic codes.}

By virtue of the binary refinement, given a pair $(f,\cR)$, we can assume it is geometrically modelable  and can consider its associated finite-type sub-shift as introduced in  \ref{Defi:Induced shift}, as well as its projection $\pi_{(f, \cR)}:\Sigma_{A(f,\cR)} \to S$. A periodic code $\underline{w}\in\Sigma_{A(f,\cR)}$ projects onto a periodic point of $f$ (Lemma \ref{Lemm: Periodic to peridic}), and moreover, any iteration $t\in \NN$ of them determines a subset of codes called \emph{stable interval codes} (see \ref{Def: Stable intervals of codes})):

\begin{equation*}[\ref{Equa: Projection stable intervals}]
	\underline{I}_{t, \underline{w}} := \{\underline{v} \in \underline{F}^s(\sigma^t(\underline{w})): v_{n} = w_{t+n} \text{ for all } n \in \NN\},
\end{equation*}

These have the property of being projected under the action of $\pi_{(f, \cR)}$ onto a horizontal interval $I_{t,\underline{w}}$ of the rectangle $R_{w_t}\in \cR$ (Lemma\ref{Lemm: who the intervals intersect } ). Our goal now is to obtain a geometric Markov partition of $f$, $(f,\cR[\cS(\cW)])$, such that the only periodic points on the boundary of this partition are the orbits of $\{\pi_{(f, \cR)}(\underline{w}): \underline{w}\in \cW\}$ and subsequently compute its geometric type. This construction occupies most of Section \ref{Subsec: Codigos periodicos ref}, and the following theorem is the result of our entire discussion.

\begin{theo*}[\ref{Theo: Refinamiento s frontera fam}]
	Given a geometric Markov  partition $\mathcal{R} = \{R_i\}_{i=1}^n$ of $f$, the geometric type $T$ of the pair $(f, \mathcal{R})$, and a finite family of periodic codes $\cW = \{\underline{w}^1, \cdots, \underline{w}^Q\} \subset \Sigma_{A(T)}$, there exists a finite algorithm to construct a geometric refinement of $(f, \mathcal{R})$ by horizontal sub-rectangles of $ \mathcal{R}$, known as the $s$-\emph{boundary} refinement of $(f, \mathcal{R})$ with respect to $\cW$, denoted by $(f, \textbf{S}_{\cW}(\mathcal{R}))$, with the following properties:
	\begin{itemize}
		\item Every rectangle in the Markov partition $\textbf{S}_{\cW}(\mathcal{R})$ has as stable boundary components either a stable boundary arc of a rectangle $R_i\in \mathcal{R}$ or a stable arc $I_{t, \underline{w}}$ determined by the iteration $t$ of a code $\underline{w} \in \cW$. Moreover, every arc of the form $I_{t, \underline{w}}$ for $\underline{w} \in \cW$ is the stable boundary component of some rectangle in the refinement $\textbf{S}_{\cW}(\mathcal{R})$.
		\item The periodic boundary points of $\textbf{S}_{\cW}(\mathcal{R})$ are the union of the periodic boundary points of $\mathcal{R}$ and those in the set $\{\pi_{(f,\cR)}(\sigma_A^t(\underline{w})) \mid \underline{w} \in \cW \text{ and } t \in \NN\}$.
		\item There are explicit formulas, in terms of the geometric type $T$ and the codes in $\cW$, to compute the geometric type of $\textbf{S}_{\cW}(\mathcal{R})$.
	\end{itemize} 
\end{theo*}

The $u$-boundary refinement $\cR_{\cU(\cW)}$ of $(f,\cR)$ along a family of periodic codes is introduced in \ref{Defi: U boundary refinament}, and a symmetric statement of Theorem \ref{Theo: Refinamiento s frontera fam} holds for such refinement.  For the last refinement presented here, we must combine both procedures.

\subsubsection{The corner refinement}

A Markov partition has the corner property (\ref{Defi: Partion wellsuited/corner/adapted partition}) if every periodic point on its boundary is at the corner of any rectangle in the partition that contains it. This property is fundamental as it simplifies combinatorial arguments and combines them with topological arguments when carrying out our classification in \cite{CruzDiaz2024} and \cite{IntiThesis}. In Section \ref{Section: Corner ref}, we will prove the following theorem.

\begin{theo*}[\ref{Theo: The corner refinamiento}]
	Let $T \in \mathcal{T}(\textbf{p-A})^{SIM}$ and let $(f, \cR)$ be a pair that represents $T$. Then there exists a geometric Markov partition of $f$, denoted $\cR_{\cC}$, called the corner refinement of $(f, \cR)$ with the following properties:
	\begin{enumerate}
		\item The geometric Markov partition $\cR_{\cC}$ has the corner property.
		\item The periodic boundary codes of $(f, \cR_{\cC}$ are the boundary periodic points of $(f, \cR)$.
		\item There exists an algorithm and explicit formulas to compute the geometric type $T_{\cC}$ of $(f, \cR_{\cC}$.
	\end{enumerate}
\end{theo*}

The construction of the corner refinement is given as follows. First, we perform the $s$-boundary refinement along the periodic boundary codes of $(f,\cR)$. In this way, all periodic points of $(f,\cR)$ are now stable periodic points of $(f,\cR_{\cS(Per(f,\cR))}$. The idea is now to perform a $u$-boundary refinement of $(f,\cR_{\cS(Per(f,\cR))}$ along the boundary codes of this partition. However, there is a complication that we must resolve. We need to be able to compute the boundary periodic codes of the new partition $(f,\cR_{\cS(Per(f,\cR))}$. More generally, given a geometric type $T$, we must provide an algorithm to determine the set of its periodic boundary stable, unstable, and corner codes.  This is why we have dedicated Subsection \ref{Subsec: SU generating}
 to making these calculations explicit.

We make use of the so-called $S$ and $U$ generating functions given in \ref{Defi: s-boundary generating funtion} and \ref{Defi: u-boundary generating funtion}, and we use their iterations to obtain a list of stable and unstable boundary codes in $\Sigma_{A(f,\cR)}$ (\ref{Defi: positive negative boundary codes }). The main result of this subsection is the following corollary:

\begin{coro*}[\ref{Coro: algoritmic per codes}]
	Given a specific geometric type $T$, we can compute the sets: $\underline{(S(T)})$, $\underline{U(T)}$, $\underline{B(T)}$, and $\underline{C(T)}$, in terms of the geometric type $T$.
\end{coro*}

Finally, we will combine all these objects to obtain a corner refinement with respect to the family $\cW$ of periodic codes for $(f,\cR)$, and thus, we will obtain the following corollary with which we conclude our article.

\begin{coro}[\ref{Coro: Boundary refinament}]
	Let $\cR_{ \cC(\cW) }$  be the corner refinement of $(f,\cR_{S(\cW)})$ with geometric type $T_{\cC(\cW)})$. Then:
	
	\begin{itemize}
		\item[i)] The boundary periodic points of $(f, \cR_{\cC(\cW))})$ and the boundary periodic points of $(f, \cR_{S(\cW)})$ are exactly the same. So they are the union of the boundary periodic points of $(f,\cR)$ such that are the protection of the family $\cW$ by $\pi_{(f, \cR)}$.
		
		\item[ii)] The Markov partition $(f, \cR_{\cC(\cW))})$ exhibits the corner property.
		
		\item[iii)] There are formulas and an algorithm to compute the geometric type $T_{\cC(\cW)}$.
	\end{itemize}
\end{coro}

\section*{Acknowledgments}

The results presented in this article are part of my PhD thesis at the \emph{Université de Bourgogne}, completed under the supervision of Christian Bonatti, to whom I would like to express my deepest gratitude for the ideas and advice that guided me throughout my research. My work was supported by a \emph{CONACyT} scholarship in partnership with the French government, and I am thankful for their trust in me and my research. The writing of this preprint was completed at the Institute of Mathematics, UNAM, in Oaxaca, Mexico, with the support of the \emph{DGAPA-PAPIIT} grant \emph{IA100724} awarded by Lara Bossinger.

%% file: Preliminares/Preliminaresmain.tex
\section{Geometric Markov partition and abstract geometric types}

In this section, we shall introduce the main actors in this article: pseudo-Anosov homeomorphisms, their geometric Markov partitions, and the set of abstract geometric types. Once and for all, we assume that $S$ is a closed surface with a fixed orientation.

\input{Preliminares/MFoliationspAnosov}

\input{Preliminares/GeoPartitions}

\input{Preliminares/Geotypes}

%% file: Preliminares/MFoliationspAnosov.tex
\subsection{Generalized pseudo-Anosov homeomorphisms}\label{Subsec: Foliaciones singulares y homeomorfismos pA}
The following definition of singular foliation is given by Farb and Margalit in \cite[Subsection 11.2.2]{farb2011primer} where, however, we note that apart from $k \geq 3$-prong type singularities, we also have spine-type singularities, that is, $1$-pronged singularities. In \cite{fathi2021thurston} they are called \emph{generalized pseudo-Anosov} homeomorphism. 

\begin{defi}\label{Defi: Foliacion singular}
A \emph{singular foliation} $\mathcal{F}$ on the smooth surface $S$ is a decomposition of $S$ into a disjoint union of path-connected subsets of $S$, called the \emph{leaves} of $\mathcal{F}$, and a finite set of points \textbf{Sing}$(\mathcal{F}) \subset S$, called \emph{singular points} of $\mathcal{F}$, such that the following two conditions hold.

\begin{enumerate}
	\item For each non-singular point $p \in S$, there is a \emph{smooth foliated chart} from a neighborhood of $p$ into $\mathbb{R}^2$ that takes leaves to horizontal line segments. The transition maps between any two of these charts are \emph{smooth maps} from  of the form $(x, y) \rightarrow (f(x, y), g(y))$. In other words, the transition maps take horizontal lines to horizontal lines.
	\item For each singular point $p \in \textbf{Sing}(S)$, there is a smooth chart from a neighborhood of $p$ to $\mathbb{R}^2$ that takes leaves to the level sets of a $k$-pronged saddle, $k \geq 3$, or to a $1$-pronged singularity that we call \emph{Spine}-type singularity: as they are described in figure 	\ref{Fig: Cartas foliadas}.
\end{enumerate}
\end{defi}

\begin{figure}[ht]
	\centering
	\includegraphics[width=0.3\textwidth]{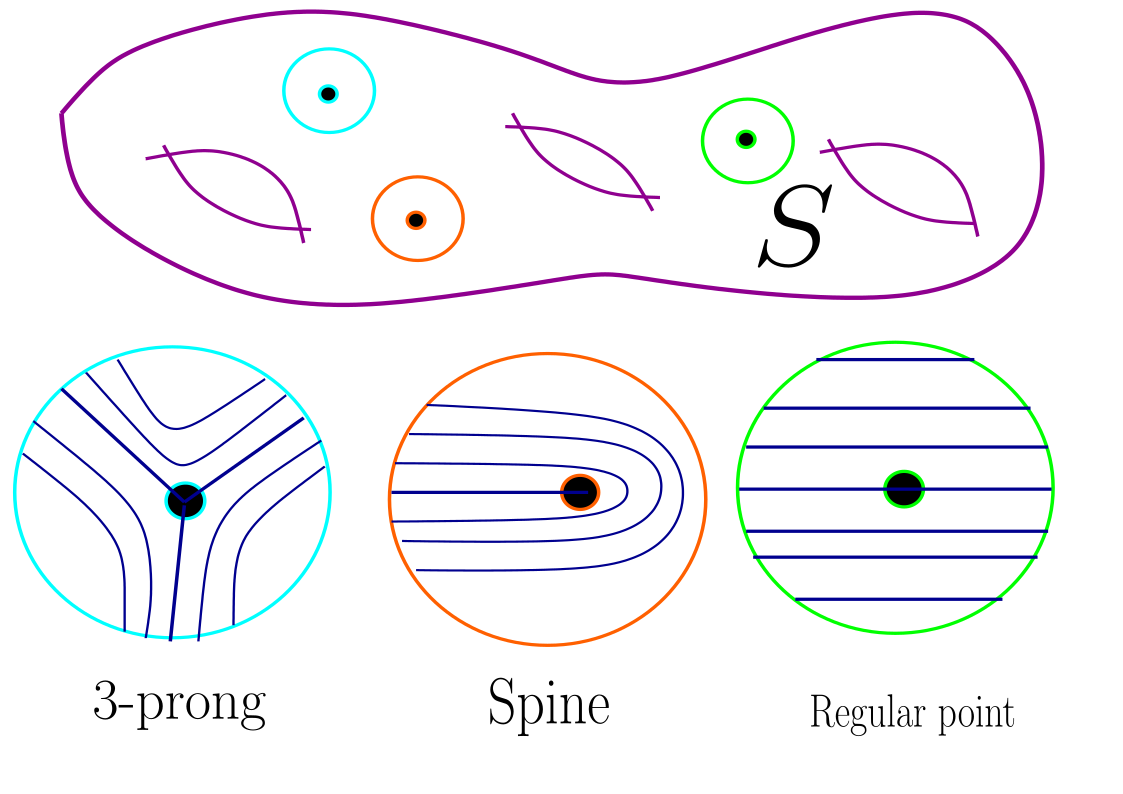}
	\caption{Different foliated charts}
	\label{Fig: Cartas foliadas}
\end{figure}

Let $g$ be the genus of the surface $S$ and assume $\mathcal{F}$ is a singular foliation on $S$. Let $s \in \textbf{Sing}(\mathcal{F})$ be a singular point with $P_s \geq 1$ prongs. The \emph{Euler-Poincaré Formula} \cite[Proposition 11.4]{farb2011primer}: $2\chi(S) = \sum_{p \in \textbf{Sing}(\mathcal{F})} (2 - P_s)$, imposes restrictions on the type of singularities $\mathcal{F}$ has. But since we permit $P_s$ to take the value $1$, we don't have any restriction on the genus of $S$, for example in the two dimensional sphere $\SS^2$ we have a singular foliation with four spines.

Let $\cF$ be a singular foliation of the surface $S$ and let $F$ be a leaf of $\cF$. A \emph{separatrix} of a point $x \in F$ is a connected component of $F \setminus \{x\}$. A regular point of $\cF$ has two different separatrices, while a $k$-pronged singularity has $k$ separatrices. In particular, a spine has only one separatrix.

A compact and connected subset of the surface $J \subset S$ is called an \emph{arc} if it is the image of an injective continuous map from the unit interval $I=[0,1]$ into the surface that is a homeomorphism onto its image, i.e., there exists a \emph{topological embedding} $\gamma: I \to S$ such that $\gamma(I) = J$. The \emph{interior} of the arc $J$ is given by $\overset{o}{J} = \gamma((0,1))$, and as a set, it doesn't depend on the specific parametrization of the arc.

\begin{defi}\label{Defi:Arcos transversales}
Let $\cF$ be a singular foliation of $S$. An arc $J \subset S$ is \emph{transversal} to the foliation $\cF$ if $\overset{o}{J} \cap \textbf{Sing}(\cF) = \emptyset$ and for every $x \in \overset{o}{J}$ there exists a foliated chart (\ref{Defi: Foliacion singular}) $\phi_x: U_x \to \RR^2$ with the following property: 

Let $\textbf{P}_i: \RR^2 \to \RR$ be the projection onto the real axis, $(x,y) \mapsto x$. Take $V_x \subset U_x \cap \overset{o}{J}$ to be the only connected component of $U_x \cap \overset{o}{J}$ that contains $x$, then $\textbf{P}_i \circ \phi_x: V_x \to \RR$ is injective. In other words, $\phi_x(\overset{o}{J})$ is \emph{topologically transversal} to the  leafs of $\cF$.
\end{defi}

Let $\cF$ be a singular foliation of $S$. A \emph{leaf-preserving isotopy} (Figure ) between two arcs $\alpha$ and $\beta$ that are transversal to $\cF$ is an isotopy $H: [0,1] \times [0,1] \to S$ such that $H([0,1] \times \{0\}) = \alpha$ and $H([0,1] \times \{1\}) = \beta$, and for all $t \in [0,1]$, the following conditions hold:
\begin{itemize}
	\item The arc $\gamma_t = H([0,1] \times \{t\})$ is transverse to $\cF$.
	\item Each one of the endpoints of $\gamma_t$ given by $H(\{0\} \times [0,1])$ and $H(\{1\} \times [0,1])$ is contained in a single leaf of $\cF$.
\end{itemize}

\begin{defi}\label{Defi: Medida transversal invariante}
Let $\cF$ be a singular foliation on $S$. A \emph{transverse invariant measure} of $\cF$, denoted by $\mu$, is a measure defined over any arc transversal to $\cF$, non-atomic and absolutely continuous with respect to the Lebesgue measure, that assigns a positive number to every arc that is not reduced to a point. Additionally, it is demanded that every pair of leaf-isotopic arcs have the same measure. That is, if $\alpha$ and $\beta$ are leaf-isotopic arcs, then $\mu(\alpha) = \mu(\beta)$.
\end{defi}

\begin{defi}\label{Defi: Foliacion medible}
	A \emph{measured foliation} on the surface $S$ is a pair $(\cF, \mu)$ consisting of a singular foliation $\cF$ of $S$ and a transverse invariant measure $\mu$ of $\cF$.
\end{defi}

Let $f: S \to S$ be a homeomorphism of $S$ and let $(\cF, \mu)$ be a measured foliation of $S$. The homeomorphism $f$ acts on the measured foliation to produce a new measured foliation on $S$, denoted by $f(\cF, \mu) = (f(\cF), f_*(\mu))$, where $f(\cF)$ is the singular foliation of $S$ given by the image of the leaves and singularities of $\cF$ under $f$, and for an arc $\gamma$ transverse to $f(\cF)$, the measure $f_*\mu$ is given by $f_*(\mu)(\gamma) = \mu(f^{-1}(\gamma))$. We must to note that,  $f_*(\mu)(\gamma)=\mu(f^{-1}(\gamma))$ only makes sense if $f^{-1}(\gamma)$ is transverse to  $\cF$, since the transversality has been defined at a topological level, this operation is well-defined. But isn't  true that $h(\cF)$ have smooth foliated charts, for the differential structure of $S$.  We are ready to formally introduce the maps that we must to study.

\begin{defi}\label{Defi: pseudo-Anosov homeomorphism}
	Let $S$ be a closed and oriented surface. An orientation-preserving homeomorphism $f: S \rightarrow S$ is a \emph{generalized pseudo-Anosov homeomorphism}, abbreviated as \textbf{p-A}-homeomorphism, if there exist a pair of measured foliations $(\cF^s, \mu^s)$ and $(\cF^u, \mu^u)$, called the \emph{stable} and \emph{unstable} foliations of $f$, that satisfy the following properties:
	\begin{itemize}
		\item A point $p \in S$ is a $k$-prong singularity of $\cF^s$ if and only if $p$ is a $k$-prong singularity of $\cF^u$. Therefore, the set of singularities of $f$ is defined as the set of singularities of the singular foliation, denoted by $\textbf{Sing}(f) := \textbf{Sing}(\cF^s) = \textbf{Sing}(\cF^u)$.
				
		\item The foliations are invariant under the action of $f$, which means that $f(\cF^s) = \cF^s$ and $f(\cF^u) = \cF^u$. We refer to $\cF^s$ and $\cF^u$ collectively as the \emph{invariant foliations} of $f$.
					
		\item Every leaf of the stable foliation $\cF^s$ is (at least) topologically transverse to every leaf of $\cF^u$ in the complement of their singularities.

		\item There exists a number $\lambda > 1$, called the \emph{stretch factor} of $f$, that:
		$$
		f_*(\mu^u) = \lambda \, \mu^u \quad \text{and} \quad f_*(\mu^s) = \lambda^{-1}\, \mu^s.
		$$
	\end{itemize}

\end{defi}

When the surface is the $2$-dimensional torus $\TT^2$ and the invariant foliations of the \textbf{p-A}-homeomorphism $f: \TT^2 \to \TT^2$ may have no singularities, since every pseudo-Anosov has an infinite number of periodic points, as proved in \cite[Proposition 9.20.]{fathi2021thurston}, we should adopt the following convention.

\begin{conv}\label{Conv: Singularidades caso del Toro}
	If the the invariant foliations of $f$ have no singularities, we agree that the set of singularities of $f$ consists of a finite number of periodic points of $f$.
\end{conv}

\begin{prop}\label{Prop: Propiedades foliaciones de p-A}
	Let $f: S \rightarrow S$ be a \textbf{p-A}  homeomorphism with invariant foliations $(\cF^s, \mu^s)$ and $(\cF^u, \mu^u)$. They have the following properties:
	\begin{enumerate}
		\item  None of the foliations contains a closed leaf, and none of the leaves contains two singularities (\cite[Lemma  14.11]{farb2011primer})	
		\item They are \emph{minimal}, i.e., any leaf of the foliations is dense in $S$  \cite[Corollary 14.15]{farb2011primer}).
	\end{enumerate}
\end{prop}

\begin{defi}\label{Defi: Top conjugacy y semiconjugacion}
	Let $S$ and $S'$ be two closed and orientable surfaces, and let $f : S \to S$ and $g : S' \to S'$ be two homeomorphisms. A map $h : S \to S'$ is a \emph{topological semi-conjugacy} from $f$ to $g$ provided $h$ is continuous, surjective, and satisfies $h \circ f = g \circ h$. We also say that $f$ is \emph{topologically semi-conjugated} to $g$ by $h$.
	
	If in addition $h$ is a homeomorphism, we say that $h$ is a \emph{topological conjugacy} between $f$ and $g$, and they are \emph{topologically conjugated} through $h$.
\end{defi}

\begin{prop}\label{Prop: conjugacion entre pseudo-Anosov}
	Let $f: S \rightarrow S$ and $g: S' \rightarrow S'$ be a pair of \textbf{p-A} homeomorphisms, and assume they are topologically conjugated through a homeomorphism $h: S \rightarrow S'$. For $\delta = u$ or $s$, let $(\cF^{\delta},\mu^{\delta})$ and $(\cG^{\delta},\nu^{\delta})$ denote the invariant foliations of $f$ and $g$ respectively. Finally, let $\delta_f$ and $\delta_g$ be the stretch factors of $f$ and $g$. In this manner:
	\begin{itemize}
		\item For $\delta = u$ or $s$, $h(\cF^{\delta}) = \cG^{\delta}$ and $\textbf{Sing}(g) = h(\textbf{Sing}(f))$.
		\item For $\delta = u$ or $s$, $h_*\mu^{\delta} = \nu^{\delta}$.
	\end{itemize}
\end{prop}

\begin{proof}[Sketch of the proof]
We only give the argument fo the stable foliations, the unstables case follow the same ideas.

Take a point $x \in S$ and let $L$ be a  leaf of $\cF^s$ passing through $x$. For any metric  $d(\cdot, \cdot)$ compatible with the topology of $S$, a point $y \in S$ is on the stable leaf of $x$  if and only if $d(f^n(x), f^n(y)) \to 0$. Since $h$ is a homeomorphism, for any metric $d'(\cdot, \cdot)$ compatible with the topology of $S$, the limit is preserved $d'(h \circ f^n(x), h \circ f^n(y)) \to 0$. Since $h$ is a conjugation between $f$ and $g$, the previous limit is equivalent to $d'(g^n \circ h(x), g^n \circ h(y)) \to 0$. Therefore, $h(x)$ is in the stable manifold of $h(y)$. We can conclude that $h(L_x) = G_{h(x)} \in \cG^s$. This probe our first item.

The second one is a following  computation, where let $J$ an arc transverse to $\cG^s$:
\begin{eqnarray*}
	h_*(\mu^s(g(J))) = \mu^s(h^{-1}(h \circ f \circ h^{-1})(J)) = \\
	\mu^s(f(h^{-1}(J))) = \lambda^{-1} \mu^s(h^{-1}(J)) = h_*\mu^s(J).
\end{eqnarray*}

This implies that $g$ has the same dilation factor as $f$.

\end{proof}

%% file: Preliminares/GeoPartitions.tex
\subsection{Geometric Markov partitions}\label{Subsec: Geometric Markov partitions}

Let $f: S \rightarrow S$ be a \textbf{p-A}-homeomorphism whose stable and unstable foliations are $(\cF^s, \mu^s)$ and $(\cF^u, \mu^u)$. A Markov partition for $f$ is a decomposition of the surface into pieces called \emph{rectangles}. Similar to a triangulation of the surface that carries topological information about $S$, a Markov partition carries dynamical information about $f$. So let's start by introducing these pieces.

\subsubsection{Parametrized rectangles.}

Let $D \subset S$ be a connected open set whose closure $\overline{D}$ is compact. For every $x \in D$, let $L_x^s$ be the stable leaf of $\cF$ passing through $x$. Define $\overset{o}{I_x}$ as the unique connected component of the intersection $L_x^s \cap D$ that contains $x$. Similarly, let $\overset{o}{J_x}$ be the unique connected component of the intersection $L_x^u \cap D$ (the unstable leaf passing through $x$ and the disc) that contains $x$.

\begin{defi}\label{Defi: Abierto bi-foleado}
Let $D \subset S$ be a connected and open set, we say $D$ is \emph{trivially bi-foliated} by $\cF^s$ and $\cF^u$ if for all $x \in D$, $\overset{o}{I_x}$ and $\overset{o}{J_x}$ are open arcs and for every $x, y \in D$, $\vert \overset{o}{I_x} \cap \overset{o}{J_y}\vert = 1$.
\end{defi}

\begin{defi}\label{Defi: Rectangulo}

A \emph{rectangle} for $f$ is the closure $\overline{D}$ of any connected and open set $D \subset S$, provided its closure is compact and $D$ is trivially bi-foliated.

\end{defi}

Henceforth, we should denote the unit square in $\RR^2$ with the standard orientation and its trivial foliations by vertical and horizontal lines as $\II^2 := [0,1] \times [0,1]$,  and its interior by $\overset{o}{\II^2} := (0,1) \times (0,1)$. If $\alpha$ is a compact arc, its boundary consists of two points which we call the \emph{endpoints} or extreme points of $\alpha$, and its interior is denoted by $\overset{o}{\alpha}$. With these definitions, we state the following proposition, which we must use to induce a vertical direction in our rectangles.

\begin{prop}\label{Prop: Rectangulos parametrizados}
Let $D$ be an open and connected set whose closure $R := \overline{D}$ is a rectangle for the \textbf{p-A} homeomorphism $r$. Let $R \subset S$ be a compact subset. Then there exists a continuous function $\rho : \II^2 \rightarrow S$ satisfying the following conditions:
\begin{enumerate}
	\item The image of $\rho$ is the rectangle $R$.
	
	\item The \emph{interior} of the rectangle $R$ is given by $\overset{o}{R} := \rho(\overset{o}{\II^2})$ map $\rho: \overset{o}{\II^2} \to \overset{o}{R}$ is an orientation-preserving homeomorphism.
			
	\item For every $t \in [0,1]$, $I_t := \rho([0,1] \times \{t\})$ is contained in a unique leaf of $\mathcal{F}^s$, and $\rho$ restricted to $[0,1] \times \{t\}$ is a homeomorphism onto its image. The arc $I_t$ is called \emph{horizontal leaf} of $R$.
	
	\item For every $s \in [0,1]$, $J_s := \rho(\{t\} \times [0,1])$ is contained in a unique leaf of $\mathcal{F}^u$, and $\rho$ restricted to $\{s\} \times [0,1]$ is a homeomorphism onto its image. The arc $J_s$ is called \emph{vertical leaf} of $R$.
\end{enumerate}
A map like $\rho$ is then called  \emph{parametrization} of $R$.
\end{prop}

\begin{proof}
First, we shall to construct a homeomorphism $\rho: (0,1) \times (0,1) \to S$ such that $\rho((0,1) \times (0,1)) = D$, and satisfies the last three items in the proposition. Subsequently, we must to extend $\rho$ continuously to $[0,1] \times [0,1] \to S$ in such a way that $\rho([0,1] \times [0,1]) = \overline{D} = R$, while satisfy all the other properties to be a parametrization of $R$.

Let $x \in D$ be a distinguished point, and let $y \in D$. Since $D$ is trivially bi-foliated by $\cF^s$ and $\cF^u$, there exists a unique point $y_1$ in the intersection $\overset{o}{I_x} \cup \overset{o}{J_y}$ and a unique $y_2$ in $\overset{o}{J_x} \cup \overset{o}{I_y}$. Following the notation in Fig. \ref{Fig: Arcos para la parametrizacion}  let $\epsilon(y) = 1$ if $\overset{o}{I_y}$ intersects the separatrix $J_1$ of $x$, and $-1$ otherwise. Similarly, $\delta(y) = 1$ if $\overset{o}{J_y}$ intersects the separatrix $I_1$ of $x$, and $-1$ otherwise. This allows us to define the map $\phi: D \to \RR^2$ which assigns to  $y\in D$ the point:
$$
(\epsilon(y)\cdot\mu^u([x,y_1]^s),\delta(y)\cdot \mu^s([x,y_2]^u)) \in \RR^2.
$$

\begin{figure}[ht]
	\centering
	\includegraphics[width=0.3\textwidth]{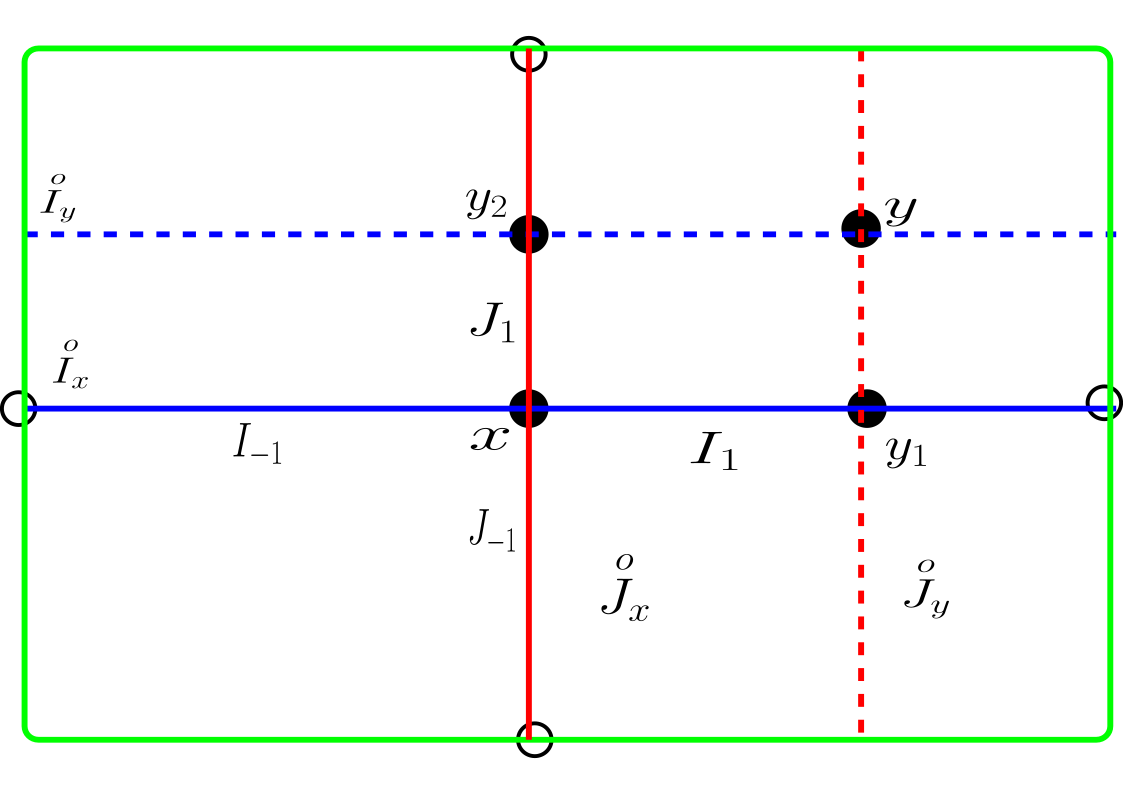}
	\caption{The arcs $[x,y_1]^s$ and $[x,y_1]^u$. }
	\label{Fig: Arcos para la parametrizacion}
\end{figure}

Since both measures are absolutely continuous with respect to Lebesgue and non-atomic, we can conclude that $\phi$ is a continuous homeomorphism, whose inverse is also continuous. Furthermore, since $D$ is trivially bi-foliated by $\cF^s$ and $\cF^u$ and the transverse measures $\mu^s$ and $\mu^u$ are invariant by isotopy on the leaves, the measures of any two arcs $I_x := \overline{\overset{o}{I_x}}$ and $I_y := \overline{\overset{o}{I_y}}$ have the same transverse measure, $\mu^u(I_x) = \mu^u(I_y)$. Similarly, the arcs $J_x := \overline{\overset{o}{J_x}}$ and $J_y := \overline{\overset{o}{J_y}}$ satisfy $\mu^s(J_x) = \mu^s(J_y)$. This means that the image of $D$ under $\phi$ is the interior of a rectangle $\overset{o}{H} := (A,B) \times (C,D)$ contained in $\RR^2$. Define $\rho := \phi^{-1}: \overset{o}{H} \to D$. This function clearly satisfies all the necessary items to be a parametrization of $R$. Now, we will extend it to $H := [A,B] \times [C,D] \to S$ in such a way that its image is the rectangle $R$ and satisfies the requirements to be a parametrization. Now we must to extend $\rho$ to $H$.

The closure of any interval $\overset{o}{J_x}\subset D$ is an unstable arc $J_x$ whose boundary consists of two different points $a(x)_1 \neq a(x)_{2}$, as $J_x$ cannot be a closed curve. If any of these points, say $a(x)_1$, were a singularity of $\cF^s$, since $D$ is trivially bi-foliated then no local stable separatrix of $a(x)_1$ could intersect $D$ (Fig. \ref{Fig: No stable separatrice enter D}). Similarly, this applies to the points on the boundary of a stable  $I_x=\overline{\overset{o}{I_x}}$, we must to preserve this notation.
 
 \begin{figure}[ht]
 	\centering
 	\includegraphics[width=0.3\textwidth]{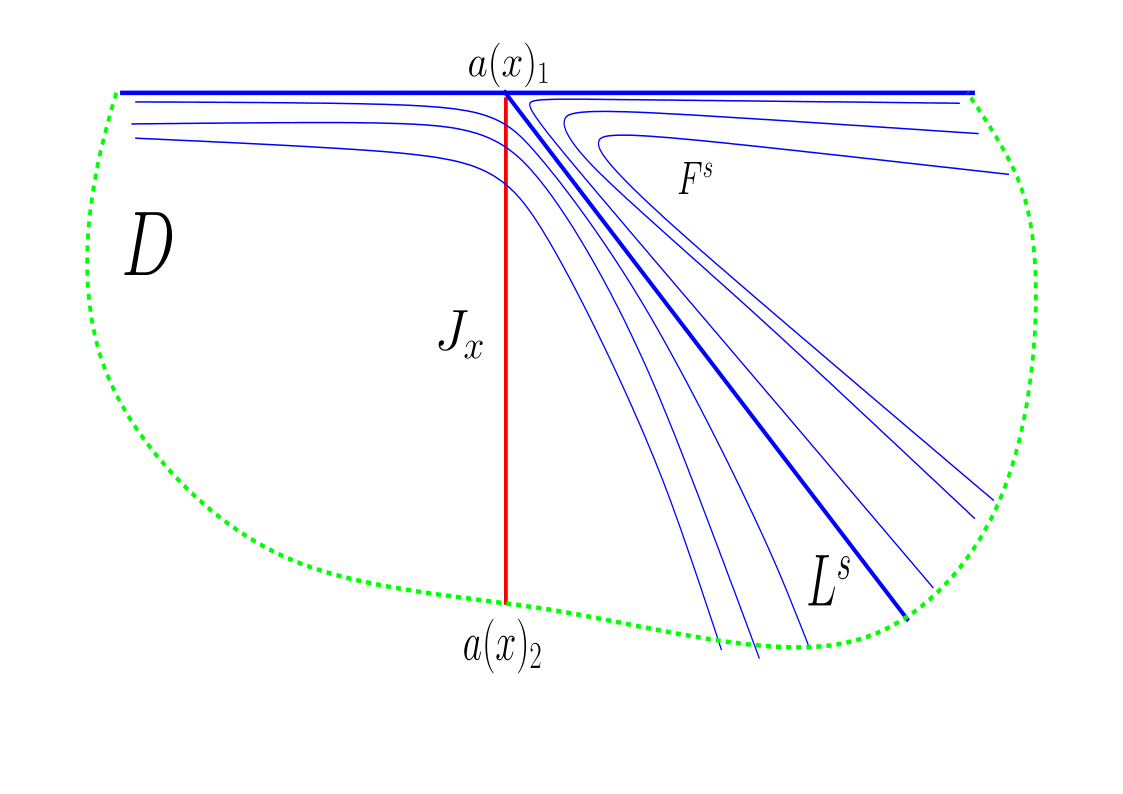}
 	\caption{ A stable separatrice of $a(x)_1$  entering $D$.}
 	\label{Fig: No stable separatrice enter D}
 \end{figure}

For every vertical interval $J \subset \overset{o}{H}$, $\rho(\overset{o}{J}) = \overset{o}{J_y}$ for certain $y \in D$ and there is a unique way to continuously extend $\rho$ to $J$ by coherently assigning the endpoints of $J$ to the extreme points of $J_x$. 
This can be done for every horizontal interval in $\overset{o}{H}$, thus extending $\rho$ to $H \setminus \{A, B\} \times [C, D]$.

We claim, sets $J_C:=\rho([A, B] \times \{C\})$ and $J_D:=\rho([A, B] \times \{D\})$ are each one, contained in a single leaf of $\cF^s$. In effect, since $D$ is trivially bi-foliated by $\cF^s$ and $\cF^u$, we can construct an isotopy along the stables leaves of $f$, between the unstable intervals $\overset{o}{J_x}$ and $\overset{o}{J_y}$. Given that the transverse measures of $f$ are non-atomic, $\mu^s(J_x) = \mu^s(J_y)$, and this isotopy extends to the endpoints of the interval, in particular, they are in the same leaves.

Let $z \in [A, B] \times \{C\} \subset H$. Open neighborhoods such as $H_1$ and $H_2$, around $z$, that are comprised between two vertical intervals $J_1$ and $J_2$ and the horizontal line $I_1$, as in the left-hand picture in Fig. \ref{Fig: Continuidad de rho}, are called \emph{rectangular neighborhoods} of $z$. Clearly, $\rho(H_1) = 1$ is as shown on the right side of Fig. \ref{Fig: Continuidad de rho}, where the opposite sides of the 'rectangle' $V_1$ have the same $\mu^s$ and $\mu^u$ measures, corresponding to leaf-isotopic arcs. Define its length $L(H_1)$ and width $W(H_1)$ as the measures of its unstable and stable sides, respectively, with respect to the transverse measures of $f$.

 \begin{figure}[ht]
	\centering
	\includegraphics[width=0.4\textwidth]{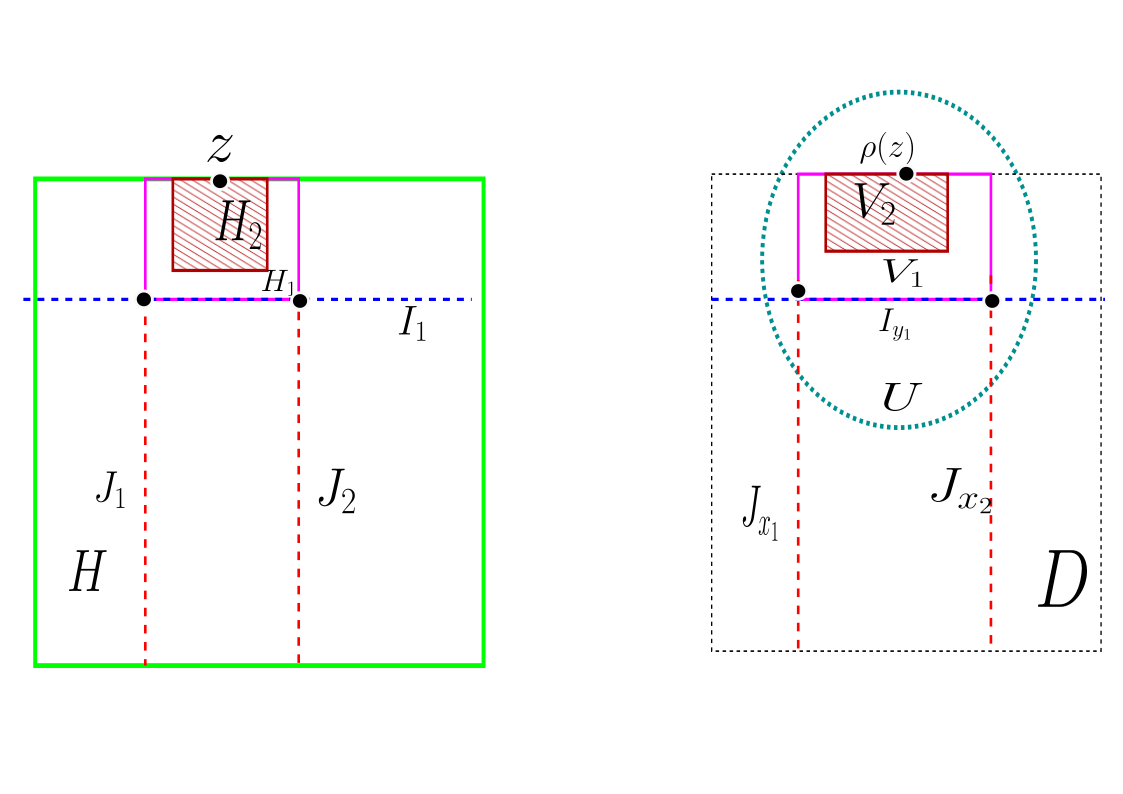}
	\caption{The image of a rectangular neighborhood of $z$}
	\label{Fig: Continuidad de rho}
\end{figure}

For any open set $U \subset S$ that contains $\rho(z)$, there is a neighborhood $V_1$, comprised between the arcs $J_{x_1}$ and $J_{x_2}$ and 'above' the stable arc $I_{y_1}$ as shown in the right-hand picture in Fig. \ref{Fig: Continuidad de rho}, with sufficiently small length and width such that it is contained in $U$. Note that $\rho^{-1}(V_1)$ is a rectangular neighborhood of $z$, and we can take a sufficiently small rectangular neighborhood $H_2 \subset \rho^{-1}(V_1)$ around $z$, as in Fig. \ref{Fig: Continuidad de rho}, whose image is contained in $U$. This implies the continuity of the extension of $\rho$.

The same procedure allows the continuous extension of $\rho$ to the endpoints of any stable arc $I_x$. Subsequently, we can extend $\rho$ to the corners of $H$. This map is a parametrization.

With respect to Item $(2)$, where it is required that $\rho: (0,1) \times (0,1) \to \overset{o}{R}$ be an orientation-preserving homeomorphism, we just need to compose $\rho$ with an affine transformation $A + c: \RR^s \to \RR^2$ that takes $H$ to the unit rectangle and such that the matrix $A$ preserves the orientation of $\RR^2$ if $\rho$ preserves the orientation in $(0,1) \times (0,1)$, or changes the orientation otherwise. This ends our proof.

\end{proof}

\subsubsection{Geometric rectangles}

A rectangle $R\subset S$ admits more than one parametrization, so we impose the following equivalence between them.

\begin{defi}\label{Defi: Parametrizaciones equiv}
	Let $\rho_1$ and $\rho_2$ be two parametrizations of the rectangle $R\subset S$. They are equivalent if the transition map between $\rho_1$ and $\rho_2$, restricted to the interior of the rectangle, $\overset{o}{R}$, is given by:
	\begin{equation}\label{Equa: cambio de coordenadas parametrizaciones}
	\rho_2^{-1} \circ \rho_1 := (\varphi_s, \varphi_u) : (0,1) \times (0,1) \rightarrow (0,1) \times (0,1),
	\end{equation}
	where $\varphi_s, \varphi_u: (0,1) \to (0,1)$ are increasing homeomorphisms.
\end{defi}

Since the parametrizations we are considering are orientation-preserving, it should be clear that there is only one equivalence class of parametrizations for a rectangle. Moreover, assume $\rho_1$ and $\rho_t$ are equivalent parametrizations, and the interiors of some horizontal leaves of $R$, $\overline{\overset{o}{I_1}} = \rho_1((0,1) \times \{t\})$ and $\overline{\overset{o}{I_2}} = \rho_1((0,1) \times \{s\})$, intersect. In that case, they must coincide, i.e., $I_1=I_2$. This property allows us to introduce the vertical and horizontal foliations of $R$ independently of its parametrization.

\begin{defi}\label{Defi: Foliacion vertical y horizontal de Rec}
Let $R$ be a rectangle, and let $\rho: \II^2 \to R \subset S$ be any parametrization of $R$. Then the \emph{horizontal} or\emph{ stable foliation} of $R$ is the family $\cF^s(R) = \{I_t\}_{t \in [0,1]}$ of all horizontal arcs as described in item $(3)$ of Proposition \ref{Prop: Rectangulos parametrizados}. Similarly, the \emph{vertical} or \emph{unstable foliation} of $R$ is the family $\cF^u(R) = \{J_s\}_{s \in [0,1]}$ of all vertical arcs induced by the parametrization $\rho$, as in item $(4)$ of Proposition \ref{Prop: Rectangulos parametrizados}.
\end{defi}

The vertical and horizontal lines in $\mathbb{R}^2$ admit two different pairs of \emph{orientations}, as presented in Fig. \ref{Fig: Orientaciones verticales en II2}, in such a manner that the product of these directions coincides with the standard orientation of $\mathbb{R}^2$. In other words, once we have chosen a direction in the vertical lines, the standard orientation of $\mathbb{R}^2$ determines a unique direction in the horizontal lines.

 \begin{figure}[h]
	\centering
	\includegraphics[width=0.3\textwidth]{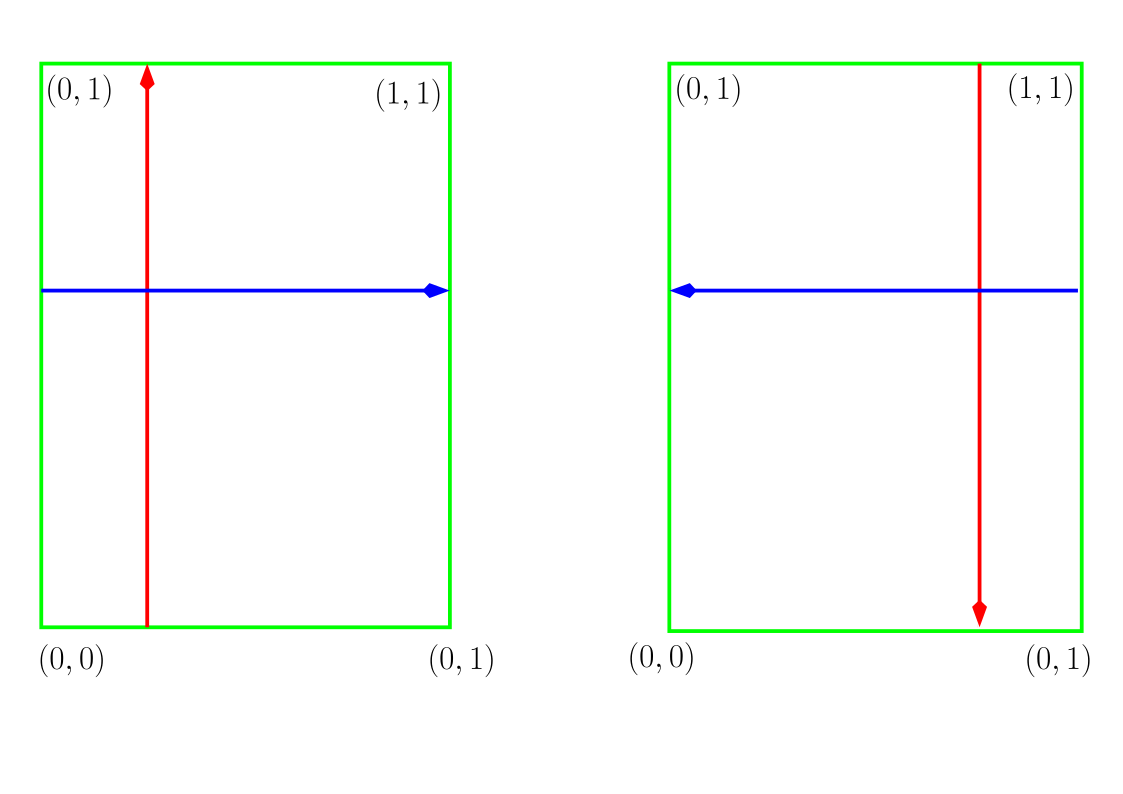}
	\caption{Same orientations but different vertical directions in $\II^2$.}
	\label{Fig: Orientaciones verticales en II2}
\end{figure}

We must to induce a vertical and horizontal direction a rectangle by using these two pair orientations in the unitary rectangle. 

\begin{defi}\label{Defi. Vertical direction}
Consider a parametrization $\rho: \II^2 \to S$ of the rectangle $R \subset S$. Choosing a vertical direction in $R$ involves selecting a direction for the vertical lines in $\II^2$, taking the corresponding direction for the horizontal lines in $\II^2$, as indicated in Fig. \ref{Fig: Orientaciones verticales en II2}, and finally assigning to the arcs in the vertical and horizontal foliations of $R$ the orientations induced by $\rho$.
\end{defi}
 A \emph{geometrized rectangle} is a rectangle for which we have fixed (or chosen) a vertical direction. 
\subsubsection{Geometric Markov partitions}

We must to introduce some important subset of a rectangle its horizontal and vertical sub-rectangles.
 
 \begin{defi}\label{Defi: subrec Vertical/horizontal }
Let $R \subset S$ be a rectangle. A rectangle $H \subset R$ is a \emph{horizontal sub-rectangle} of $R$ if, for all $x \in \overset{o}{H}$, the horizontal leaf of $\mathcal{I}(H)$ passing through $x$ and the horizontal leaf of $\mathcal{I}(R)$ passing through $x$ coincide. Similarly, a rectangle $V \subset R$ is a \emph{vertical sub-rectangle} of $R$ if, for all $x \in \overset{o}{V}$, the vertical leaf of $\mathcal{J}(V)$ passing through $x$ and the vertical leaf of $\mathcal{J}(R)$ passing through $x$ coincide.
 \end{defi}
 
 \begin{defi}\label{Defi: orientacion inducida sub-rec}
If $R$ is a geometrized rectangle and $H$ and $V$ are horizontal and vertical sub-rectangles of $R$, they inherit the same vertical and horizontal directions that $R$. This is the \emph{induced} or \emph{canonical geometrization} of the vertical sub-rectangles of $R$.
 \end{defi}

 \begin{defi}\label{Defi: Geo Markov partition}
 	Let $f: S \rightarrow S$ be a \textbf{p-A} homeomorphism. A \emph{Markov partition} for $f$ is a family of labeled rectangles $\mathcal{R} = \{ R_i \}_{i=1}^n$ with the following properties:
 	\begin{enumerate}
 		\item The surface $S$ is the union such rectangles: $S = \cup_{i=1}^n R_i$.
 		
 		\item They have disjoint interior, i.e. for all $i \neq j$,  $\overset{o}{R_i} \cap \overset{o}{R_j} = \emptyset$.
 		
 		\item For every $i, j \in \{1, \cdots, n\}$, the closure of each non-empty connected component of $\overset{o}{R_i} \cap f^{-1}(\overset{o}{R_j})$ is a horizontal sub-rectangle or $R_i$.
 		
 		\item For every $i, j \in \{1, \cdots, n\}$, the closure of each non-empty connected component of  $f(\overset{o}{R_i}) \cap \overset{o}{R_j}$ is a vertical sub-rectangle of $R_j$.
 	\end{enumerate}
If, in addition, every rectangle in the family $\cR$ is a geometrized rectangle, we say that $\cR$ is a geometric Markov partition of $f$. In this manner, the notation $(f, \cR)$ is reserved to indicate that $\cR$ is a geometric Markov partition of $f$.
 \end{defi}

  \begin{defi}\label{Defi: Horizontal sub-rec particion}
 	The family of horizontal and vertical sub-rectangles described in Items $iii)$ and $iv)$ of Definition \ref{Defi: Geo Markov partition}  are the \emph{horizontal} and \emph{vertical sub-rectangles} of the Markov partition $(f, \cR)$, respectively, and we denote them as $\cH(f, \cR)$ and $\cV(f, \cR)$, respectively.
 \end{defi}

 \subsubsection{Some distinguished subsets of a rectangle and a geometric Markov partition}

We must introduce some other interesting subsets inside a rectangle.

\begin{defi}\label{Defi: L-R sides}
	Let $R$ be a geometric rectangle for $f$, and let $\rho:\II^2\to R$ be a parametrization of $R$. We have the following distinguished sets:
	\begin{itemize}
		\item The \textbf{left} and the \textbf{right} sides of $R$ are given by $\partial^u_{-1}R := \rho(\{0\} \times [0,1])$ and $\partial^u_{1}R := \rho(\{1\} \times [0,1])$, respectively. Each of these arcs is an $s$-\emph{boundary component} of $R$.
		
		\item The \textbf{lower} and \textbf{upper} side of $R$ are defined as $\partial^s_{-1}R := \rho([0,1] \times \{0\})$ and $\partial^s_{+1}R := \rho([0,1] \times \{1\})$, respectively. Each of these arcs is a $u$-\emph{boundary component} of $R$.
		
		\item The \textbf{horizontal (or stable) boundary} of $R$ is defined as $\partial^s R := \partial^s_{-1}R \cup \partial^s_{+1}R$, and the \emph{vertical (or unstable) boundary} of $R$ as $\partial^u R := \partial^u_{-1}R \cup \partial^u_{+1}R$.
		
		\item The \textbf{boundary} of $R$ is $\partial R := \partial^s R \cup \partial^u R$.
		
		\item The \textbf{corners} of $R$ are points defined by the image of $\rho$ of points $(s,t) \in \II^2$, where the respective corner is labeled $C_{s,t} = \rho(s,t)$.
		
		\item For all $x \in \overset{o}{R}$, $I_x$ denotes the horizontal leaf of $\mathcal{I}(R)$ that passes through $x$, and $J_x$ denotes the vertical leaf of $\mathcal{J}(R)$ that passes through $x$.
		
	\end{itemize}
\end{defi}

In the same spirit, we must assign names to some distinguished points and subsets within a geometric Markov partition.

\begin{defi}\label{Defi: Boundary points Markov partition}
	Let $\mathcal{R} = \{R_i\}_{i=1}^n$ be a geometric Markov partition of $f$. The stable boundary of $(f, \mathcal{R})$ is given by $\partial^s \mathcal{R} = \cup_{i=1}^n \partial^s R_i$, the unstable boundary by $\partial^u \mathcal{R} = \cup_{i=1}^n \partial^u R_i$, and the boundary by $\partial \mathcal{R} = \cup_{i=1}^n \partial R_i$. Finally, the interior of the Markov partition $(f, \mathcal{R})$ is the set $\overset{o}{\mathcal{R}} = \cup_{i=1}^n \overset{o}{R_i}$. 
	
	Now, let $p$ be a periodic point of $f$. In this case:
	\begin{enumerate}
		\item We say that $p$ is an $s$-\emph{boundary} periodic point of $(f, \mathcal{R})$ if $p \in \partial^s \mathcal{R}$, a $u$-\emph{boundary} periodic point if $p \in \partial^u \mathcal{R}$, and in general, that $p$ is a \emph{boundary periodic point} if $p \in \partial \mathcal{R}$. These sets of $s$, $u$, and $b$ boundary periodic points are denoted as $\text{Per}^{s}(f, \mathcal{R})$, $\text{Per}^{u}(f, \mathcal{R})$, and $\text{Per}^{b}(f, \mathcal{R})$, respectively.
		
		\item We say that $p$ is an \emph{interior} periodic point if $p \in \overset{o}{\mathcal{R}}$. The set of interior periodic points is denoted as $\text{Per}^I(f, \mathcal{R})$.
		
		\item We say that $p$ is a \emph{corner} periodic point if there exists $i \in \{1, \ldots, n\}$ such that $p$ is a corner point of the rectangle $R_i$. The set of corner periodic points is denoted as $\text{Per}^C(f, \mathcal{R})$.
	\end{enumerate}
\end{defi}

The following lemma will be used many times in our future explanations.

\begin{lemm}\label{Lemm: Boundary of Markov partition is periodic}
The upper and lower boundaries of each rectangle in a Markov partition $\cR$ of $f$ lay on the stable leaf of some periodic point of $f$. Similarly, the left and right boundaries are contained in the unstable leaf of some periodic point.
\end{lemm}

\begin{proof}
	Let $x$ be a point on the stable boundary of $R_i$. For all $n\geq 0$, $f^{n}(x)$ remains on the stable boundary of some rectangle because the image of the stable boundary of $R_i$ coincides with the stable boundary of some vertical rectangle in the Markov partition, and such vertical rectangle have stable boundary in $\partial^s \cR$. However, there are only a finite number of stable boundaries in $\cR$. Hence, there exist $n_1, n_2\in \mathbb{N}$ such that $f^{n_1}(x)$ and $f^{n_2}(x)$ are on the same stable boundary component. This implies that this leaf is periodic and corresponds to the stable leaf of some periodic point. Therefore, $x$ is one of these leaves and is located on the stable leaf of a periodic point.
	A similar reasoning applies to the case of vertical boundaries.
\end{proof}

\subsubsection{The sectors of a point}  The following result corresponds to \cite[Lemme 8.1.4]{bonatti1998diffeomorphismes} and states that around any point on the surface, there exists a local stratification of $S$ given by the transverse foliations of a \textbf{p-A} homeomorphism $f$, which we will exploit.

\begin{theo}\label{Theo: Regular neighborhood}
Let $f: S \rightarrow S$ be a generalized pseudo-Anosov homeomorphism, and let $p \in S$ be a singularity with $k \geq 1$ separatrices. Then, there exists $\epsilon_0 > 0$ such that for all $0 < \epsilon < \epsilon_0$, there is a neighborhood $D(p,\epsilon)$ of $p$ with the following properties:
\begin{itemize}
	\item The boundary of $D(p,\epsilon)$ consists of $k$ segments of unstable leaves alternating with $k$ segments of stable leaves.
	\item For every (stable or unstable) separatrix $\delta$ of $p$, the connected component of $\delta \cap D(p,\epsilon)$ which contains $p$ has $\mu^u$ and $\mu^s$ measure (as corresponding) equal to $\epsilon$.
\end{itemize}
We say that $D(p,\epsilon)$ is a \emph{regular neighborhood} of $p$ with length equal to $\epsilon$.
\end{theo}

\begin{figure}[ht]
	\centering
	\includegraphics[width=0.4\textwidth]{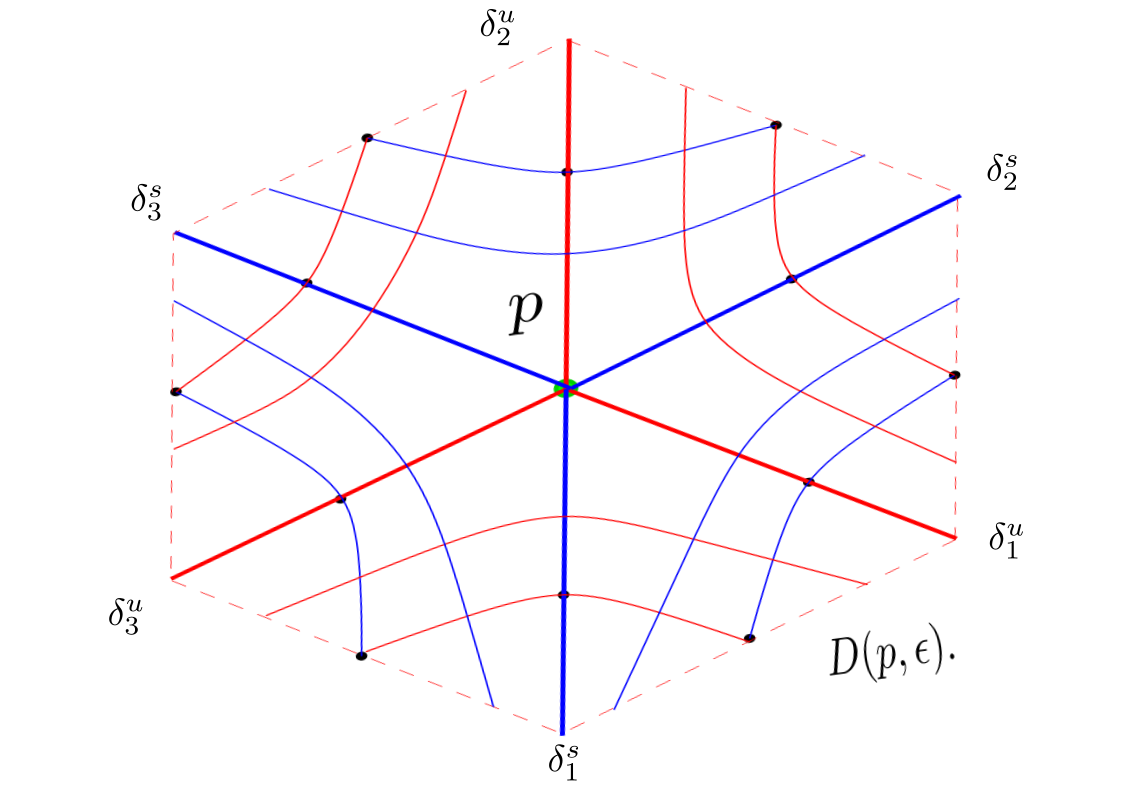}
	\caption{Regular neighborhood of a $3$-prong}
	\label{Fig: Regular neighborhood}
\end{figure}

Let $D(p,\epsilon)$ be a regular neighborhood of $p \in S$ with length $0 < \epsilon \leq \epsilon_0$ (see Fig. \ref{Fig: Regular neighborhood}), and let us assume that $p$ has $k \in \NN_+$ separatrices (so it could be a regular point). By using the orientation of $S$, we label the stable and unstable separatrices of $p$ cyclically counterclockwise as $\{\delta_i^s\}_{i=1}^k$ and $\{\delta_i^u\}_{i=1}^k$. Let $\delta_i^s(\epsilon)$ and $\delta_i^u(\epsilon)$ be the unique connected components of $\delta_i^s \cap D(p,\epsilon)$ and $\delta_i^u \cap D(p,\epsilon)$ that has $p$ as and end point. By giving a proper labeling to the local separatrices of $p$, we may assume that $\delta_i^u(\epsilon)$ is located between $\delta_i^s(\epsilon)$ and $\delta_{i+1}^s(\epsilon)$, with the index $i$ taken modulo $k$ (as in Fig. \ref{Fig: Regular neighborhood}).

 Finally, the connected components of $\text{Int}(D(p,\epsilon) \setminus (\cup_{i=1}^k \delta_{i=1}^s(\epsilon) \cup \delta_{i=1}^u(\epsilon)))$ are labeled in a counterclockwise cyclic order and denoted as $\{E(p,\epsilon)_j\}_{j=1}^{2k}$, where the boundary of $E(p,\epsilon)_1$ consists of $\delta^s_1(\epsilon)$ and $\delta^u_1(\epsilon)$. These conventions lead to the following definition.

\begin{defi}\label{Defi:converge in a sector }
	Let $\{x_n\}$ be a sequence convergent to $p$. We say $\{x_n\}$ \emph{converges to $p$ in the sector} $j$ if there exists $N \in \mathbb{N}$ such that for every $n > N$, $x_n \in E(p,\epsilon_0)$. The set of sequences that convergent to $p$ in the sector $j$ is denoted by $E(p)_j$.
\end{defi}

In this manner set of all the  sequences that converge to $p$ in a sector is then given by  $\cup_{j=1}^{2k} E(p)_j$. We are going to define an equivalence relation on these sets.

\begin{defi}\label{Defi: Sector equiv}
Let $\{x_n\}$ and $\{y_n\}$ be sequences that converge to $p$ in a sector. We say they are in the same sector of $p$, denoted by the relation $\{x_n\} \sim_q \{y_n\}$, if and only if there exist   $j \in \{1, \ldots, 2k\}$ and $N \in \mathbb{N}$ such that for all $n \geq N$, $x_n, y_n \in E(p,\epsilon_0)_j$.
\end{defi}

\begin{lemm}\label{Lemm: Sector Equiv relation}
In the set of sequences that converge to $p$ in a sector, the relation $\sim_q$ is an equivalence relation. Moreover, every class for this relation coincides with the set $E(p)_j$ for a certain $j \in \{1, \ldots, 2k\}$.
\end{lemm}

\begin{proof}
	The least evident property is transitivity. If $\{x_n\} \sim_q \{y_n\}$ and $\{y_n\} \sim_q \{z_n\}$, suppose $x_n, y_n \in E(p,\epsilon_0)_j$ for $n > N_1$ and $y_n, z_n \in E(p,\epsilon_0)_{j'}$ for $n > N_2$. For every $n > N := \max(N_1, N_2)$, $y_n \in E(p,\epsilon_0)_j \cap E(p,\epsilon_0)_{j'}$. This is only possible if and only if $j = j'$.
\end{proof}

Now we can introduce the formal definition of a sector of a point, and then we can establish in Lemma \ref{Prop: image secto is a sector} that the action of the \textbf{p-A} homeomorphism on the sectors of a point is well-defined.

\begin{defi}\label{Defi: Sector}
The equivalence class $E(p)_j$ of those sequences that converge to $p$ in sector $j$ is what we call the \emph{sector} $j$ of the point $p$ an is denoted by $e(x)_j$.
\end{defi}

\begin{prop}\label{Prop: image secto is a sector}
Let $f: S \to S$ be a generalized \textbf{p-A} homeomorphism, and let $p \in S$ be any point on the surface with $k$ separatrices. If $\{x_n\} \in E(p)_j$, there is a unique $i \in \{1, \dots, 2k\}$ such that $\{f(x_n)\} \in E(f(p))_i$. We summarize this by saying that the image of the sector $E(p)_j$ is the sector $E(f(p))_i$.
\end{prop}

\begin{proof}
 Since $f$ is continuous, there exists $0 < \epsilon < \epsilon_0$ such that $f(D(p,\epsilon)) \subset D(f(p),\epsilon_0)$. Let $\delta^s(p)_j$ and $\delta^u(p)_j$ be the separatrices of $p$ that bound the set $E(p,\epsilon)_j \subset S$. Then $f(\delta^s(p)_j)$ and $f(\delta^u(p)_j)$ are contained in a pair of adjacent separatrices of $f(p)$, which determine a unique set $E(f(p),\epsilon_0)_i$ for some $i \in \{1, \ldots, 2k\}$ (remember that $p$ and $f(p)$ have the same number of sectors).Finally, let $N \in \NN$ be such that, for all $n > N$, $x_n \in E(p,\epsilon)_j$. This implies that, for every $n > N$, $f(x_n) \in E(f(p),\epsilon_0)_i$, and then  the sequence $\{f(x_n)\}$ is in the sector $e(f(p))_i$ probing our statement.
 
\end{proof}

We have the following lemma, which we will need to use in the future when we study the sub-shift of finite type associated with a Markov partition.

\begin{lemm}\label{Lemm: sector contined unique rectangle}
	Let $f: S \rightarrow S$ be a \textbf{p-A} homeomorphism with a  Markov partition $\cR$. Let $x \in S$ and let $e(x)_j$ be a sector of $x$. Then, there exists a unique rectangle in the Markov partition that contains the sector $e(x)_j$.
\end{lemm}

\begin{proof}
	Let $\{x_n\}$ be a sequence that converges to $x$ within the sector $e(x)_j$. Consider a canonical neighborhood $U$ of size $\epsilon > 0$ around $x$, and let $E$ be the unique connected component of $U$ minus the local stable and unstable manifolds of $x$ that contains the sequence $\{x_n\}$.
	
	By choosing $\epsilon$ small enough, we can assume that the local stable separatrix $I$ of $x$ that bound $E$ is contained in at most two rectangles of the Markov partition, and similarly, the local unstable separatrix $J$ of $x$ is contained in at most two rectangles. By take the right side of the local separatrices, we can take:  rectangles $R$ and $R'$ in the Markov partition, a horizontal sub-rectangle $H$ of $R$ that contains $I$ in its upper or lower boundary, and a vertical sub-rectangle $V$ of $R'$ whose upper or left boundary contains $J$. These rectangles can be chosen small enough such that the intersection of their interiors is a rectangle contained within $E$, denoted as $\overset{o}{Q} := \overset{o}{H} \cap \overset{o}{V} \subset E$. This implies that $R = R'$ since the intersection of the interiors of rectangles in the Markov partition is empty.
	
	Furthermore, by considering a sub-sequence of $\{x_n\}$, we don't change its equivalent sector class. Therefore, $\{x_n\} \subset \overset{o}{Q} \subset R$. This completes our proof.
\end{proof}

%% file: Preliminares/Geotypes.tex
\subsection{Geometric types} 

We are about to introduce the set of \emph{abstract geometric types}. These are finite combinatorial objects, as they are defined by a pair of maps between finite sets. Immediately afterward, we will describe how to associate a geometric type with a geometric Markov partition. This last procedure should justify the names given to the parameters of the geometric type in the following definition.

\begin{defi}\label{Defi: Abstract geometric type}
An abstract geometric type is an ordered quintuple:
\begin{equation}
	T := (n, \{h_i, v_i\}_{i=1}^n, \rho_T, \epsilon_T)
\end{equation}
which contains the following information:
\begin{itemize}
	\item A positive integer $n \in \mathbb{N}_+$ that is the \emph{number of rectangles} in $T$.
	\item For all $i \in \{1, \cdots, n\}$, $h_i, v_i \in \mathbb{N}_+$ are non-negative integers, where $h_i$ is the number of \emph{horizontal} sub-rectangles and $v_i$ is the number of \emph{vertical} sub-rectangles for rectangle $i$ in $T$.
	\item These numbers satisfy $\sum_{i=1}^{n} h_i = \sum_{i=1}^n v_i = \alpha(T)$, and $\alpha(T)$ is referred to as the number of \emph{sub-rectangles} in $T$.
	
	\item The map $\rho_T: \mathcal{H}(T) \rightarrow \mathcal{V}(T)$ is a bijection between the set:
\begin{equation}
	\mathcal{H}(T) = \{(i,j) \mid 1 \leq i \leq n \text{ and } 1 \leq j \leq h_i\}
\end{equation}
	of \emph{horizontal labels} of $T$ and the set of \emph{vertical labels} of $T$
\begin{equation}
	\mathcal{V}(T) = \{(k,l) \mid 1 \leq k \leq n \text{ and } 1 \leq l \leq v_k\}.
\end{equation}
	\item Finally, $\epsilon_T: \mathcal{H}(T) \rightarrow \{-1, 1\}$ is any function from the set of horizontal labels of $T$ to $\{-1, 1\}$.
\end{itemize}
The set of abstract geometric types is denoted by $\mathcal{G}\mathcal{T}$.
\end{defi}

\subsubsection{The pseudo-Anosov class}
Let $f: S \to S$ be a \textbf{p-A} homeomorphism, and let $\mathcal{R}$ be a geometric Markov partition of $f$. We are going to associate a \emph{geometric type} with the pair $(f, \mathcal{R})$.  Let's start by labeling the horizontal sub-rectangles of the Markov partition $(f, \mathcal{R})$ (Definition \ref{Defi: Horizontal sub-rec particion}) that are contained in the rectangle $R_i \in \mathcal{R}$. This must be done in increasing way (from the bottom to the top) with respect to the vertical direction of $R_i$ as $\{H^i_j\}_{j=1}^{h_i}$, where $h_i \geq 1$. Similarly, the vertical sub-rectangles of $(f, \mathcal{R})$ contained in $R_k$ are labeled in increasing order (from left to right) with respect to the horizontal direction of $R_k$ as $\{V_l^k\}_{l=1}^{v_k}$, where $v_k \geq 1$.

Such indexations determine the set of \emph{horizontal labels} of $(f, \mathcal{R})$,
\begin{equation}\label{Equa: H(f,cR)}
	\mathcal{H}(f,\mathcal{R}) = \{(i,j) \mid i \in \{1, \cdots, n\} \text{ and } j \in \{1, \cdots, h_i\}\}.
\end{equation}
and the set of \emph{vertical labels} of $(f, \mathcal{R})$,
\begin{equation}\label{Equa: V(f,cR)}
	\mathcal{V}(f,\mathcal{R}) = \{(k,l) \mid k \in \{1, \cdots, n\} \text{ and } l \in \{1, \cdots, v_k\}\}.
\end{equation}

Clearly, the homeomorphism $f$ induces a bijection between the set of horizontal and vertical rectangles of $(f, \mathcal{R})$, and thus a bijection between the set of vertical and horizontal labels of $(f, \mathcal{R})$. Consequently, $\sum_{i=1}^{n} h_i = \sum_{i=1}^n v_i = \alpha(f, \mathcal{R})$, and we can define a map between the set of vertical and horizontal labels of the pair $(f, \mathcal{R})$, as indicated below.

\begin{defi}\label{Defi: Permutation (f,cR)}
	Let  $\rho:\cH(f,\cR) \rightarrow \cV(f,\cR)$ be the map  defined as $\rho(i,j)=(k,l)$ if and only if $f(H^i_j)=V^k_l$.
\end{defi}

The final step is to capture, in another map, the change in the vertical direction induced by $f$ when restricted to one horizontal sub-rectangle of the partition, as defined below.

\begin{defi}\label{Defi: Orientation (f,cR)}
Suppose $f(H_j^i) = V_l^k$. Let $\epsilon: \mathcal{H}(f, \mathcal{R}) \rightarrow \{1, -1\}$ be the map defined by $\epsilon(i,j) = 1$ if $f$ maps the vertical direction of $H_j^i$ to the vertical direction of $V_l^k$, and $-1$ otherwise.
\end{defi}

By integrating all this information, we are able to define the geometric type of a geometric Markov partition.

\begin{defi}\label{Defi: Geometric type of (f,cR)}
Let $f: S \to S$ be a \textbf{p-A} homeomorphism, and let $\mathcal{R}$ be a geometric Markov partition of $f$. The \emph{geometric type} of the pair $(f, \mathcal{R})$ is denoted by $T(f, \mathcal{R})$ and is given by:
\begin{equation}
	T(f, \mathcal{R}) = (n, \{(h_i, v_i)\}_{i=1}^n, \rho, \epsilon).
\end{equation}
where 

\begin{itemize}
	\item $n$ is the number of rectangles in the Markov partition $\mathcal{R}$.
	\item $h_i$ and $v_i$ are the numbers of horizontal and vertical sub-rectangles of the pair $(f, \mathcal{R})$ contained in the rectangle $R_i \in \mathcal{R}$.
	\item The map $\rho$ is as defined in \ref{Defi: Permutation (f,cR)}.
	\item The map $\epsilon$ is as defined in \ref{Defi: Orientation (f,cR)}.
\end{itemize}
\end{defi}

Not every abstract geometric type is the geometric type of the Markov partition of a \textbf{p-A} homeomorphism, but those that are determine a special subset of $\mathcal{G}\mathcal{T}$, and we need to give them a name.

\begin{defi}\label{Defi: pseudo-Anosov class}
A geometric type $T \in \mathcal{G}\mathcal{T}$ is in the \emph{pseudo-Anosov class} of geometric types, $\mathcal{G}\mathcal{T}(\textbf{p-A})$, if there exists a closed and orientable surface $S$ that supports a \textbf{p-A} homeomorphism $f: S \rightarrow S$ with a geometric Markov partition $\mathcal{R}$, such that the geometric type of $(f, \mathcal{R})$ is $T$, i.e., $T = T(f, \mathcal{R})$.

We often say that $T$ is \emph{realized} by $(f, \mathcal{R})$ or that $(f, \mathcal{R})$ is a \emph{realization} of $T$.
\end{defi}

%% file: Existencia/Existencia.tex
\section{An Algorithmic Construction of Markov Partitions}\label{Sec: Contruccion Particion Markov}

It is a well-known result in the theory of surface dynamics that any pseudo-Anosov homeomorphism $f$ has a Markov partition (see \cite[Proposition 10.17]{fathi2021thurston}). However, in Proposition \ref{Prop:Compatibles implies Markov partition} in this section, we provide criteria for families of stable and unstable intervals in the invariant foliations of a \textbf{p-A} homeomorphism $f$ to be the boundary of a Markov partition of $f$. After this general construction, in \ref{Sub-sec: Primitive MP and first intersection}, we apply Proposition \ref{Prop:Compatibles implies Markov partition} to construct a Markov partition for each \emph{first intersection point} of $f$ (\ref{Defi: first intersection points}). These Markov partitions have favorable properties regarding their iterations and geometric types. Perhaps the main interest in our results lies in the versatility of the criterion and the various points to which it has been applied. A more careful selection of such points should yield a canonical representation of the conjugacy class of $f$.

Different kinds of boundary periodic points were introduced in \ref{Defi: Boundary points Markov partition}. It is important to note that a singularity of $f$ is always a boundary periodic point of $(f, \mathcal{R})$, but there could also exist non-singular periodic boundary points. Furthermore, such periodic boundary points may be located on the boundary of two rectangles: one where it is a corner and another where it is not. We formalize this behavior in the following definition.

\begin{defi}\label{Defi: Partion wellsuited/corner/adapted partition}
Let $\mathcal{R}$ be a Markov partition of the \textbf{p-A} homeomorphism $f$. Then the pair $(f, \mathcal{R})$ 
\begin{itemize}
	\item is \emph{well-suited} to $f$ if $\textbf{Per}^b(f, \mathcal{R}) = \textbf{Sing}(f)$, meaning that its only periodic boundary points are singularities.
	\item has the \emph{corner property} if every rectangle $R \in \mathcal{R}$ containing a periodic boundary point $p$ has $p$ as one of its corner points.
	\item is \emph{adapted} to $f$ if it is both well-suited and has the corner property. That is, its only periodic boundary points are singularities, and each of them is a corner of any rectangle that contains it.
\end{itemize}
\end{defi}

The Markov partitions that we will obtain at the end of this section will be adapted to our homeomorphism. We will now begin with this construction.

\subsection{Graphs adapted to the \text{p-A} Homeomorphism}\label{Subsec: Adapted graps to MP} 

While the definition \ref{Defi: Geo Markov partition} of a Markov partition is useful for visualizing and determining its geometric type using its sub-rectangles, the following characterization is more practical for constructing a Markov partition.

\begin{prop}\label{Prop: Markov criterion boundary}
	Let $f: S \rightarrow S$ be a \textbf{p-A} homeomorphism, and let $\mathcal{R} = \{R_i\}_{i=1}^n$ be a family of rectangles whose union is $S$ and whose interiors are disjoint. Then, $\mathcal{R}$ is a Markov partition for $f$ if and only if its stable boundary $\partial^s\mathcal{R} := \cup_{i=1}^n \partial^s R_i$ is $f$-invariant, and its unstable boundary $\partial^u\mathcal{R} := \cup_{i=1}^n \partial^u R_i$ is $f^{-1}$-invariant.
\end{prop}

\begin{proof}
	If $\mathcal{R}$ is a Markov partition of $f$ and $I$ is an $s$-boundary component of $R_i$, then $I$ is an $s$-boundary component of a horizontal sub-rectangle $H \subset R_i$ of $(f, \mathcal{R})$. Since $f(H) = V \subset R_k$ is a vertical sub-rectangle of $R_k \in \mathcal{R}$, $f(I)$ must be contained in the $s$-boundary of $R_k$. This shows that $\partial^s \mathcal{R}$ is $f$-invariant. The $f^{-1}$-invariance of $\partial^u \mathcal{R}$ is similarly proven.
	
	Now, assume that $\partial^s \mathcal{R}$ is $f$-invariant and $\partial^u \mathcal{R}$ is $f^{-1}$-invariant. Let $C$ be a nonempty connected component of $f^{-1}(\overset{o}{R_k}) \cap \overset{o}{R_i}$. We claim that $C$ is the interior of a horizontal sub-rectangle of $R_i$. Take $x \in C$ and let $\overset{o}{I_x}$ be the interior of the horizontal segment of $R_i$ passing through $x$, and let $\overset{o}{I_x'}$ be the connected component of $\mathcal{F}^s \cap C$ containing $x$. Clearly, $\overset{o}{I_x'} \subset \overset{o}{I_x}$, but if $I_x \neq I_x'$, at least one endpoint $z$ of $I_x'$ lies in $\overset{o}{I_x}$. Therefore, $z$ is in the interior of $R_i$. However, $z$ must be equal to $f^{-1}(z')$ for some $z' \in \partial^u R_k$ (or we could extend $I_x'$ a little bit more inside $\overset{o}{R_i}$), and since $\partial^u \mathcal{R}$ is $f^{-1}$-invariant, $z \in \partial^u \mathcal{R}$, which is a contradiction as $\overset{o}{R_i} \cap \partial^u \mathcal{R} = \emptyset$. Therefore, $z \in \partial^u \mathcal{R} \cap \overset{o}{R_i} = \emptyset$. Similarly, we can show that $f(C)$ is the interior of a vertical sub-rectangle of $R_k$.
\end{proof}

The boundary of a compact arc $I \subset S$ consists of two distinct points that we call the \emph{endpoints} of $I$. With this convention, we introduce a definition that captures the main properties of the stable and unstable boundaries of a Markov partition.

\begin{defi}\label{Defi: Adapted graph}
	Let $\delta^s=\{I_i\}_{i=1}^n$ be a family of compact stable intervals not reduced to a point. The family $\delta^s$ is an $s$-\emph{graph adapted} to $f$ if:
	\begin{itemize}
		\item[i)] Their union, $\cup \delta^s:=\cup_{i=1}^n I_i$, is $f$-invariant.
		\item[ii)] For all $i\in \{1,\cdots,n\}$, $\exists p\in\textbf{Sing}(f)$ such that $I_i \subset \mathcal{F}^s(p)$.
		\item[iii)] For all $p\in \textbf{Sing}(f)$ and every stable separatrix of $\mathcal{F}^s(p)$, there exists $I_i\in \delta^s$ contained in such separatrix and having $p$ as an endpoint.
		\item[iv)] The endpoints of every $I_i$ belong to $\mathcal{F}^u(\textbf{Sing}(f))$.
	\end{itemize}
\end{defi}

An $u$-\emph{graph adapted} to $f$, $\delta^u := \{J_i\}_{i=1}^n$, is similarly defined, but with the intervals contained in unstable leaves of the singularities, and $\cup \delta^u:= \cup_{i=1}^n \{J_i\}_{i=1}^n $ must be $f^{-1}$-invariant. We are now going to construct adapted $s$-graphs starting with a family of $f^{-1}$-invariant unstable segments.

\begin{lemm}\label{Lemm: the generated graph is adapted}
Let $\mathcal{J} = \{J_j\}_{j=1}^n$ be a family of nontrivial compact segments, each contained in $\mathcal{F}^u(\textbf{Sing}(f))$, and such that $\cup \mathcal{J} = \cup_{j=1}^n J_j$ is $f^{-1}$-invariant.

Let $\delta^s(\mathcal{J}) = \{I_i\}_{i=1}^m$ be the family of all stable segments that have one endpoint in \textbf{Sing}$(f)$, the other endpoint in $\cup \mathcal{J}$, and whose interior is disjoint from $\cup \mathcal{J}$. Then the family $\delta^s(\mathcal{J})$ is an $s$-adapted graph for $f$, and we call $\delta^s(\mathcal{J})$ the $s$-\emph{adapted graph generated by} $\mathcal{J}$.
\end{lemm}

\begin{proof}
	By construction, each interval in the graph $\delta^s(\cJ)$ is contained in $\mathcal{F}^s(\textbf{Sing}(f))$, has one endpoint in $\textbf{Sing}(f)$ and the other in $\mathcal{J} \subset \mathcal{F}^u(\textbf{Sing}(f))$, and for every stable separatrice of any point in $\textbf{Sing}(f)$, there exists a segment in $\delta^s(\mathcal{J})$ that belongs to that separatrice. It remains to show that $\cup \delta^s(\mathcal{J})$ is $f$-invariant.
		
	Let $I:=I_i$ be an interval in $\delta^s(\mathcal{J})$. $I$ is contained in a stable separatrice of a point $p \in \textbf{Sing}(f)$, therefore $f(I)$ is contained in a stable separatrice of $f(p)$, and within that separatrice, there exists an interval $I'$ of $\delta^s(\mathcal{J})$. If $f(I)$ is not a subset of $I'$, then $I' \subset f(I)$, which implies that $I'$ has an endpoint $x'$ in $\cup \mathcal{J}$ that belongs to the interior of $f(I)$. Since $\cup \mathcal{J}$ is $f^{-1}$-invariant, then $f^{-1}(x') \in \cup \mathcal{J}$ is in the interior of $I$, which contradicts our definition of $I$ because the interior of $I$ cannot intersect $\cup \mathcal{J}$. Consequently, $f(I) \subset I' \subset \cup \mathcal{J}$, and thus $\cup \delta^s(\mathcal{J})$ is $f$-invariant.
\end{proof}

If $\mathcal{I}$ is a family of stable segments contained in $\mathcal{F}^s(\textbf{Sing}(f))$ and its union is $f$-invariant, the $u$-\emph{adapted graph generated by} $\mathcal{I}$ is similarly defined and is denoted by $\delta^u(\mathcal{I})$.

\subsubsection{Compatibles graphs.} Even if $\delta^s$ and $\delta^u$ are graphs adapted to $f$, their union is not necessarily the boundary of a Markov partition of $f$. We must require a more subtle relation between the adapted graphs in order for them to bound bi-foliated open discs and then take the rectangles in our partition as the closures of such discs. The relation between them is given in \ref{Defi: Compatible graphs}. Before presenting it, we need to introduce a few new concepts.

\begin{defi}\label{Defi: Regular part delta-s}
	Let $\delta^s=\{I_i\}_{i=1}^n$ be an $s$-graph adapted to $f$. The \emph{regular part} of $\delta^s$ is  $\overset{o}{\delta^s}:=\cup \delta^s \setminus \cup_{i=1}^n \{ \partial I_i \}$. The regular part of an adapted $u$-graph is similarly obtained by removing the endpoints of each interval in $\delta^u$ and take their union.
\end{defi}

\begin{defi}\label{Defi: Rail regular/extremal}
	Let $\delta^s$ be an $s$-graph adapted to $f$.
	\begin{itemize}
		\item  A \emph{regular $u$-rail} of $\delta^s$ is a unstable interval $J$ whose interior is disjoint from $\cup\delta^s$ but has endpoints in $\overset{o}{\delta{s}}$.
		
		\item An \emph{extreme $u$-rail}  of $\delta^s$ is a unstable segment $J$ that have at lest one endpoint not contained in  $\overset{o}{\delta^s}$ and  for which there exists an \emph{embedded rectangle} \footnote{This mean the parametrization of $R$ is an homeomorphism in the unitary rectangle $\II^2$} $R$ such that: $\partial^s R\subset \cup\delta^s$,  $J$ is one  $u$-boundary component of $R$  and  any other vertical leaf of $R \setminus J$ is a regular $u$-rail of $\delta^s$.
	\end{itemize}
	
	The set of extreme $u$-rails of $\delta^s$ is denoted by \textbf{Ex}$\,^u(\delta^s)$.
\end{defi}

\begin{defi}\label{Defi: Compatible graphs}
	Let $\delta^s = \{I_i\}_{i=1}^n$ and $\delta^u = \{J_j\}_{j=1}^m$ be $s$-graphs and $u$-graphs adapted to $f$, respectively. They are \emph{compatible} if they satisfy the following properties:
	\begin{itemize}
		\item The endpoints of $\delta^s$ belong to $\cup \delta^u := \cup_{j=1}^m J_j$, and the endpoints of $\delta^u$ belong to $\cup \delta^s := \cup_{i=1}^n I_i$.
		\item The extreme $u$-rails of $\delta^s$ are contained in $\cup \delta^u$, and the extreme $s$-rails of $\delta^u$ are contained in $\cup \delta^s$.
	\end{itemize}
\end{defi}

\subsubsection{The existence of Markov partition}

\begin{prop}\label{Prop:Compatibles implies Markov partition}
If $\delta^u$ and $\delta^s$ are graphs adapted to $f$ and compatible, then the closure in $S$ of every connected component of
$$
\overset{o}{S} := S \setminus \left[ \cup \textbf{Ex}^s(\delta^u) \cup \cup \textbf{Ex}^u(\delta^s) \right]
$$
is a rectangle adapted to $f$ whose stable and unstable boundaries are contained in $\cup \delta^s$ and $\cup \delta^u$, respectively. Furthermore, the family $\mathcal{R}(\delta^s, \delta^u)$ given by such rectangles is an adapted Markov partition of $f$.
\end{prop}

The proof will follow from a series of lemmas.

\begin{lemm}\label{Lemm: Finite c.c.}	
	The set $\overset{o}{S}=S\setminus (\delta^s \cup \text{Ex}^u(\delta^s))$ has a finite number of connected components.
\end{lemm}

\begin{proof}
	Since there are only a finite number of endpoints in $\delta^s$ and $\delta^u$ and a finite number of stable and unstable separatrixes passing through each endpoint, there is a finite number of (compact) extreme rails of $\delta^s$ and $\delta^s$. Finally, since $\delta^s \cup \text{Ex}^u(\delta^s))$ consists of a finite number of compact intervals, their complement must have a finite number of connected components.
\end{proof}

\begin{lemm}\label{Lemm: Equivalent class rectangle}
	Let $r$ be a connected component of $\overset{o}{S}$. Then $R:=\overline{r}$ is a rectangle adapted to $f$, whose stable and unstable boundaries are contained in $\cup\delta^s$ and $\cup\delta^u$, respectively.
\end{lemm}

\begin{proof}
It is clear that $r$ is open, connected, its closure is compact, and it does not contain any singularity in its interior, so it is trivially bi-foliated. By Definition \ref{Defi: Rectangulo}, $R := \overline{r}$ is a rectangle. Furthermore, its stable and unstable boundaries must lie in the adapted $s$ and $u$ graphs, as they must be contained in $S\setminus \overset{o}{S}$.
\end{proof}

\begin{proof}[Proposition  \ref{Prop:Compatibles implies Markov partition}]
In view of Lemma \ref{Lemm: Finite c.c.} and Lemma \ref{Lemm: Equivalent class rectangle}, the family $\cR$ is a finite collection of rectangles, as their interiors correspond to different connected components of $\overset{o}{S}$, and their interiors are disjoint. At the same time, since $\partial^{s,u} \cR = \delta^{s,u}$, the stable boundary of $\cR$ is $f$-invariant, and its unstable boundary is $f^{-1}$-invariant. By Proposition \ref{Prop: Markov criterion boundary}, $\cR$ is a Markov partition of $f$.
The fact that $\delta^s$ and $\delta^u$ are adapted and compatibles, implies directly that $\cR$ is an adapted Markov partition.
\end{proof}

\subsubsection{How to obtain compatible graphs?}

\begin{lemm}\label{Lemm: Iteration to be adapted}
	Let $\delta^s$ and $\delta^u$ be graphs adapted to $f$. Then there exists an $n := n(\delta^s, \delta^u) \in \mathbb{N}$ such that $\delta^s$ and $h^n(\delta^u)$ are compatible whenever $m \geq n$, and $n$ is the minimum natural number with this property.
\end{lemm}

\begin{proof}
	
	First, we prove that there exists $N_1 \in \mathbb{N}$ such that for all $n > N_1$, $\delta^s$ contains the extreme points of $h^n(\delta^u)$ and $h^n(\delta^u)$ contains the extreme points of $\delta^s$. Next, we establish an $N_2 \in \mathbb{N}$ such that for all $n > N_2$, the extreme $u$-rails of $\delta^s$ are in $h^n(\delta^u)$ and the extreme $s$-rails of $h^n(\delta^u)$ are contained in $\delta^s$. These conditions imply that:
	
	$$
	\{n\in \NN \mid \delta^s \text{ is compatible with }f^n(\delta^u)\}\neq \emptyset,
	$$
	
	and then, we can define $n := n(\delta^s, \delta^u)$ as the minimum of this set.
	
	Let's start by assuming that  $\delta^u=\{J_j\}_{j=1}^h$ and $\delta^s=\{I_i\}_{i=1}^l$. We define:
	$$
	L^u:=\min\{\mu^s(J_j) \mid 1\leq j\leq h\}.
	$$
	It is clear that for every $n\in \NN$ and every interval $J_j$
	$$
	\mu^s(f^n(J_i))=\lambda^n\mu^s(J_i)\leq \lambda^n L^u.
	$$
	
	Let $z\in \delta^s$ be an extreme point of $\delta^s$. Since $\delta^s$ is an adapted graph, the point $z$ lies on the unstable leaf of a singularity of $f$. Let $[p_z,z]^u$ denote the unique compact interval of the unstable leaf passing through $z$ and connecting a singularity $p_z$ of $f$ with $z$. Now, let's define:
	$$
	F^u:=\max\{\mu^s([p_z,z]^u) \mid z \text{ is extreme point of  } \delta^s \}.
	$$

	By the uniform expansion in the unstable leaves of $\mathcal{F}^u$, there exists $n_1\in \mathbb{N}$ such that for every $n\geq n_1$, we have $\lambda^nL^u > F^u$. Moreover, if $n > n_1$ and $z$ is an extreme point of $\delta^s$, there exists a unique $J_j$ in $\delta^u$ such that $z$ lies on the same separatrice as $f^n(J_j)$. However,
	$$
	\mu^u(f^n(J_j))\geq \lambda^nL^u > F^u\geq \mu^s([p_z,z]^u),
	$$
	implying that $z$ belongs to $f^n(J_j)$, or equivalently $z$ belongs to $f^n(\delta^u)$ for all $n\geq n_1$.

	In the same way, but using $f^{-1}$ and the measure $\mu^u$, we can deduce the existence of $n_2\in \NN$ such that for all $n\geq n_2$ and every extreme point $z$ of $\delta^u$, $z$ is also an extreme point of $f^{-n}(\delta^s)$. In other words, for all $n\geq n_2$ and every extreme point $z$ of $f^n(\delta^u)$, $z$ is contained in $\delta^s$. Let's take $N_1=\max\{n_1,n_2\}$.

	Now, let's consider the set of $u$-extreme rails of $\delta^s$. This set is finite, and each $u$-extreme rail is a closed interval.

	For each $p\in \textbf{Sing}(f)$ and each unstable separatrice $F^u_i(p)$ of $p$ (if $p$ is a $k$-prong, we consider $i=1,\cdots,k$), we define $J(p, i)\subset F^u_i(p)$ as the minimal compact interval containing the intersection of all the $u$-extreme rails of $\delta^s$ with the separatrice $F^u_i(p)$, in case that some $u$-extremal rail $J$ of of $\delta^s$ contains a singularity in its interior we consider just the sub-interval of $J$ contained in the separatrice $F^u_i(p)$. Now, we consider the following quantity, which is finite:
	
	$$
	M^u:=\max\{ \mu^s(J(p,i)) \mid p\in \textbf{Sing}(f) \text{ is a $k$- prong and } i=1,\cdots k  \}.
	$$
	
	Like $\delta^u=\{J_j\}_{j=1}^h$ the next quantity is finite too,
	$$
	G^u=\min\{\mu^s(J_j) \mid j=1,\cdots,h\}.
	$$

	By the uniform expansion in unstable leaves of $\mathcal{F}^u$, there exists $n_1\in \mathbb{N}$ such that for all $n\geq n_1$, $\lambda^n G^u \geq M^u$. Furthermore, for any $k$-prong $p\in \textbf{Sing}(f)$ and every $i=1, \cdots, k$, there exists an interval $J_j$ in $\delta^u$ such that $f^n(J_j)$ is contained in the same unstable separatrice as $J(p,i)$, indeed $J_j$ is the interval of $\delta^u$ contained in $F^u_i(p)$ with and end point in $p$ that is given by item $iii)$ in Definition \ref{Defi: Adapted graph} (in the $u$ case). Moreover, we have the following computation:
	
	$$
	\mu^s(f^n(J_j))=\lambda^n\mu^s(J_i)\geq \lambda^n G^u \geq M^u.
	$$

	Since $\mu^s(J(p,i)) \leq M^u$, this calculation implies that $J(p,i)\subset f^n(J_j)$. By the construction of $J(p,i)$, we deduce that for each $n\geq n_1$, any $u$-extreme rail of $\delta^s$ is contained in $f^n(\delta^u)$.

	A similar proof using $f^{-1}$ and the measure $\mu^u$ gives another natural number $n_2\in \NN$ such that for all $n\geq n_2$, any extreme $s$-rail of $\delta^u$ is contained in $f^{-n}(\delta^s)$, or equivalently, all extreme $s$-rails of $f^{n}(\delta^u)$ are contained in $\delta^s$. Let $N_2=\max(n_1,n_2)$. The conclusion is that the following quantity exists:
	$$
	n(\delta^s,\delta^u):=\min \{N\in \NN \mid \forall n\geq N \, \delta^s \text{ is compatible with } f^n(\delta^u) \},
	$$ 
	
	Which is the number we were looking for
\end{proof}

We can now summarize all these steps to construct a Markov partition for a \textbf{p-A} homeomorphism $f:S \to S$:

\begin{cons}\label{Cons: Recipe for Markov partitions}
	Let $p$ and $q$ be singular points of $f$, with separatrices $F^s(p)$ and $F^u(q)$ respectively, and let $z$ be a point in their intersection $F^s(p)\cap F^u(q)$. Consider the following construction:
	\begin{enumerate}
		\item Define $J^u(z)$ as the unstable interval $[p,z]^u$ in $F^u(p)$, which is referred to as the \emph{primitive segment}.
		
		\item Define $\cJ^u(z)=\cup_{i\in \NN}f^{-i}(J^u(z))$. Due to contraction in the unstable foliation, $\mathcal{J}^u(z)$ is a finite union of closed intervals and is $f^{-1}$-invariant.
		
		\item The graph $\mathcal{J}^u(z)$ satisfies the conditions of Lemma \ref{Lemm: the generated graph is adapted} and thus its generated $s$-graph is adapted to $f$. We denote it as $\delta^s(z)$.
		
		\item The graph $\delta^s(z)$ is $f$-invariant, and by Lemma \ref{Lemm: the generated graph is adapted} (applied to $f^{-1}$), its generated $u$-graph, $\delta^u(z)$, is adapted to $f$.
		
		\item By Lemma \ref{Lemm: Iteration to be adapted}, there exists a number $n(z) = n(\delta^s(z), \delta^u(z))$, called the \emph{compatibility coefficient} of $z$, such that, for all $n > n(z)$, $\delta^s(z)$ and $f^n(\delta^u(z))$ are compatible.
		
		\item Finally, for all $n\geq n(z)$, Proposition \ref{Prop:Compatibles implies Markov partition} implies the existence of an adapted Markov partition $\mathcal{R}(z,n)$ with a stable boundary equal to $\cup \delta^s(z)$ and an unstable boundary equal to $\cup f^n(\delta^u(z))$.
	\end{enumerate}
\end{cons}

Therefore, we recover a classic result.

\begin{coro}\label{Coro: Existence adapted Markov partitions}
	Every generalized \textbf{p-A} homeomorphism has adapted Markov partitions.
\end{coro}

\subsection{Primitive Markov partitions}\label{Sub-sec: Primitive MP and first intersection} 

We must  to apply Construction \ref{Cons: Recipe for Markov partitions} to a family of distinguished points known as \emph{first intersection points}. 

\subsubsection{First intersection points of $f$.}
\begin{defi}\label{Defi: first intersection points}	
	Let $f$ be a generalized pseudo-Anosov homeomorphism, $p, q \in \textbf{Sing}(f)$, $F^s(p)$ be a stable separatrice of $p$, and $F^u(q)$ be an unstable separatrice of $q$.
	A point $x \in [\cF^s(p) \cap \cF^u(q)] \setminus \textbf{Sing}(f)$ is a \emph{first intersection point} of $f$ if the stable interval $[p,x]^s \subset F^s(p)$ and the unstable segment $[q,x]^u \subset F^u(q)$ have disjoint interiors. In other words:
	$$
	(p,x]^s\cap (q,x]^u=\{x\}.
	$$. 
\end{defi}

\begin{figure}[h]
	\centering
	\includegraphics[width=0.4\textwidth]{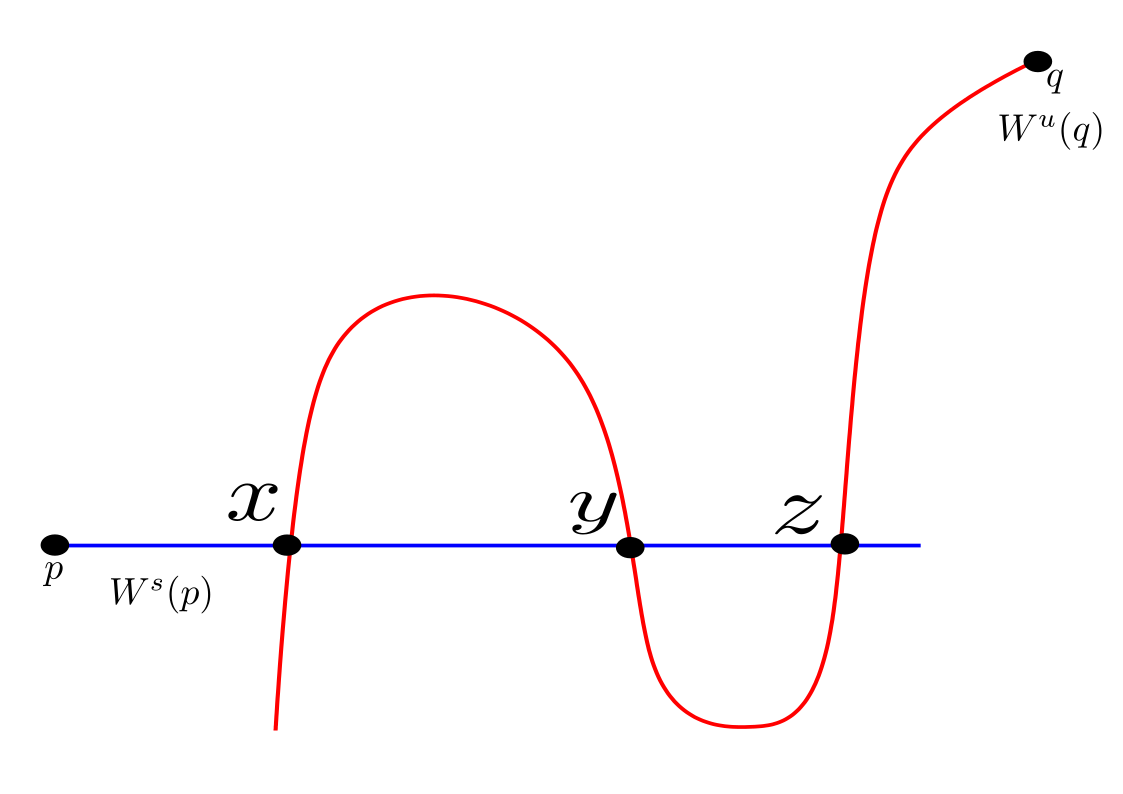}
	\caption{$z$ is a first intersection point, but $z'$ is not.}
	\label{Fig: First intersection point}
\end{figure}

\begin{lemm}\label{Lemm: Existencia puntos de primera intersecion}
	There exists at least one first intersection point for $f$
\end{lemm}

\begin{proof}
	Let $I$ be a compact interval contained in $F^s(p)$ with one endpoint equal to $p$. Take a singularity $q$ (that could be equal to $p$), like any unstable separatrice $F^u(q)$ of $q$ is minimal there exist a closed interval $J\subset F^u(q)$ such that $\overset{o}{J}\cap \overset{o}{I}=\emptyset$.
	Like $J$ and $I$ are compact sets, their intersection $J\cap I$ consist in a finite number of points, $\{z_0,\cdots,z_n\}$. We can assume that $n>1$ and we orient the interval $J$ pointing towards $q$ and whit this orientation $z_i<z_{i+1}$. In this manner:
	\begin{itemize}
		\item If $p=q$, then $z_0=p=q$ but $p\neq z_1$ and we take $z=z_1$.
		\item If $p\neq q$, then $p\neq z_0 \neq q$ and we take $z=z_0$.
	\end{itemize}
	
	Clearly, $(p,z]^s\cap (q,z]^u=\{z\}$, and $z$ is a first intersection point.
\end{proof}

\begin{lemm} \label{Lemm: Image of first intersection is first intersection}
	If $z$ is a first intersection point of $f$, then $f(z)$ is also a first intersection point of $f$.
\end{lemm}

\begin{proof}
	If $z$ is a first intersection point, there exist singular points of $f$, $p$ and $q$, such that:
	$$
	\{z\}=(p,z]^s\cap (q,z]^u.
	$$
	Therefore,
	$$
	\{f(z)\}=(f(p),f(z)]^s\cap (f(q),f(z)]^u
	$$
	is a first intersection point of $f$.
\end{proof}

The following result establishes the finiteness of the orbits of these first intersection points:

\begin{prop}\label{Prop: Finite number first intersection points}
	Let $f$ be a generalized pseudo-Anosov homeomorphism. Then there exists a finite number of orbits of first intersection points under $f$.
\end{prop}

\begin{proof}

	There is only a finite number of singularities and separatrices of $f$, so Proposition \ref{Prop: Finite number first intersection points} is a direct consequence of the following lemma.
\end{proof}

	\begin{lemm}\label{lemm: finite intersect in fundamental domain}
		Let $p$ and $q$ be singularities of $f$, and let $F^s(p)$ and $F^u(q)$ be stable and unstable separatrices of these points. There exists a finite number of orbits of first intersection points contained in $F^s(p)\cap F^u(q)$
	\end{lemm}
	
	\begin{proof}
		
		Let $n$ be the smallest natural number such that $g:=f^n$ fixes the separatrices $F^s(p)$ and $F^u(q)$. Take a point $z_0 \in F^s(p)\cap F^u(q)$, which may or may not be a first intersection point of $f$. The interval $(g(z_0),z_0]^s \subset F^s(p)$ serves as a fundamental domain for $g$ and contains a fundamental domain for $f$. To prove the lemma, we need to show that within this interval there exists a finite number of first intersection points.
		
		Consider a first intersection point $z$ in $F^s(p) \cap F^u(q)$. There exists an integer $k \in \mathbb{Z}$ such that $g^k(z)\in (g(z_0),z_0]^s \cap F^u(q)$. Therefore, there are two possible configurations:
		
		\begin{enumerate}
			\item Either $g^k(z)\in (g(z_0),z_0]^s \cap (q,g(z_0)]^u$, or
			\item $g^k(z)\in (g(z_0),z_0]^s \cap (g(z_0),\infty))^u$.
		\end{enumerate}

		We claim that option $(2)$ is not possible. If such a configuration were to exist, it would imply that $g(z_0) \in (p,g^k(z)]^s \cap [q,g^k(z)]^u$, and thus $g^k(z)$ would not be a first intersection point. However, Lemma \ref{Lemm: Image of first intersection is first intersection} implies that $f^{kn}(z)$ is a first intersection point, leading to a contradiction.
		
		Hence, we deduce that $g^k(z) \in (g(z_0),z_0]^s \cap [q,g(z_0)]^u \subset [g(z_0),z_0]^s \cap [q,g(z_0)]^u$, which is the intersection of two compact intervals. Therefore, this intersection is a finite set.
		
		This implies that every first intersection point $z\in F^s(p) \cap F^u(q)$ has some iteration lying in a finite set. Consequently, there are only a finite number of orbits of first intersection points in $F^s(p)\cap F^u(q)$.
		
	\end{proof}

\subsubsection{Primitive geometric Markov partitions.} If we applied the construction  \ref{Cons: Recipe for Markov partitions} to certain first intersection points of $f$, we obtain a distinguished family of Markov partitions.

\begin{defi}\label{Defi: Primitive Markov partitions}
	Let $z$ be a first intersection point of $f$. For all $n \geq n(z)$, the Markov partition $\mathcal{R}(z,n)$ constructed as indicated in \ref{Cons: Recipe for Markov partitions} is the \emph{primitive Markov partition} of $f$ generated by $z$ of order $n$.
\end{defi}

\begin{theo}\label{Theo: Conjugates then primitive Markov partition}
Let $f: S_f \rightarrow S_f$ and $g: S_g \rightarrow S_g$ be two \textbf{p-A} homeomorphisms that are topologically conjugated through a homeomorphism $h: S_f \rightarrow S_g$, i.e., $g = h \circ f \circ h^{-1}$. Let $z \in S_f$ be a first intersection point of $f$. The following statements are true:

\begin{itemize}
	\item[i)] $h(z)$ is a first intersection point of $g$.
	\item[ii)] $h(\cJ^u(z)) = \cJ^u(h(z))$.
	\item[iii)] $h(\delta^s(z)) = \delta^s(h(z))$ and $h(\delta^u(z)) = \delta^u(h(z))$.
	\item[iv)] The graph $\delta^s(z)$ is compatible with $f^n(\delta^u(z))$ if and only if the graph $\delta^s(h(z))$ is compatible with $g^n(\delta^u(h(z)))$.
	\item[v)] The compatibility coefficients of $z$ and $h(z)$ are equal, i.e., $n(z) = n(h(z))$.
	\item[vi)] $r$ is an equivalence class of $u$-rails for $\delta^s(z)$ if and only if $h(r)$ is an equivalence class of $u$-rails for $\delta^s(h(z))$. Similarly for equivalence classes of $s$-rails of $f^n(\delta^u(z))$ and $g^n(\delta^u(h(z)))$.
	\item[vii)] For all $n \geq n(z) = n(h(z))$, $h(\cR(p,z)) = \cR(p,h(z))$ is a primitive Markov partition of $g$ generated by $h(z)$ of order $n$.
\end{itemize}
\end{theo}

\begin{proof}
	First, $h$ send transverse foliations of $f$ into transverse foliations of $g$ and singularities of $f$ into singularities of $g$, that is how we use the hypothesis that $f$ and $g$ are pseudo-Anosov homeomorphisms. But even more, $h$ sends disjoint segments of stable and unstable leaves of $f$ to disjoint segments of stable and unstable leaves of $g$, therefore, if $z$ is first intersection point of $f$, $h(z)$ is first intersection point of $g$. This shows the accuracy of Item $i)$.
	
	A direct computation proof Item $ii)$:
	$$
	h(\cJ^u(z))=h(\cup f^{-n}(J^u(z)))=\cup g^{-n}(J^u(h(z)))= \cJ^u(h(z)).
	$$

	A compact interval $[p,x]^s \in \delta^s(z)$  have its interior disjoint from $\cJ^u(z)$ and  endpoint  $x\in \cJ^u(z)$. Then $h([p,x]^s)$ is a segment with disjoint interior of $h(\cJ^u(z))=\cJ^u(h(z))$ but endpoint $h(x)\in \cJ^u(h(z))$, this implies $h([p,x]^s)$ is contained in $\delta^s(h(z))$, hence $h(\delta^s(z))\subset \delta^s(h(z))$. The symmetric argument starting from the segment  $[p',x']^s\in \delta^s(h(z))$ and using $h^{-1}$ gives the equality. Analogously it is possible to show that $\delta^u(h(z))=h(\delta^u(z))$. The arguments in this paragraph proved the Item $iii)$.	
	
	Clearly the point $e$ is an extreme point of $\delta^s(z)$ if and only if $h(e)$ is an extreme point of $h(\delta^s(z))=\delta^s(h(z))$, moreover $e\in f^n(\delta^n(z))$ if and only if $h(e)\in g^{n}(\delta^u(h(z)))$. Therefore an extreme point $e$ of $\delta^s(z)$ is contained in $f^n(\delta^n(z))$ if and only if the extreme point $h(e)$ of $\delta^s(h(z))$ is contained in $g^{n}(\delta^u(h(z))$.	Another consequence is that the regular part of $\delta(z)$ is sent by $h$ to the regular part of $\delta(z)$ and an $u$-regular $J$-rail of $\delta(z)$ is such that $h(J)$ is a regular rail of $\delta^s(h(z))$.
	
	We claim that $J$ is an extremal rail of $\delta^s(z)$ if and only if $h(J)$ is an extremal lane of $\delta^u(h(z))$. Indeed if $R$ is the embedded rectangle of $S_f$ given by the definition of extreme rail having $J$ as a vertical boundary  component, therefore the embedded rectangle $h(R)$ satisfies the definition for $h(J)$ to be an extreme rail of $\delta^s(h(z))$, since  except for $h(J)$ all its other vertical segments are regular rails of $\delta^s(h(z))$ and it stable boundary is contained in $h(\delta^s(z))=\delta^s(h(z))$. Moreover, $J \subset f^n(\delta^u(z))$ if and only if $h(J)\subset g^n(\delta^u(h(z)))$.
	
	This proves that $\delta^s(z)$ is compatible with  $f^n(\delta^u(z))$ if and only if $\delta^s(h(z))$ is compatible with $g^n(\delta^u(h(z)))$. In this way the point $iv)$ is corroborated, but at the same time we deduce that $n(z)=n(h(z))$ because they are the minimum of the same set of natural numbers and thus the point $v)$ is obtained.
	
	Let $r$ be an equivalent class of $u$-rails for $\delta^s(z)$, as stated before $I$ is a regular rail for $\delta^s(z)$ if and only if $h(I)$ is a regular rail for $\delta^s(h(z))$. 	We denote $I\equiv_{\delta^s(z)}I'$ to indicate that any point in $I$ is $\delta^s(z)$ equivalent to any point in $I'$. In this manner $h(I)\equiv_{\delta^s(h(z))}h(I')$, because the image by $h$ of rectangles and regular rails realizing the equivalence between points in $I$ and points in $I'$ are rectangles and regular rails realizing the equivalence between points in $h(I)$ and points in ´$h(I')$. This implies that $h(r)$ is an equivalent class of $u$-rails for $\delta^s(h(z))$. Analogously the image by $h$ of an equivalent class $r'$ of $s$-rails for $f^n(\delta^u(z)$ is an equivalent class of $s$-rails for $g^n(\delta^u(h(z)))$.

	As the interior of a rectangle $R$ of $R(z,p)$ is a connected component of the intersection a equivalent class $r$ of $u$-rails of $\delta^s(z)$ with  a class $r'$ of $s$-rails for $f^n(\delta^u(z))$, then $h(R)$ is a connected component of the intersection of $h(r)\cap h(r')$ and $h(r)$ and $h(r')$ are $s,u$-rail classes for $\delta^s(h(z))$ and $g^n(\delta^u(h(z)))$, by definition $h(\overset{o}{R})$ correspond to the interior of a rectangle of $\cR(h(z),n)$ and then $h(\cR(z,n))=\cR(h(z),n)$. In this manner we obtain Item $vi)$

	Since the interior of a rectangle  $R$ of $R(z,p)$ is a connected component of the intersection an $r$ equivalent class of $u$-rails of $\delta^s(z)$ with an $r'$ class of $s$-rails for $f^n(\delta^u(z))$, then $h(R)$ is a connected component of the intersection of $h(r)\cap h(r')$, anyway $h(r)$ and $h(r')$ are $s,u$-rail classes for $\delta^s(h(z))$ and $g^n(\delta^u(h(z)))$, therefore $h(\overset{o}{R})$ correspond to the interior of a rectangle of $\cR(h(z),n)$ and then $h(\cR(z,p))=\cR(h(z),p)$. This probe the Item $vii)$ and finish our proof.
\end{proof}

In particular, if $f = g = h$, we have $f = f \circ f \circ f^{-1}$. We deduce that $n(z) = n(f(z))$ for every first intersection point $z$ of $f$. Therefore, the quantity $n(f^m(z))$ is constant over the entire orbit of $z$. In view of Proposition \ref{Prop: Finite number first intersection points}, there are only a finite number of orbits of first intersection points, and the following corollary follows immediately.

\begin{coro}\label{Coro: n(f) number}
	There exists $n(f) \in \NN$, called the \emph{compatibility order} of $f$, such that:
	$$
	n(f) = \max\{n(z) : z \text{ is a first intersection point of } f\}.
	$$
\end{coro}

\begin{coro}\label{Coro: n(f) conjugacy invariant}
	Let $f: S_f \rightarrow S_f$ and $g: S_g \rightarrow S_g$ be two pseudo-Anosov homeomorphisms that are conjugated through a homeomorphism $h: S_f \rightarrow S_g$. Then the compatibility order of $f$ and $g$ coincides, i.e., $n(f) = n(g)$.
\end{coro}

\begin{proof}
	Items $i)$ and $v)$ of Theorem \ref{Theo: Conjugates then primitive Markov partition} imply the following set equalities:
	\begin{align*}
		\{n(z): z \text{ is a first intersection point of } f \}= \\
		\{n(h(z)): z \text{ is a first intersection point of } f \}=\\
		\{n(z'): z' \text{ is a first intersection point of } g \}.
	\end{align*}
	Therefore, it follows that its maximum is the same and finally that $n(f) = n(g)$.
\end{proof}

\begin{defi}\label{Defi: primitive Markov parition}
	Let $n \geq n(f)$. The set of primitive Markov partitions of $f$ of order $n$ consists of all the Markov partitions of order $n$ generated by the first intersection points of $f$, and is denoted by:
	$$
	\mathcal{M}(f,n) := \{\mathcal{R}(z,n) : z \text{ is a first intersection point of } f\}.
	$$
\end{defi}

Another application of Theorem \ref{Theo: Conjugates then primitive Markov partition} in the case $f=g=h$ is that if $z$ is a first intersection point of $f$ and $n\geq n(f)$, then $f(\cR(z,n))=\cR(f(z),n)$. This yields the following corollary:

\begin{coro}\label{Coro: Finite orbits of primitive Markov partitions}		
	Let $n\geq n(f)$. Then, there exists a finite but nonempty set of orbits of primitive Markov partitions of order $n$.
\end{coro}

\begin{proof}
	As $n\geq n(f)$, there exists at least one Markov partition $\cR(z,n)$ in $\cM(f,n)$, and therefore the orbit of $\cR(z,n)$, given by $\{\cR(f^m(z),n)\}_{m\in \ZZ}$, is contained in $\cM(f,n)$.
	
	Let $\{z_1, \cdots, z_k\}$ be a set of first intersection points of $f$ such that any other first intersection point can be written as $f^m(z_i)$ for a unique $i \in \{1,\cdots, k\}$, and no two different points $z_i$ and $z_j$ are in the same orbit. . Let $\cR(z,n)$ be a primitive Markov partition of order $n$. If $z$ is any first intersection point, $z=f^m(z_i)$ and we have:
	$$
	f^m(\cR(z_i,n))=\cR(f^m(z_i),n)=\cR(z,n). 
	$$
	Therefore, in $\cM(f,n)$, there are at most $k$ different orbits of primitive Markov partitions of order $n$.
\end{proof}

\subsubsection{The geometric type of the orbit of a primitive Markov partition.}

We are interested in studying the set of geometric types produced by the orbit of a primitive Markov partition. Therefore, it is important to understand the behavior of a Markov partition under the action of a homeomorphism that preserve the orientation.

Let  $\cR=\{R_i\}_{i=1}^n$ be a geometric Markov partition of a generalized pseudo-Anosov homeomorphism $f:S\rightarrow S$, and let $g:S\rightarrow S$ be another generalized pseudo-Anosov homeomorphism that is conjugate to $f$ by a homeomorphism $h:S\rightarrow S$ preserving the orientation. The function $h$ maps the foliations and singularities of $f$ to the foliations and singularities of $g$ while preserving the orientation of the transversal foliations.

It is clear that if $r:[0,1]\times [0,1] \rightarrow R \subset S$ is an oriented rectangle for $f$, then $h\circ r: [0,1]\times [0,1] \rightarrow h(R) \subset S$ is an oriented rectangle for $g$. This is because $h$ determines a unique orientation for the transverse foliations of $h(R)$. Furthermore,  $h(\cR)=\{h(R_i)\}_{i=1}^n$ has a $g$-invariant horizontal boundary and a $g^{-1}$-invariant vertical boundary. These properties are deduced from the conjugation provided by $h$. Based on these observations, we can now give the following definition.

\begin{defi}\label{Defi: induced geometric Markov partition}
	Let $f:S\rightarrow S$ be a generalized pseudo-Anosov homeomorphism, and let $h:S\rightarrow S$ be an orientation-preserving homeomorphism, and $g:=h\circ f \circ h^{-1}$. If $\cR$ is a geometric Markov partition of $f$, the \emph{induced geometric Markov partition} of $g$ by $h$ is the Markov partition $h(\mathcal{R})=\{h(R_i)\}_{i=1}^n$. For each $h(R_i)$, we choose the unique orientation in the vertical and horizontal foliations such that $h$ preserves both orientations at the same time.
\end{defi}

The following lemma clarifies the correspondence between the horizontal and vertical rectangles of the partition $\mathcal{R}$ and those of $h(\mathcal{R})$.

\begin{lemm}\label{Lemm: Conjugated partitions same type.}
	Let $f$ and $g$ be two generalized pseudo-Anosov homeomorphisms conjugated through a homeomorphism $h$ that preserves the orientation. Let $\cR=\{R_i\}_{i=1}^n$  be a geometric Markov partition for $f$, and let $h(\cR)=\{h(R_i)\}_{i=1}^n$ be the geometric Markov partition of $g$ induced by $h$. In this situation, $H$ is a horizontal sub-rectangle of $(f,\mathcal{R})$ if and only if $h(H)$ is a horizontal sub-rectangle of $(g,h(\mathcal{R}))$. Similarly, $V$ is a vertical sub-rectangle of $(f,\mathcal{R})$ if and only if $h(V)$ is a vertical sub-rectangle of $(g,h(\mathcal{R}))$.
\end{lemm}

\begin{proof}
	
	Observe that $h(\overset{o}{R_i})=\overset{o}{h(R_i)}$. Therefore, $C$ is a connected component of $\overset{o}{R_i}\cap f^{\pm}(\overset{o}{R_j})$ if and only if $h(C)$ is a connected component of $\overset{o}{h(R_i)}\cap g^{\pm}(\overset{o}{h(R_j)})$.
	
\end{proof}

Let $T(f,\cR)=(n,\{h_i,v_i\},\Phi:=(\rho,\epsilon))$ be the geometric type of the geometric Markov partition $\cR$, and let $T(g,h(\cR))=(n',\{h'_i,v'_i\},\Phi'_T:=(\rho',\epsilon'))$be the geometric type of the induced Markov partition.  A direct consequence of Lemma  \ref{Lemm: Conjugated partitions same type.} is that: $n=n'$, $h_i=h'_i$ and $v_i=v'_i$.

Let $\{\overline{H^i_j}\}_{j=1}^{h_i}$ be the set of horizontal sub-rectangles of $h(R_i)$, labeled with respect to the induced vertical orientation in $h(R_i)$. Similarly, we define $\{\overline{V^k_l}\}_{l=1}^{v_k}$ as the set of vertical sub-rectangles of $h(R_k)$, labeled with respect to the horizontal orientation induced by $h$ in $h(R_k)$.

By the we choose the orientations in $h(R_i)$ and in $h(R_k)$ is clear that:
$$
h(H^i_j) =\overline{H^i_j} \text{ and } h(V^k_l) =\overline{V^k_l}.
$$
Even more, using the conjugacy we have that: if $f(H^i_j)=V^k_l$ then
$$
g(\overline{H^i_j})= g(h(H^i_j))=h(f(H^i_j))=h(V^k_l)=\overline{V^k_l}.
$$

This implies that in the geometric types, $\rho = \rho'$.

Suppose that $\overline{V^k_l} = g(\overline{H^i_j})$, so the latter set is equal to $h \circ f \circ h^{-1}(\overline{H^i_j})$. The homeomorphism $h$ preserves the vertical orientations between $R_i$ and $h(R_i)$, as well as between $R_k$ and $h(R_k)$. Therefore, $f$ sends the positive vertical orientation of $H^i_j$ with respect to $R_i$ to the positive vertical orientation of $V^k_l$ with respect to $R_k$ if and only if $g$ sends the positive vertical orientation of $\overline{H^i_j}$ with respect to $h(R_i)$ to the positive vertical orientation of $\overline{V^k_l}$ with respect to $h(R_k)$. It follows then that $\epsilon(i,j) = \epsilon'(i,j)$. We summarize this discussion in the following Theorem

\begin{theo}\label{Theo: Conjugated partitions same types}
	Let $f$ and $g$ be generalized pseudo-Anosov homeomorphisms conjugated through a homeomorphism $h$ that preserves the orientation, i.e., $g = h \circ f \circ h^{-1}$. Let $\cR=\{R_i\}_{i=1}^n$ be a geometric Markov partition for $f$, and let $h(\cR)=\{h(R_i)\}_{i=1}^n$ be the geometric Markov partition of $g$ induced by $h$. In this situation, the geometric types of $(g,h(\cR))$ and $(f,\cR)$ are the same.
\end{theo}

Since $f$ preserves the orientation, we can consider the case where $f = g = h$, and apply the previous theorem. If $n \geq n(f)$ and $z$ is a first intersection point of $f$, then for all $m \in \mathbb{Z}$, the geometric type of $f^m(\cR(z,n)) = \cR(f^m(z),n)$ is the same as the geometric type of $\cR(z,n)$. This leads to the following corollary.

\begin{coro}\label{Coro: constant type in orbit}
	For every primitive Markov partition $\cR(z,n)$ with $n \geq n(f)$ and $z$ as a first intersection point of $f$, we have the following property for all $m \in \mathbb{Z}$:
	$$
	T(\cR(z,n))= T(\cR(f^m(z),n)).
	$$
\end{coro}

In view of Corollary \ref{Coro: Finite orbits of primitive Markov partitions}, for all $n \geq n(f)$ there exists a finite number of primitive Markov partition orbits of $f$ of order $n$. Moreover, the geometric type is constant within each orbit of such a Markov partition. Therefore, for all $n \geq n(f)$, there exists only a finite number of distinct geometric types corresponding to the geometric types of partitions in the set $\cM(f,n)$. We summarize this in the following theorem.

\begin{theo}\label{Theo: finite geometric types}
	For every $n > n(f)$, the set $\cT(f,n) := \{T(\cR) : \cR \in \cM(f,n)\}$ is finite. These geometric types are referred to as the \emph{primitive geometric types} of $f$ of order $n$.
\end{theo}

Of course, within all the primitive geometric types, there is a distinguished family that, in some sense, has the least complexity with respect to its compatibility coefficient.

\begin{defi}\label{defi: Canonical types}
	
	The canonical types of $f$ are the set $\cT(f,n(f))$.
\end{defi}

In fact, in Corollary \ref{Coro: n(f) conjugacy invariant}, we have proved that if $f$ and $g$ are topologically conjugate, then $n(f)=n(g)$. By Theorem  \ref{Theo: Conjugated partitions same types}, we deduce that for every $n\geq n(f)=n(g)$, the set of primitive types of $f$ of order $n$ coincides with the set of primitive geometric types of $g$ of order $n$. In particular, we have the following corollary.

\begin{coro}\label{Coro: conjugated implies same cannonical types}
	If $f$ and $g$ are conjugate generalized pseudo-Anosov homeomorphisms,then,  for all $m\geq n(f)=n(g)$, $\cT(f,m)=\cT(g,m)$ and their canonical types are the same, i.e., $\cT(f,n(f))=\cT(g,n(g))$.
\end{coro}

%% file: Refinamientos/MainRefinamientos.tex
\section{Geometric Markov partitions with prescribed properties.}\label{Section: GM con propiedades prescritas}

In this section, we are interested in the following problem. Let $f:S\to S$ be a \textbf{p-A} homeomorphism. Assume that $f$ has a geometric Markov partition $\cR$ and that the geometric type $T$ of the pair $(f,\cR)$ is known. Can you construct another geometric Markov partition $\cR'$ for $f$ such that the geometric type $T'$ of the new pair $(f,\cR')$ can be computed, and furthermore, that the new geometric type $T'$ and the new partition $\cR'$ have some combinatorial or topological property of our interest? For example, can we obtain a Markov partition whose incidence matrix is binary? What about predetermining which periodic points of $f$ must belong to the boundary of the Markov partition? Can we obtain a Markov partition with the corner property or one that is well-suited? Or can we obtain a Markov partition whose rectangles are embedded? We provide some answers and constructions for these questions.

In subsection \ref{Subsec: Generales refinamientos}, definitions of refinements for Markov partitions, their canonical geometrization, and therefore their geometric type are presented. In subsequent sections, we will construct refinements of Markov partitions with useful properties for performing an algorithmic classification of the conjugacy classes of \textbf{p-A} homeomorphisms, as done by the author in \cite{IntiThesis}. The classification itself can be found in \cite{CruzDiaz2024}; meanwhile, here we will focus on the necessary algorithmic and constructive aspects. In subsection \ref{Subsec: Refinamiento Binario}, a geometric Markov partition with a binary \emph{incidence matrix} is constructed, and a procedure to determine its geometric type is developed. Associated with a pair $(f,\cR)$, there exists a \emph{sub-shift of finite type} $(\Sigma_A,\sigma_A)$ (where $A$ is the incidence matrix of the pair $(f,\cR)$) that is semiconjugate to $f$ through a map $\pi_{(f,\cR)}:\Sigma_A \to S$. In subsection \ref{Subsec: Codigos periodicos ref}, given a finite family of periodic codes in $\Sigma_A$, we construct another Markov partition for $f$ such that the projections by $\pi_{(f,\cR)}$ of these periodic codes lie on the stable boundary of the new Markov partition. We conclude the subsection by providing formulas to compute the geometric type of the new geometric Markov partition.

\subsection{Refinement of a geometric  Markov partition}\label{Subsec: Generales refinamientos}

Let's clarify what a refinement of a Markov partition is. The following definition is given for horizontal sub-rectangles; for vertical sub-rectangles, the definition is completely symmetrical. 

\begin{defi}\label{Defi: Refianmiento horizontal}
	Let $f:S \to S$ be a \textbf{p-A} homeomorphism and let $\cR=\{R_i\}_{i=1}^n$ be a Markov partition of $f$. A family of labeled rectangles $\cR'=\{H_{(i,h)}\mid i\in\{1,\cdots, n\}, \, h\in\{1,\cdots, H(i)\} \text{ for } H(i)\in \mathbb{N}_+\}$ is  a \emph{refinement by horizontal} (resp. vertical) sub-rectangles of the pair $(f,\cR)$ if the following properties are satisfied:
	
	\begin{itemize}
		\item The family $\cR'$ is a Markov partition of $f$.
		\item For all $i\in\{1,\cdots, n\}$ and $h\in\{1,\cdots, H(i)\}$, $H_{(i,h)}$ is a horizontal (resp. vertical) sub-rectangle of $R_i \in \cR$.
		\item The union of these rectangles in $\cR'$ that are contained within the same rectangle $R_i \in \cR$ is equal to $R_i$: $\bigcup_{h=1}^{H(i)} H_{(i,h)} = R_i$.
	\end{itemize}
		The set of \emph{horizontal labels of the refinement} $(f,\cR') $ is given by:
		\begin{equation}
\mathcal{H}((f,\mathcal{R}')) := \{(i,h) \mid i\in\{1,\cdots, n\}, \, h\in\{1,\cdots, H(i)\} \text{ for } H(i)\in \mathbb{N}_+\}.
		\end{equation}
\end{defi}

We say that $(f, \mathcal{R}')$ is a refinement of $(f, \mathcal{R})$ to indicate that $\mathcal{R}'$ is a refinement of $\mathcal{R}$, without specifying whether the refinement is by horizontal or vertical rectangles unless necessary to avoid ambiguity. The next step is to define the canonical geometrization of our refinement, which will be completely determined by the geometrization of the original partition. But first, let us introduce proper notation for the vertical and horizontal order in a geometrized rectangle.

\begin{defi}\label{Defi: Notacion orden en los rectangulos}
Let $(f, \mathcal{R})$ be a geometric Markov partition of the \textbf{p-A} homeomorphism $f: S \to S$. Let $H_1, H_2 \subset R_i \in \mathcal{R}$ be horizontal sub-rectangles of $R_i \in \mathcal{R}$. To denote that, with respect to the vertical direction of $R_i$, the sub-rectangle $H_{1}$ is below $H_{2}$, we use the notation $H_{(i,s)} <_{v(i)} H_{(i,s')}$. 

Similarly, if $V_1, V_2 \subset R_i$ are vertical sub-rectangles of $R_i \in \mathcal{R}$, the notation $V_{1} <_{h(i)} V_{2}$ indicates that the vertical sub-rectangle $V_{1}$ is to the left of $V_{2}$ with respect to the horizontal direction of $R_i$.
\end{defi}

We must introduce a way to order the rectangles in the refinement; we will use the lexicographic order defined over the set of labels of the refinement.

\begin{defi}\label{Defi: Orden Lexicodromico}
Let  $(f,\cR')$ be  a refinement by horizontal sub-rectangles of the pair $(f,\cR)$. The  \emph{lexicographic order} of the horizontal labels of the refinement $(f,\cR')$ is the bijective map $\underline{r}: \mathcal{H}((f,\mathcal{R}'))\rightarrow \{1,\cdots, \alpha(f,\cR'):=\sum_{i=1}^{n}H(i) \}$ given by:
	\begin{equation}\label{Equa: Orden lexicodromico}
		\underline{r}(i,h)=\sum_{i'<i} H(i') + h.
	\end{equation}
\end{defi}

Now we are ready to define a geometrization in our refinements. Once again, this definition and these conditions are formulated in terms of horizontal sub-rectangles, but the situation with vertical sub-rectangles is entirely symmetric to what is explicitly presented below.

\begin{defi}\label{Defi: Geometric refinement}
	Let $(f, \mathcal{R}')$ be a horizontal refinement of $(f, \mathcal{R})$ by horizontal sub-rectangles, and let $\textbf{H} = \{H_{(i, h)} \mid i \in \{1, \dots, n\} \text{ and } h \in \{1, \dots, H(i)\} \text{ for } H(i) \in \mathbb{N}\}\}$ be the set of horizontal labels of the refinement $(f, \mathcal{R}')$. We say $(f, \mathcal{R}')$ is a \emph{geometric refinement} of the geometric Markov partition $(f, \mathcal{R})$ if, for all $(i,h), (i,h') \in \mathcal{H}(f, \mathcal{R}')$, $H_{(i, h)} <_{v(i)} H_{(i, h')}$ if and only if $h < h'$. If this is the case, we endow the Markov partition $\mathcal{R}'$ with the \emph{canonical geometrization}, described as follows:
	\begin{itemize}
		\item The rectangles in $\mathcal{R}'$ are labeled with the \emph{lexicographic order} (\ref{Defi: Orden Lexicodromico}) on their labels, i.e., we set $H_{(i, s)} := H_{\underline{r}(i, s)}$. The set of labeled rectangles is denoted by $\mathcal{R}_{\textbf{H}}$.
		\item If $H_{(i, s)} \subset R_i$, then we assign to $H(i, s)$ the vertical and horizontal directions that coincide with the vertical and horizontal directions of $R_i$ when restricted to $H_{(i, s)}$. We call these the \emph{induced directions} by $(f, \mathcal{R})$.
	\end{itemize}
	The geometric Markov partition $(f, \mathcal{R}')$ with the canonical geometrization is called a \emph{geometric refinement} of $(f, \mathcal{R})$. Its geometric type is denoted by $T((f, \mathcal{R}'))$.
\end{defi}

\begin{figure}[ht]
	\centering
	\includegraphics[width=0.4\textwidth]{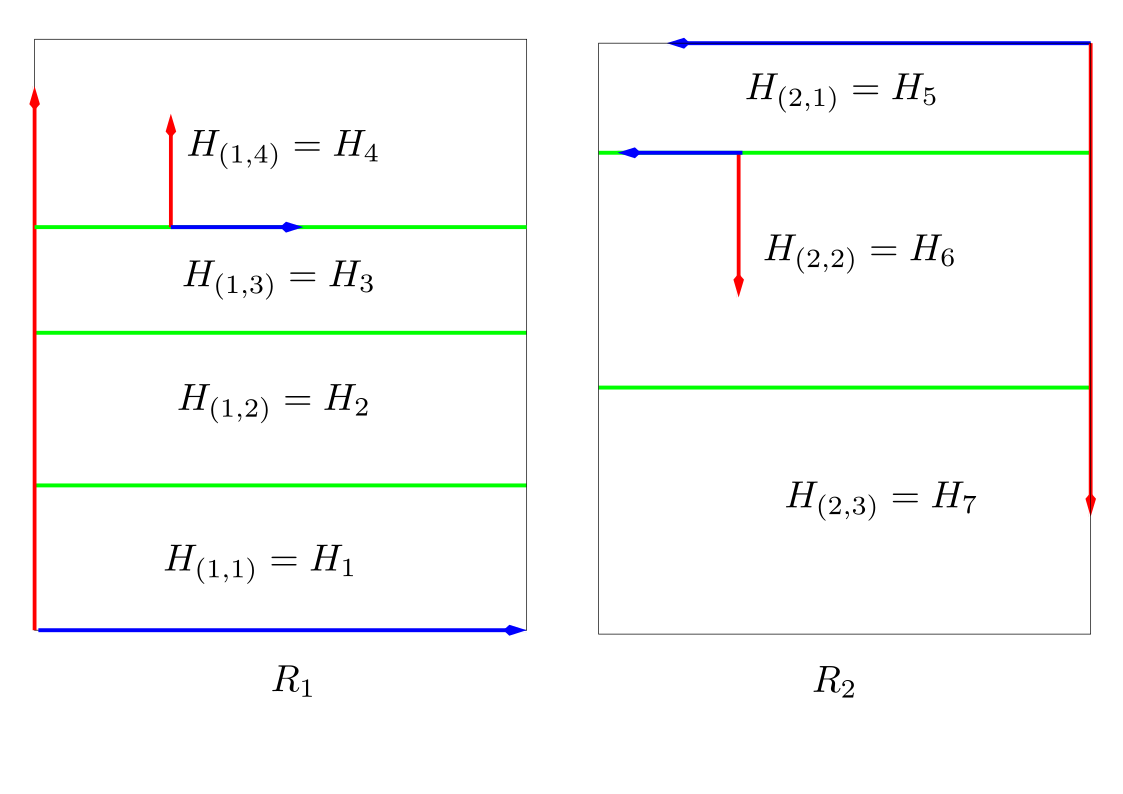}
	\caption{Labeling and Orientation in the Geometric Refinement of $\mathcal{R} = \{R_1, R_2\}$}
	\label{Fig: Geometric refinament}
\end{figure}

Fig. \ref{Fig: Geometric refinament} illustrates how the lexicographic order and orientations are fixed for a refinement by horizontal sub-rectangles of a certain Markov partition that has just two rectangles, $\mathcal{R} = \{R_1, R_2\}$.

%% file: Refinamientos/SecBinaryref.tex
\section{The Binary refinement}\label{Subsec: Refinamiento Binario}

Let $(f, \mathcal{R})$ be a pair consisting of a \textbf{p-A} homeomorphism and a geometric Markov partition. According to \cite[Exposition 10]{fathi2021thurston}, the incidence matrix of $(f, \mathcal{R})$, denoted by $A(f, \mathcal{R})$, has a coefficient $a_{ij}$ equal to $1$ if $f(\overset{o}{R_i}) \cap \overset{o}{R_k} \neq \emptyset$ and $0$ otherwise. Such definition  is in the spirit of R. Bowen's approach in \cite{BowenAnosovbook} for is an Axiom A homeomorphisms. However, this definition does not take into account the number of times the image of $R_i$ intersects $R_k$, which is crucial for our purposes. Therefore, we adopt the following definition.

\begin{defi}\label{Defi: Incidence matrix}
	Let $f$ be a \textbf{p-A} homeomorphism, and let $\mathcal{R}$ be a Markov partition of $f$. The incidence matrix of the pair $(f, \mathcal{R})$, denoted by $A(f, \mathcal{R}) = (a_{ik})$, has coefficients $a_{ik}$ equal to the number of horizontal sub-rectangles of $(f, \mathcal{R})$ contained in $R_i$ that are mapped by the action of $f$ into vertical sub-rectangles of $R_k$.
\end{defi}

\begin{defi}\label{Defi: Matriz binary, positive}
	A square matrix $A = (a_{ij})$ of size $n \times n$ is \emph{non-negative} if for all $1 \leq i, j \leq n$, $a_{i,j} \in \mathbb{N}_0$. It is \emph{positive definite} if for all $1 \leq i, j \leq n$, $a_{i,j} \in \mathbb{N}_+$. A matrix whose coefficients are $0$ or $1$ is called a \emph{binary matrix}. Finally, we recall that if $n \in \mathbb{N}$, the notation for the coefficients of the $n$-th power of $A$ is $A^n = (a_{i,j}^{(n)})$.
	Finally a non-negative matrix $A$ is mixing if there exist $N\in \NN_+$ such that $A^N$ is positive defined.
\end{defi}

The theorem to be proven in this subsection is as follows.

\begin{theo}\label{Theo: Refinamiento binario}
	Given a geometric partition $\mathcal{R}$ of $f$ and the geometric type $T$ of the pair $(f, \mathcal{R})$, there exists a finite algorithm to construct a geometric refinement of the pair $(f, \mathcal{R})$, known as the \emph{binary refinement} of $(f, \mathcal{R})$, denoted by $(f, \textbf{Bin}(\mathcal{R}))$ with the following properties:
	\begin{itemize}
		\item The  rectangles in $\textbf{Bin}(\cR)$ that are the horizontal sub-rectangles of the Markov partition $(f, \mathcal{R})$.
		\item The incidence matrix of $(f, \textbf{Bin}(\cR) )$ is binary.
		\item  there are explicit formulas in terms of the geometric type $T$ of $(f, \mathcal{R})$ to compute the geometric type of  $(f, \textbf{Bin}(\mathcal{R}))$.
	\end{itemize} 
\end{theo}

The proof is constructive and will be carried out through a series of lemmas and constructions.

\begin{lemm}\label{Lemm: Binary partition}
	Then the family $H(f,\cR) := \{H^i_j\}_{(i,j)\in \cH(T)}$ of horizontal sub-rectangles of $(f, \cR)$ is a Markov partition of $f$.
\end{lemm}

To gain some intuition about the proof see Figure \ref{Fig: Bin refinament}.

\begin{proof}
	Let us suppose that $f: S \rightarrow S$ and $\mathcal{R} = \{R_i\}_{i=1}^n$. We will then fix the notation of the geometric type $T$ as follows:
	
	$$
	T(f,\cR):=T=(n,\{(h_i,v_i)\}_{i=1}^n, \Phi_T=(\rho_T,\epsilon_T)).
	$$
	Clearly, $\bigcup \{ H_{j}^i \mid (i, j) \in \mathcal{H} \} = S$, as it is equal to the union of rectangles in $\mathcal{R}$. Two different horizontal sub-rectangles of the same $R_i \in \mathcal{R}$ have disjoint interiors, and this is also true for horizontal sub-rectangles from different rectangles in $\mathcal{R}$. Therefore, the elements in $H(f, \mathcal{R})$ have disjoint interiors, and $H(f, \mathcal{R})$ is a partition of $S$ into rectangles.
	
	The unstable boundary of $H(f,\cR)$, denoted by $\partial^u H(f,\cR)=\cup_{(i,j)\in \cH} \partial^u H^i_j$, coincides with the unstable boundary of the Markov partition $\cR$, so it is $f^{-1}$-invariant. The stable boundary of $H(f,\cR)$ is  the union of the stable boundaries of the horizontal sub-rectangles; $\partial^s H(f,\cR)=\cup_{(i,j)\in \cH(T)} \partial^s H^i_j$.  We can suppose that $f(H^i_j)=V^k_l$, since $V^k_l$ is a vertical sub-rectangle of $R_k\in \cR$,  $\partial^{s} V^k_l \subset \partial^s R_k$ and then $f(\partial^s H^i_j) \subset \partial^s R_k$. Hence the stable boundary of $H(f,\cR)$ is $f$-invariant.
	
	Since the family $H(f,\cR)$ satisfies the properties of Proposition \ref{Prop: Markov criterion boundary}, it is a Markov partition of $f$.
	
\end{proof}

\begin{figure}[h]
	\centering
	\includegraphics[width=0.6\textwidth]{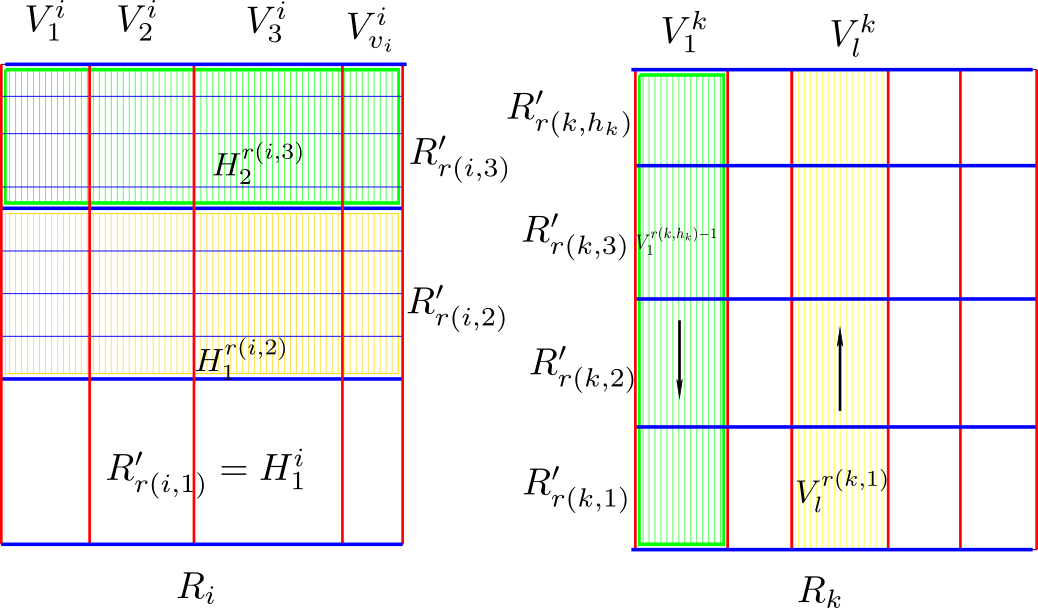}
	\caption{The binary refinement}
	\label{Fig: Bin refinament}
\end{figure}

It should not be difficult for the reader to observe that the set of horizontal labels of $(f, \mathcal{R})$ and the labeling of the binary refinement are the same. If we label them as $H^i_j = H_{(i, j)}$, the labels are ordered in such a manner that for all $(i, j), (i, j') \in \mathcal{H}(f, \textbf{Bin}(\mathcal{R}))$, $H_{(i, j)} = H^i_j <_{v(i)} H^i_{j'} = H_{(i, j')}$ whenever $j > j'$. The other properties necessary for a geometric refinement are easily verifiable; therefore, $(f, \textbf{Bin}(\mathcal{R}))$ is a geometric refinement of $\mathcal{R}$.


\begin{prop}\label{Prop: Refinamiento binario propiedades}
	Let $(f, \mathcal{R}_f)$ and $(g, \mathcal{R}_g)$ be two pairs that realize the geometric type $T$. Then they satisfy the following properties:
	\begin{itemize}
		\item[i)] The incidence matrix of $(f, \textbf{Bin}(\mathcal{R}_f))$ is binary.
		\item[ii)] The geometric type of $(f, \textbf{Bin}(\mathcal{R}_f))$ is uniquely determined by $T$, as it is computed through an algorithm that only uses the information contained in $T$.
		\item[iii)] The binary refinements $\textbf{Bin}(f, \mathcal{R}_f)$ and $\textbf{Bin}(g, \mathcal{R}_g)$ have the same geometric type.
	\end{itemize}
	The geometric type of the binary refinement of the pair $(f, \mathcal{R})$ that realizes $T$ is denoted by $\textbf{Bin}(T) := T(f, \textbf{Bin}(\mathcal{R}))$.
\end{prop}

\begin{proof}
	
	\textbf{Item} $i)$.	For every $(i,j), (k,j') \in \cH(T)$, the number of connected components in the intersection $f(H_{r(i,j)}) \cap H_{r(k,j')}$ is equal to the coefficient $a_{r(i,j)r(k,j')}$ of the incidence matrix $A:=A(\textbf{Bin}(f,R))$ of the binary refinement, an we applied the following computation:
	\begin{equation}\label{Equ: incidence matrice}
		f(H_{r(i,j)}) \cap H_{r(k,j')}= f(\overset{o}{ H^i_j})\cap \overset{o}{H^k_{j'}}  =
		\overset{o}{ V^{k}_{l} }\cap \overset{o}{H^k_{j'}}
	\end{equation}
	
	By definition and the left side of Equation  \ref{Equ: incidence matrice}, the coefficient $a_{r(i,j)r(k,j')}$ in $A$ is equal to the number of connected components that result from the intersection of the interiors of a horizontal sub-rectangle and a vertical sub-rectangle of $R_k$.
	
	With this interpretation in mind, if $k \neq k'$, the interiors of the rectangles $R_k$ and $R_{k'}$ are disjoint, and thus $a_{r(i,j),r(k,j')} = 0$. Since the interior of $R_k$ is an embedded (open) rectangle, if $k = k'$, the intersection of the interior of a horizontal sub-rectangle and a vertical sub-rectangle of $R_k$ has only one connected component. Therefore, $a_{r(i,j)r(k,j')} = 1$. This proves our item.
	
	\textbf{Item} $ii)$.		Let's deduce the geometric type of  $\textbf{Bin}(f, \cR_f)$ using only the information from $T$, the chosen orientation, and the lexicographic order. Let's denote the resulting geometric type as 
	$$
	T((f, \textbf{Bin}(\cR))=\{n',\{(h_i',v_i')\}_{i'=1}^{n'}, \Phi_{T'}=(\rho',\epsilon')\}.
	$$
	We need to determine all the parameters. The number of elements in $\cH(T)$ determines the number $n'$, which can be calculated using the formula:

	\begin{equation}\label{Equ: number rectangles in horizontal ref}
		n'=\sum_{i=1}^{n}h_i=\sum_{i=1}^n v_i=\alpha(T).
	\end{equation}
	
	Just remember that the number of elements in $\cH(T)$ corresponds to the number of horizontal sub-rectangles in $(f,\cR)$.
	
	Let $(i,j)\in \cH(T)$ and suppose $\Phi_T(i,j)=((k,l), \epsilon_T(i,j))$ for the rest of the proof. The pair $(f,H(f,\cR))$ denotes the horizontal refinement $H(f,\cR)$ viewed as a Markov partition of $f$.
	
	Based on our previous analysis, the vertical sub-rectangles of the rectangles $H^{r(i,j)}$ in $(f,H(f,\cR))$ correspond to the connected components of the intersection  of the interiors of all the vertical sub-rectangles $V^i_l$ of $R_i$ with $H^i_j$. The incidence matrix determines that all vertical sub-rectangles of $R_i$ intersect every horizontal sub-rectangle of $R_i$ exactly once. Therefore, the number of vertical sub-rectangles of $H_{r(i,j)}$ is constant with respect to $i$. We can infer that the number of vertical sub-rectangles of $H_{r(i,j)}$ is equal to $v_i$, i.e.
	
	\begin{equation}\label{Equ: number vertical sub in the horizontal ref}
		v'_{r(i,j)}=v_{i}
	\end{equation}
	
	Moreover, these vertical sub-rectangles are ordered from left to right in a coherent way with the horizontal orientation of $R_i$ by 
	
	\begin{equation}\label{Equ: Vertical order of horizontal ype}
		\{V^{r(i,j)}_{l}\}_{l=1}^{v_i}.
	\end{equation}

	The number of horizontal sub-rectangles of $H_{r(i,j)}$, denoted as $h'_{r(i,j)}$, is equal to $h_{k}$ because  $f(H_{r(i, j)})=f(H^{i}_{j})=V^{k}_{l}$  intersects exactly $h_{k}$ distinct horizontal sub-rectangles of $R_k$ and no other horizontal sub-rectangles of $(f,\cR)$. These horizontal sub-rectangles are ordered in increasing order according to the vertical orientation of $H_{r(i,j)}$, which is inherited from $R_i$, in the following manner:
	
	\begin{equation}\label{Equa: Horizontal rec of the horizontal ref}
		\{H^{r(i,j)}_{j'}\}_{j'=1}^{h_k}.
	\end{equation}

	We proceed to determine $\rho'$ and $\epsilon'$. Recall that we assumed $\Phi_T(i,j)=((k,l), \epsilon_T(i,j))$. To establish $\rho'$, we need to consider the change of orientation in $f(H^k_j)$ as given by $\epsilon_T(i,j)$. We are going to split the computations based on the two cases.

	First, assume $\epsilon_T(i,j)=1$ and $j_0\in \{1,\cdots, h_{k}\}$, representing the horizontal sub-rectangle $H^{r(i,j)}_{j_0}$ of $H_{r(i,j)}$. Since $f(H_{r(i,j)})=V^k_l$ preserves the vertical orientation, the horizontal sub-rectangle of $R_k$ that intersects $f(H_{r(i,j)})=V^k_l$ at position $j_0$ with respect to the vertical orientation of $R_k$ is labeled by $r(k,j_0)$. Also, $f(H^{r(i,j)}_{j_0})$ corresponds to the vertical sub-rectangle of $R_k$ that is at position $l$, so $f(H^{r(i,j)}_{j_0})$ is the vertical sub-rectangle $V^{r(k,j_0)}_l$. We can express this construction in terms of the geometric type using the formula:

	\begin{equation}\label{Equa: horizontal type for preseving orientation }
		\Phi_{T'}(r(i,j),j_0):=(\rho'(r(i,j),j_0),\epsilon'(r(i,j),j_0)=(r(k,j_0),l,\epsilon_T(i,j)).
	\end{equation}
	
	In the case of $\epsilon_T(i,j)=-1$, $f$ changes the vertical orientation of $H^i_j$ with respect to $V^k_l=f(H^i_j)$. This implies that the horizontal sub-rectangle of $R_k$ containing the image of $H^{r(i,j)}_{j_0}$ is located at position $j_0$ with the inverse vertical orientation of $R_k$, which corresponds to position $(h_k-(j_0-1))$ with the prescribed orientation in $R_k$. Therefore, the horizontal sub-rectangle of $R_k$ corresponding to $f(H^{r(i,j)}_{j_0})$ is given by:
	
	$$
	H^k_{h_k-(j_0-1)}=H_{r(k,h_k-(j_0-1))}\in H(f,\cR)
	$$
	
	The vertical sub-rectangle of $(f,H_{r(k,h_k-(j_0-1))})$ that contains $f(H^{r(i,j)}_{j_0})$ is a subset of $V^k_l$, so it is located at position $l$ with respect to the horizontal orientation of $R_k$. We can conclude that:
	
	$$
	f(H^{r(i,j)}_{j_0})=V^{r(k,h{}-(j_0-1))}_l.
	$$
	
	The change in the vertical direction of the sub-rectangle $H^{r(i,j)}{j_0}$ is quantified by $\epsilon'(r(i,j),j_0)$. However, the change in the vertical direction of $H^{r(i,j)}{j_0}$ is determined by the change in the orientation of $f$ restricted to $H^i_j$. This information is governed by $\epsilon_T(i,j)$ in the geometric type $T$, and we have $\epsilon'(r(i,j),j_0)=\epsilon_T(i,j)$. In terms of $T$, we have the formula

	\begin{equation}\label{Equa: horizontal type for change orientation }
		\Phi_{T'}(r(i,j),j_0)=(r(k,h_{k}) -(j_0-1),l,\epsilon_T(i^,j_0)).
	\end{equation}

	The equations \ref{Equa: horizontal type for preseving orientation } and  \ref{Equa: horizontal type for change orientation } are determined by $T$, and their computation is algorithmic. This proves  Item $ii)$.
	
	\textbf{Item} $iii)$	The binary refinements $\textbf{Bin}(f, \cR_f)$ and $\textbf{Bin}(g, \cR_g)$ yield the same equation and computation than in Equation:  (\ref{Equ: number rectangles in horizontal ref}), 	(\ref{Equ: number vertical sub in the horizontal ref}), (\ref{Equ: Vertical order of horizontal ype}), (\ref{Equa: Horizontal rec of the horizontal ref}), (\ref{Equa: horizontal type for preseving orientation }) and (\ref{Equa: horizontal type for change orientation }). Therefore, they result in the same geometric type, $\textbf{Bin}(T):=T(\textbf{Bin}(f,\cR_f))=T(\textbf{Bin}(g,\cR_g))$. This proves the last item.

\end{proof}

Theorem \ref{Theo: Refinamiento binario} is a direct consequence Proposition \ref{Prop: Refinamiento binario propiedades}  as it is the following corollary.

\begin{coro}\label{Coro: horizontal type is pseudo-Anosov}
	Let $T \in \mathcal{G}\mathcal{T}(\textbf{p-A})$ be a geometric type in the pseudo-Anosov class. Then, the binary refinement of every realization $(f, \mathcal{R})$ has geometric type $\textbf{Bin}(T)$ within the pseudo-Anosov class.
\end{coro}

%% file: Refinamientos/Seccoderef.tex
\section{The horizontal refinements along a family of periodic codes.}\label{Subsec: Codigos periodicos ref}

\subsection{Symbolic Dynamics of Geometric Markov Partitions.}\label{subsub: Simbolic dinamics}

The emphasis of the refinement that we are about to present is on its symbolic nature. For this reason, we recommend that the reader consult \cite[Chapter 1]{KitchensSymDym} and \cite{MarcusLindSynbolic} for a nice introduction to the theory of symbolic dynamical systems, as some of the results and definitions that we will use are carefully presented in these books. Lets star by properly pose the problem.

Let $\Sigma := \prod_{\mathbb{Z}} \{1, \cdots, n\}$ be the set of \emph{bi-infinite} sequences in $n$-digits. We call the elements of $\Sigma$ \emph{codes}, denoted $\underline{w} = (w_z)_{z \in \mathbb{Z}}$. The set $\Sigma$ is endowed with the topology induced by the metric $d_{\Sigma}(\underline{x}, \underline{y}) := \sum_{z \in \mathbb{Z}} \frac{\delta(x_z, y_z)}{2^{|z|}}$, where $\delta(x_z, y_z) = 0$ if $x_z = y_z$ and $\delta(x_z, y_z) = 1$ otherwise. 

The \emph{shift map} on $\Sigma$ is the homeomorphism $\sigma: \Sigma \rightarrow \Sigma$ given by $(\sigma(\underline{w}))_z = w_{z+1}$ for all $z \in \mathbb{Z}$ and any code $\underline{w} \in \Sigma$. The \emph{full shift} (in $n$ symbols) is the dynamical system $(\Sigma, \sigma)$. 

If $A \in \mathcal{M}_{n \times n}(\{0,1\})$ is a binary and mixing matrix, the subspace $\Sigma_A := \{ \underline{w} = (w_z)_{z \in \mathbb{Z}} \in \Sigma : \forall z \in \mathbb{Z}, (a_{w_z, w_{z+1}}) = 1 \}$ is a compact and $\sigma$-invariant subset of $\Sigma$. Assume that $A \in \mathcal{M}_{n \times n}(\{0,1\})$ is a binary and mixing matrix, then the \emph{sub-shift of finite type} associated with $A$ is the dynamical system $(\Sigma_A, \sigma_A)$, where $\sigma_A := \sigma|_{\Sigma_{A}}$. According to \cite[Lemma 10.21]{fathi2021thurston}, the incidence matrix of any \textbf{p-A} homeomorphism is mixing. Since we have constructed the binary refinement, we can assume the incidence matrix is binary. We pose the following definition, where $\mathcal{G}\mathcal{T}(\textbf{pA})^{SIM}$ is not empty by Corollary \ref{Coro: horizontal type is pseudo-Anosov}.

\begin{defi}\label{Defi:Induced shift}
The subset of geometric types in the pseudo-Anosov class whose incidence matrices are binary is denoted as $\mathcal{G}\mathcal{T}(\textbf{p-A})^{SIM} \subset \mathcal{G}\mathcal{T}(\textbf{p-A})$, and these are called \emph{symbolically modelable} geometric types. Let $T \in \mathcal{G}\mathcal{T}(\textbf{p-A})^{SIM}$. The \emph{sub-shift induced} by the geometric type $T$ is given by the sub-shift of finite type $(\Sigma_{A(T)}, \sigma_{A(T)})$, where $A(T)$ is the incidence matrix of $T$.
\end{defi}

The induced sub-shift $(\Sigma_{A(T)}, \sigma_{A(T)})$ and the homeomorphism  $f:S\to S$ are related through the \emph{projection} $\pi_{(f,\mathcal{R})}:\Sigma_{A(T)}\to S$ introduced below, such map takes a code $\underline{w} = (x_z)_{z \in \mathbb{Z}} \in \Sigma_{A}$ and sends it to a point $x \in S$ whose orbit follows the itinerary imposed by $\underline{w}$; for all $z \in \mathbb{Z}$, $f^z(x) \in R_{w_z}$. Clearly such projection strongly depends on Markov partition $(f, \mathcal{R})$. 

\begin{defi}\label{Defi: projection pi}
	Let $T$ be a symbolically modelable geometric type. The \emph{projection} with respect to $(f,\cR)$, denoted by $\pi_{(f,\cR)}: \Sigma_{A(T)} \rightarrow S$, is defined by the relation:
	\begin{equation}
		\pi_{(f,\cR)}(\underline{w}) = \bigcap_{n \in \NN} \overline{\bigcap_{z=-n}^{n} f^{-z}(\overset{o}{R_{w_z}})},
	\end{equation}
	for every code $\underline{w} = (w_z)_{z \in \ZZ} \in \Sigma_{A(T)}$.	
\end{defi}

Our definition \ref{Defi: projection pi} has a slight difference from the one given in \cite[Lemma 10.16]{fathi2021thurston}, where the authors define, $\phi_{(f,\cR)}(\underline{w}) := \bigcap_{z\in \mathbb{Z}} f^{-z}(R_{w_z})$, with the disadvantage of requiring the rectangles in the Markov partition to be embedded in order for the map to be well-defined. However, with our definition, the intersections of the closures $\bigcap_{z=-n}^{n} f^{-z}(\overset{o}{R_{w_z}})$ are rectangles whose diameters tend to $0$, so for sufficiently large $n$, such an intersection is an embedded rectangle, and we don't need this assumption. Proposition \ref{Prop:proyecion semiconjugacion} contains the main properties of the projection $\phi_{(f,\cR)}$ and it correspond to \cite[Lemma 10.16]{fathi2021thurston}, where the reader can follow the proof to confirm that $\pi_{(f,\mathcal{R})}$, as introduced by us, is well-defined, continuous, and a semi-conjugation; therefore, we do not elaborate further on this part. However, the \emph{finite-to-one property} is fundamental for our further study (especially in \cite{CruzDiaz2024}), and we must introduce a characterization of the codes in the fiber $\pi_{(f,\mathcal{R})}^{-1}(x)$, which we call \emph{sector-codes} of $x$, to then establish that $\pi_{(f,\mathcal{R})}^{-1}(x)$ is finite.

\begin{prop}\label{Prop:proyecion semiconjugacion}
	If $\cR$ is the Markov partition of a pseudo-Anosov homeomorphism $f: S \rightarrow S$, whose incidence matrix $A(\cR)$ is binary, then the projection $\pi_{(f,\cR)}: \Sigma_{A(T)} \rightarrow S$ is a continuous, surjective, and finite-to-one map. Furthermore, $\pi_{(f,\cR)}$ semi-conjugates $f$ with $\sigma_{A(T)}$, i.e., $f \circ \pi_{(f,\cR)} = \pi_{(f,\cR)} \circ \sigma_{A(T)}$.
\end{prop}

Let us recall that the sector of a point was introduced in \ref{Defi: Sector}. We prove in \ref{Prop: image secto is a sector} that the image of a sector is a sector, and in \ref{Lemm: sector contined unique rectangle} that every sector is contained in a unique rectangle of a Markov partition. Therefore, the next definition is valid.

\begin{defi}\label{Defi: Sector codes}
	Let $\cR=\{R_i\}_{i=1}^n$ be a Markov partition whose incidence matrix is binary. Let $x \in S$ be a point with sectors $\{e_1(x),\cdots, e_{2k}(x)\}$ (where $k$ is the number of stable or unstable separatrices in $x$).
	The \emph{sector code} of $e_j(x)$ is the sequence $\underline{e_j(x)}=(e(x,j)_z)_{z\in \ZZ} \in \Sigma$, given by the rule: $e(x,j)_z:=i$, where $i\in \{1,\dots,n\}$ is the index of the unique rectangle in $\cR$ such that the sector $f^z(e_j(x))$ is contained in the rectangle $R_i$.
\end{defi}

Any code $\underline{w}\in\Sigma$ that is equal to the sector code of $e_j(x)$ (for some $j$ ) is called a sector code of $x$. The space $\Sigma$ of bi-infinite sequences is larger than $\Sigma_A$. So, we need to show that every sector code is, in fact, an \emph{admissible code}, i.e., that $\underline{e_j(x)}\in \Sigma_A$.

\begin{lemm}\label{Lemm: sector code is admisible}
	For every $x \in S$, every sector code $\underline{e}:=\underline{e_j(x)}$ is an element of $\Sigma_A$.
\end{lemm}

\begin{proof}
	Let $A=(a_{ij})$ the incidence matrix of $T$.	The code $\underline{e}=(e_z)_{z\in \ZZ}$ is in $\Sigma_A$ if and only if for all $z\in \ZZ$, $a_{e_z e_{z+1}}=1$. By definition, this happens if and only if $f(\overset{o}{R_{e_z}})\cap \overset{o}{R_{e_{z+1}}}\neq \emptyset$.

	Let $\{x_n\}_{n\in \NN}$ be a sequence converging to $f^z(x)$ and contained in the sector $f^z(e)$. By Lemma \ref{Lemm: sector contined unique rectangle} the sector $f^z(e)$ is contained in a unique rectangle $R_{e_z}$, and we can assume $\{x_n\}\subset R_{e_z}$. Moreover, there exists $N\in \NN$ such that $x_n \in \overset{o}{R_{e_z}}$ for all $n\geq N$. Remember that the sector $f^z(e)$ is bounded by two consecutive local stable and unstable separatrices of $f^z(x)$: $F^s(f^z(x))$ and $F^u(f^z(x))$. If for every $n\in \NN$, $x_n$ is contained in the boundary of $R_{e_z}$, this boundary component is a local separatrice of $x$ between $F^s(f^z(x))$ and $F^u(f^z(x))$, which is not possible.

	Since the image of a sector is a sector, the sequence $\{f(x_n)\}$ converges to $f^{z+1}(x)$ and is contained in the sector $f^{z+1}(e)$. The argument in the last paragraphs also applies to this sequence, and $f(x_n)\in \overset{o}{R_{e_{z+1}}}$ for $n$ big enough. This proves that $f(\overset{o}{R_{e_z}})\cap \overset{o}{R_{e_{z+1}}}$ is not empty.
	
\end{proof}

The sector codes of a point $x$ are not only admissible they are, in fact, the only codes in $\Sigma_A$ that project to $x$.  

\begin{lemm}\label{Lemm: every code is sector code }
	If the code $\underline{w} = (w_z) \in \Sigma_A$ is mapped to the point $x \in S$ by the action of the projection $\pi_{(f,\cR)}$, then the code $\underline{w}$ is equal to a sector code of the point $x$.
\end{lemm}

\begin{proof}
	
	For each $n \in \NN$, we take the rectangle $F_n=\cap_{j=-n}^n f^{-j}(\overset{o}{R_{w_j}})$, which is non-empty because $\underline{w}$ is in $\Sigma_A$. The following properties hold:

	\begin{itemize}
		\item[i)] $\pi_f(\underline{w})=x \in \overline{F_n}$ for every $n\in \ZZ$.
		\item[ii)] For all $n\in \NN$, $F_{n+1}\subset F_n$.
		\item[iii)]	For every $n\in \NN$, there exists at least one sector $e$ of $x$ contained in $F_n$. If this does not occur, there is $\epsilon>0$ such that the regular neighborhood of size $\epsilon$ around $x$, given by Theorem \ref{Theo: Regular neighborhood}, is disjoint from $F_n$. However, $x\in F_n$ by item i).
		
		\item[iv)] If the sector $e \subset H_n$, then for every $m\in  \ZZ$ such that $\vert m\vert \leq n$:
		$$
		f^m(e) \subset f^m(H_n)= \cap_{j=m-n}^{m+n}f^{-j}(\overset{o}{R_{w_{m-j}}}) \subset \overset{o}{R_m}.
		$$ 
		which implies that $e_m=w_m$ for all $m\in \{-n,\cdots,n\}$.
	\end{itemize}
	
	By item $ii)$, if a sector $e$ is not in $F_n$, then $e$ is not in $F_{n+1}$. Together with the fact that for all $n$ there is always a sector in $F_n$ (item $iii)$), we deduce that there is at least one sector $e$ of $x$ that is in $F_n$ for all $n$. Then, we apply point $iv)$ to deduce $e_z=w_z$ for all $z$.
\end{proof}
Let $x\in S$ be a point with $k$ stable separatrices. Then $x$ has at most $2k$ sector codes projecting to $x$, and we obtain the following corollary, which proves the surjectivity and finite-to-one property ensured by Proposition \ref{Prop:proyecion semiconjugacion}.

\begin{coro}\label{Coro: Caracterisation fibers}
	For all $x \in S$, if $x$ has $k$ separatrices, then $\pi_{(f,\cR)}^{-1}(x) = \{\underline{e_j(x)}\}_{j=1}^{2k}$. In particular, $\pi_{(f,\cR)}$ is finite-to-one and surjective.
\end{coro}

\subsubsection{The relative position and behavior of code projections.}\label{subsub: posicion relativa proyecion}

For the rest of this section we make $A:=A(T)$ and we fix a finite family of periodic codes contained in $\Sigma_{A}$:
$$
\cW=\{\underline{w}^1,\cdots,\underline{w}^Q\}\in \Sigma_A.
$$

We will construct a geometric refinement of the pair $(f,\cR)$ by cutting the rectangles along the orbit of certain stable intervals determined by the projection of the codes in the family $\cW$. For this reason, we must introduce some new concepts and prove some lemmas that provide insight into the position and behavior of the codes under the projection $\pi_{(f,\mathcal{R})}$. The following lemma is well known in the theory of symbolic dynamical systems and is stated without proof (but it can be read in \cite{IntiThesis}).

\begin{lemm}\label{Lemm: Periodic to peridic}
Let $T \in \mathcal{G}\mathcal{T}(\textbf{p-A})^{SIM}$ and let $(f, \mathcal{R})$ be a realization of $T$. If $\underline{w} \in \text{Per}(\sigma)$ is a periodic code, then $\pi_f(\underline{w})$ is a periodic point of $f$. Moreover, if $p$ is a periodic point for $f$, then $\pi^{-1}_f(p) \subset \text{Per}(\sigma)$, meaning that all codes projecting onto $p$ are periodic.
\end{lemm}

The following Lemma characterizes the periodic boundary points of $(f,\cR)$ in terms of their iterations under $f$. It is a technical result that we will use in future arguments.

\begin{lemm} \label{Lemm: no periodic boundary points}
	Let $\underline{w} \in \Sigma_A$ be a code such that for every $k \in \mathbb{Z}$, $f^k(\pi_f(\underline{w})) \in \partial^s\mathcal{R}$, then $\pi_f(\underline{w})$ is a periodic point of $f$. Similarly, if $f^k(\pi_f(\underline{w})) \in \partial^u\mathcal{R}$ for all integers $k$, then $\underline{w}$ is a periodic code.
\end{lemm}

\begin{proof}
	
	Suppose $\underline{w}$ is non-periodic. Lemma \ref{Lemm: Periodic to peridic}  implies that $x = \pi_f(\underline{w})$ is non-periodic and lies on the stable boundary of $\mathcal{R}$, which is a compact set with each component having finite $\mu^u$ length. Take $[x,p]^s$ to be the stable segment joining $x$ with the periodic point on its leaf (which exists by Lemma \ref{Lemm: Boundary of Markov partition is periodic}). For all $m \in \mathbb{N}$,
	$$
	\mu^u([f^{-m}(x), f^{-m}(p)]^s)=\mu^u(f^{-m}[x,p]^s)=\lambda^m\mu^u([x,p]^s),
	$$
	which is unbounded. Therefore, there exists $m \in \mathbb{N}$ such that $f^{-m}(x)$ is no longer in $\partial^s\mathcal{R}$, contradicting the hypothesis of our lemma. A similar reasoning applies to the second proposition.
\end{proof}

\begin{coro}\label{Coro: interior periodic points unique code}
	Let $\underline{w} \in \Sigma_A$ be a periodic code. If $\pi_{(f,\cR)}(\underline{w}) \in \overset{o}{\cR}$, then for all $z \in \ZZ$, $f^z(\pi_{(f,\cR)}(\underline{w})) \in \overset{o}{\cR}$. Moreover, $\pi_{(f,\cR)}^{-1}(\pi_{(f,\cR)}(\underline{w})) = \{\underline{w}\}$, so it has a unique point.
\end{coro}

\begin{proof}
	If, for some $z \in \ZZ$, $f^z(x)$ is in the stable (or unstable) boundary of the Markov partition, then, since $x$ is periodic and the stable boundary of a Markov partition is $f$-invariant, the entire orbit of $x$ would lie in the stable boundary of $\cR$. This leads to a contradiction because at least one point in the orbit of $x$ must be in the interior of $\cR$.
	
	Thus, all the sectors of $f^z(x)$ are contained within the interior of the same rectangle, and therefore all its sector codes are the same.
\end{proof}

\begin{defi}\label{Defi: totally interior points}
	Let $(f,\cR)$ be a realization of a symbolically modelable geometric type $T \in \cG\cT(\textbf{p-A})^{SIM}$. A point $x \in S$ is a \emph{totally interior point} of $(f,\cR)$ if, for all $z \in \ZZ$, $f^z(x) \in \overset{o}{\cR}$. The set of totally interior points is denoted by \textbf{Int}$(f,\cR) \subset S$.
\end{defi}

\begin{prop}\label{Prop: Carterization injectivity of pif}
	Let $x$ be any point in  $S$. Then $\vert \pi_{(f,\cR)}^{-1}(x)\vert =1$ if and only if $x$ is a totally interior point of $\cR$.
\end{prop}

\begin{proof}
	If $\vert \pi_f^{-1}(x)\vert = 1$, for all $z\in \ZZ$ the sector codes of $f^z(x)$ are all equal. Therefore, for all $z\in \ZZ$, the sectors of $f^z(x)$ are all contained in the interior of the same rectangle. Furthermore, the union of all sectors determines an open neighborhood of $f^z(x)$, which must be contained in the interior of a rectangle. Therefore, $f^z(x)\in \overset{o}{\cR}$.
	
	On the other hand, if $f^z(x)\in \overset{o}{\cR}$ for all $z\in \ZZ$, then all sectors of $f^z(x)$ are contained in the same rectangle. This implies that the sector codes of $x$ are all equal. In view of Lemma  \ref{Lemm: every code is sector code }, this implies $\vert \pi_f^{-1}(x)\vert = 1$.
	
\end{proof}

\begin{lemm}\label{Lemm: Caraterization unique codes}
	A point $x\in S$ is a totally interior point of $\cR$ if and only if it is not in the stable or unstable leaf of a periodic boundary point of $\cR$.
\end{lemm}

\begin{proof}
	
	Suppose that for all $z\in \ZZ$, $f^z(x)\in \overset{o}{\cR}$ and also $x$ is in the stable (unstable) leaf of some periodic boundary point $p\in \partial^s R_i$ ($p\in \partial^u R_i$). The contraction (expansion) on the stable (unstable) leaves of $p$ implies the existence of a $z\in \ZZ$ such that $f^z(x)\in \partial^s R_i$ ($f^z(x)\in \partial^u R_i$), which leads to a contradiction.
	
	On the contrary, Lemma  \ref{Lemm: Boundary of Markov partition is periodic} implies that the only stable or unstable leaves intersecting the boundary $\partial^{s,u} R_i$ are the stable and unstable leaves of the $s$-boundary and $u$-boundary periodic points. If $x$ is not in these laminations, neither are its iterations $f^z(x)$, and we have $f^z(x)\in \overset{o}{\cR}$ for all $z\in \ZZ$.
\end{proof}

Please note that the following definition is entirely determined by the information contained in the geometric type $T$ and its sub-shift. 
\begin{defi}\label{Defi: s,u-leafs}
	Let $\underline{w} \in \Sigma_A$. The set of \emph{stable leaf codes} of $\underline{w}$ is defined as follows:
	$$
	\underline{F}^s(\underline{w}):=\{\underline{v}\in \Sigma_A: \exists Z\in \ZZ \text{ such that } \forall z\geq Z, \, v_z=w_z\}.
	$$
	Similarly, the \emph{unstable leaf codes} of $\underline{w}$ is the set:
	$$
	\underline{F}^u(\underline{w}):=\{\underline{v}\in \Sigma_A: \exists Z\in \ZZ \text{ such that } \forall z\leq Z, \, v_{z}=w_{z}\}.
	$$
\end{defi}

We will now prove that these sets project to the stable leaf of the point $\pi_{(f,\cR)}(\underline{w})$, that they are the only codes in $\Sigma_{A(T)}$ with such a property, and that they provide a combinatorial definition of the stable and unstable leaves of a point. To this end, let us introduce some notation. For $\underline{w} \in \Sigma_A$, its \emph{positive part} consists of $\underline{w}_+ := (w_n)_{n \in \NN}$ (with $0 \in \NN$), and its \emph{negative part} consists of $\underline{w}_- := (w_{-n})_{n \in \NN}$. We denote the set of positive codes in $\Sigma_A$ as $\Sigma_A^+$, which corresponds to the positive parts of all codes in $\Sigma_A$. Similarly, we define the set of negative codes in $\Sigma_A$ as $\Sigma_A^-$. Moreover, we use $F^s(x)$ to denote the stable leaf $\cF^s(f)$ passing through $x$, and $F^u(x)$ for the corresponding unstable leaf $\cF^u(f)$.

\begin{prop}\label{Prop: Projection foliations}
	Let $\underline{w} \in \Sigma_A$. Then we have $\pi_{(f,\cR)}(\underline{F}^s(\underline{w})) \subset F^s(\pi_{(f,\cR)}(\underline{w}))$. Furthermore, assume that $\pi_{(f,\cR)}(\underline{w}) = x$.
	\begin{itemize}
		\item If $x$ is not a $u$-boundary point, then for all $y \in F^u(x)$, there exists a code $\underline{v} \in \underline{F}^s(\underline{w})$ such that $\pi_{(f,\cR)}(\underline{v}) = y$.
		
		\item If $x$ is a $u$-boundary point and $\underline{w}_0 = w_0$, then for all $y$ in the stable separatrix of $x$ that intersects the rectangle $R_{w_0}$, denoted $F^s_0(x)$, there is a code $\underline{v} \in \underline{F}^s(\underline{w})$ such that $\pi_{(f,\cR)}(\underline{v}) = y$.
		
	\end{itemize}
	
	A similar statement applies to the unstable manifold of $\underline{w}$ and its respective projection into the unstable manifold of $\pi_{(f,\cR)}(\underline{w})$.
\end{prop}

We have used the notation $\pi_f$ instead of $\pi_{(f,\cR)}$ to improve the readability of the proof.

\begin{proof}
	
	Let $\underline{v}\in \underline{F}^s(\underline{w})$ be a stable leaf code of $\underline{w}$. Based on the definition of stable leaf codes, we can deduce that $w_z=v_z$ for all $z\geq k$ for certain $k\in \NN$. Consequently, $\pi_f(\sigma^k(\underline{w})),\pi_f(\sigma^k(\underline{v}))\in R_{w_k}$. Since the positive codes of $\underline{w}$ and $\underline{v}$ are identical starting from $k$, they define the same horizontal sub-rectangles of $R_{w_k}$ in which the codes $\sigma^k(\underline{w})$ and $\sigma^k(\underline{v})$ are projected. For $n\in \mathbb{N}$, let $H_n$ be the rectangle determined by:
	
	$$
	H_n=\cap_{z=0}^n f^{-z}(R_{w_{z+k}})=\cap_{z=0}^n f^{-z}(R_{v_{z+k}}).
	$$
	The intersection of all the $H_n$ forms a stable segment of $R_{w_k}$. Furthermore, each $H_n$ contains the rectangles:
	$$
	\overset{o}{Q_n}=\cap_{z=-n}^{n} f^{-z}(\overset{o}{R_{w_{z+k}}})= \cap_{z=-n}^{n} f^{-z}(\overset{o}{R_{v_{z+k}}})
	$$
	Therefore, the projections $\pi_f(\underline{v})$ and $\pi_f(\underline{w})$ are in the same stable leaf. 
	
	\begin{figure}[ht]
		\centering
		\includegraphics[width=.65\textwidth]{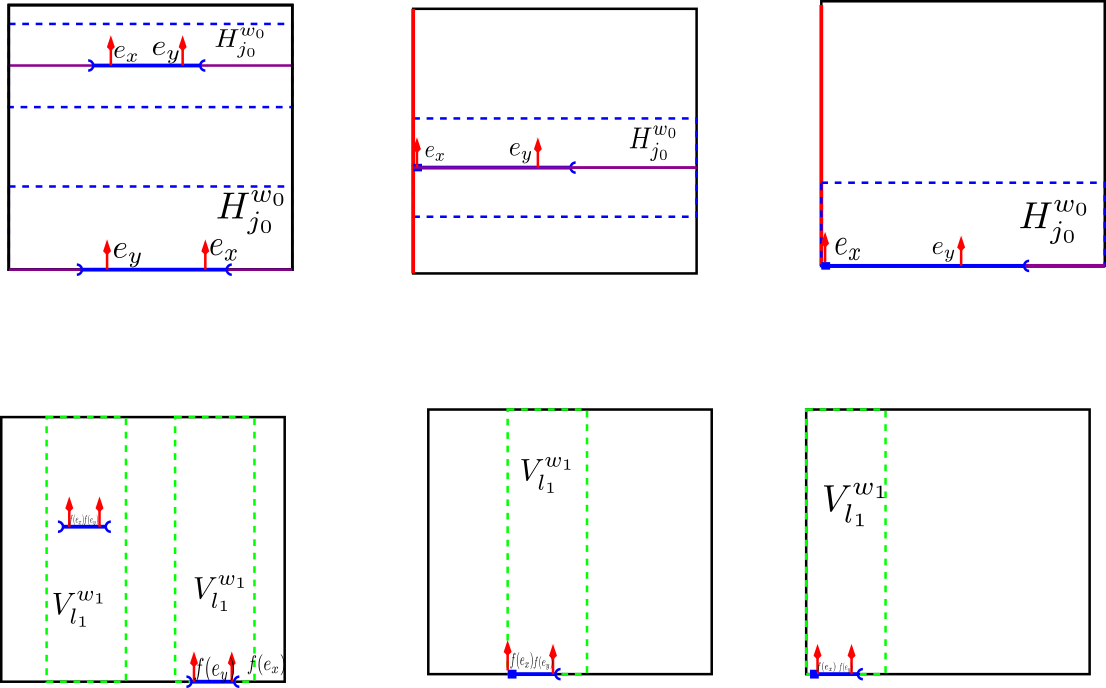}
		\caption{Projections of the stable leaf codes}
		\label{Fig: Proyection fol}
	\end{figure}

	For the other part of the argument, let's refer to the image in Figure \ref{Fig: Proyection fol} to gain some visual intuition. Let's address the situation when $x=\pi_f(\underline{w})$ is not a $u$-boundary point. We recall that is implicit that the incidence matrix $A$ is binary.

	Let $y\in F^s(\pi_f(\underline{w}))$ such that $y\neq x$. Suppose $y$ is not a periodic point (if it is, we can replace $y$ with $x$). Since  $x$ is not a $u$-boundary point, there is a small interval $I$ properly contained  in the stable segment of  $\cR_{w_0}$ that pass trough $x$. This allows us to apply  the next argument on both stable separatrices of $x$.  By definition of stable leaf, there exist $k\in \mathbb{N}$ such that $f^k(y)\in I$. If we prove there is  $\underline{v}\in \underline{F}^s(\underline{w})$ such that $\pi_f(\underline{v})=f^k(y)$, we obtain that $\sigma^{-k}(\underline{w})\in \underline{F}^s(\underline{w})$  is a code that projects to $y$, i.e $\pi_f(\sigma^{-k}(\underline{v}))=y$. Therefore we can assume  that $y\in I$. This situation is given by the left side pictures in Figure \ref{Fig: Proyection fol}.
	
	There are two situations for $f^z(y)$: either it lies in the stable boundary of $R_{w_0}$ or  not. In any case $x$, the code $\underline{w}$ is equal to  a sector code $\underline{e}_x$ of $x$. The sector $e_x$ is contained in  a unique horizontal sub-rectangle of $R_{w_0}$, $H^{w_0}_{j_0}$, therefore the sector $f(e_x)$ is contained in  $f(H^{w_0}_{j_0})=V^{w_1}_{l_1}$, implying $w_1=\underline{e_x}_1$.
	
	Take the sector   $e_y$ of $y$ in such that in the stable direction points toward $x$ and in the unstable direction point in the same direction than $e_y$. In this manner $e_y$ is contained in the rectangle $H^{w_0}_{j_0}$, therefore $f(e_y)$ is contained in $V^{w_1}_{l_1}$ and $\underline{e_y}_1=w_1$.
	
	In fact, it turns out that $f^{n}(e_y)$ and $f^n(e_x)$ are contained in the same $R_{w_n}$ for all $n\in \NN$. We can deduce that the positive part of $\underline{e_y}$ coincides with the positive part of $\underline{w}$. Therefore $\underline{e_y}\in \underline{F}^s(\underline{w})$.

	
	In the case where $x\in \partial^u\cR$, there is a slight variation. Let's assume that $\underline{w}=\underline{e_x}$, where $e_x$ is a sector of $x$. This sector code is contained within a unique horizontal sub-rectangle of $R_{w_0}$, $H^{w_0}_{j_0}$ such that $f(H^{w_0}_{j_0})=V^{w_1}_{l_1}$. This horizontal sub-rectangle contain a unique stable interval $I$ that contains $x$ in one of its end points. Consider $y\in I$.

	Similar to the previous cases, we can define a sector  $e_y$ of $x$ that is contained in horizontal sub-rectangle $H^{w_0}_{j_0}$, i.e. $\underline{e_y}=w_0$. This implies that $f(e_y)$ is contained in $V^{w_1}_{l_1}$. Then the first term of $\underline{e_y}$ is $w_1$. By applied this process inductively  for all $n\in \NN$, we get that $\underline{e_y}_n=w_n$ and therefore $\underline{e_y}\in \underline{F}^s(\underline{w})$ and projects to $y$.
\end{proof}

\subsection{Cutting rectangles along periodic codes}\label{Subsub: cutting in a fam of codes}

To describe the type of refinements we are seeking, we need to indicate how we will cut a rectangle in a Markov partition along intervals determined by a single periodic code. Once we have this, we will proceed to formulate as a theorem the existence of $s$-boundary partitions with respect to a family of periodic codes and, finally, conclude this section by carrying out the construction of such a refinement and computing its geometric type.

\begin{defi} \label{Def: Stable intervals of codes}
	Let $T$ be a symbolically modelable geometric type whose incidence matrix is denoted by $A := A(T)$. Let $\underline{w} \in \Sigma_A$ be a periodic code, and let $\underline{F}^s(\underline{w})$ denote the set of stable leaf codes of $\underline{w}$ (see Definition \ref{Defi: s,u-leafs}). For any $t \in \NN$, we define the \emph{stable interval codes} of $(t, \underline{w})$ as the set
	\begin{equation}\label{Equa: Projection stable intervals}
		\underline{I}_{t, \underline{w}} := \{\underline{v} \in \underline{F}^s(\sigma^t(\underline{w})): v_{n} = w_{t+n} \text{ for all } n \in \NN\},
	\end{equation}
	consisting of all the stable manifold codes of $\sigma^t(\underline{w})$ whose non-negative part is equal to the non-negative part of $\sigma^t(\underline{w})$.
	
\end{defi}

If $\pi_f^{-1}(\pi_f(\underline{w}))$ has more than one element, it's possible that there exist $t_1$ and $t_2$ (both not multiples of the period of $p$, to be denoted  $P$) such that $\sigma^{t_1}(\underline{w}) \neq \sigma^{t_2}(\underline{w})$ but their projections coincide. The following lemma clarifies this situation.

\begin{lemm}\label{Lemm: who the intervals intersect }
	Let  $I_{t,\underline{w}} := \pi_f(\underline{I}_{t,\underline{w}})$,  then: 
	
	\begin{itemize}
		\item[i)] $I_{t,\underline{w}}$  is a unique stable interval of $R_{w_t}$ that contains $f^t(p)$.
		
		\item[ii)] If $\underline{w}$ is a $s$-boundary code, then $I_{t,\underline{w}}$ is a stable boundary component of $R_{w_t}$.
		
		\item[iii)] If $\underline{w}$ is not a $s$-boundary code and there exist $t_1$ and $t_2$ in $\NN$ such that $\sigma^{t_1}(\underline{w}) \neq \sigma^{t_2}(\underline{w})$ but their projections coincide, i.e. $f^t(p):=\pi_f(\sigma^{t_1}(\underline{w}))=\pi_f(\sigma^{t_2}(\underline{w}))$, then the stable intervals $I_{t_1, \underline{w}}$ and $I_{t_2, \underline{w}}$ have disjoint interior and only intersect at the point $f^t(p)$.
		
	\end{itemize}
	\end{lemm}
	Let us point that in item $iii)$, $t_1-t_2$ cannot be a multiple of the period of $P$. Because if this were the case,  $t_1=t_2+mP$ and therefore:
$$
\sigma^{t_1}(\underline{w})=\sigma^{t_2}(\sigma^{mP}(\underline{w}))=\sigma^{t_2}.
$$

\begin{proof}
	\textbf{Item} $i)$. According to Proposition \ref{Prop: Projection foliations}, the set $\underline{I}_{t, \underline{w}}$ is contained in the stable manifold of $\pi_f(\underline{w})$. If $\underline{w}$ is a $u$-boundary code, then $\underline{I}_{t, \underline{w}}$ is in a unique stable separatrice. In this case, $\underline{I}_{t, \underline{w}}$ is contained in this separatrice and cover a stable interval contained in $R_{w_t}$. If $p$ is an interior point, $\underline{I}_{t, \underline{w}}$ is contained in both separatrices of $f^t(p)$ and covers a whole stable interval of $R_{w_t}$. Just remains to prove that $I_{t, \underline{w}}$ is indeed a unique  stable interval.
	
	If the positive code of $\sigma^t(\underline{v})$ coincides with $\sigma^t(\underline{v})$, then, for all $n\in \NN$, $v_{t+n}=w_{t+n}$ and $v_{t+n+1}=w_{t+n+1}$, and since the incidence matrix is binary, we can deduce that their projections are in the same horizontal sub-rectangle of $\cR$,  i.e.  $\pi_f(\sigma^{t+n}(\underline{w})),\pi_f(\sigma^{t+n}(\underline{v}))\in H^{w_{t+n}}_{j_{t+n}}$. Then:
	
	\begin{eqnarray}
		\pi_f(\sigma^t{\underline{w}}) \in \pi_f(\underline{I}_{t,\underline{w}}) \subset \bigcap_{n\in \NN} H^{w_{t+n}}_{j_{t+n}}=I \text{ and } \\
		\pi_f(\sigma^t{\underline{v}})  \in \pi_f(\underline{I}_{t,\underline{w}}) \subset \in \bigcap_{n\in \NN} H^{v_{t+n}}_{j_{t+n}}=I.
	\end{eqnarray}
	
	Therefore, $I_{t,\underline{w}}$ is contained in $I$, and $I$ is a unique horizontal interval of $R_{w_t}$, this probes the first item.
	
	\textbf{Item} $ii)$. The projection of a $s$-boundary code it is always a $s$-boundary point, therefore $I_{t,\underline{w}}$ is a stable interval of $R_{w_t}$ contained in the stable boundary of such rectangle, i.e. it is a stable boundary component of $R_{w_t}$.
	
	\textbf{Item} $iii)$. Considering that $\underline{w}$ is not an $s$-boundary code and multiple codes project to the same point, $f^t(p)$ becomes a $u$-boundary point but not an $s$-boundary point. Consequently, only two distinct codes project to $f^t(p),$ specifically $\sigma^{t_1}(\underline{w})$ and $\sigma^{t_2}(\underline{w})$. Consequently, $\underline{F}^s(\sigma^{t_1}(\underline{w}))$ and $\underline{F}^s(\sigma^{t_2}(\underline{w}))$ are mapped to different stable separatrices of $f^t(p)$. This implies that $I_{t_1, \underline{w}}$ and $I_{t_2, \underline{w}}$ belong to separate separatrices and have non-overlapping interiors. Since $f$ has no closed leaves, their combination does not form a closed curve, only intersecting at $f^t(p)$.

\end{proof}

\begin{defi}\label{Defi: s interval of R-i induced by  t }
	We call to $I_{t,\underline{w}} := \pi_f(\underline{I}_{t,\underline{w}})$ the stable interval of $R_i$ induced by the iteration $t$ of $\underline{w}$.
\end{defi}

\begin{rema}\label{Rema; finite intervals}
	If $\underline{w}$ has period $P$ then $I_{t,\underline{w}} =I_{t+P,\underline{w}}$ and the code $\underline{w}$ induces a finite number of stable intervals in the Markov partition.
\end{rema}

\begin{defi}\label{Defi: Stable intervals inside R-i}
For every rectangle $R_i \in \mathcal{R}$, the set of stable intervals of the rectangle $R_i$, determined by the periodic code $\underline{w} \in \mathcal{W}$, is given by:
\begin{equation}
	\mathcal{I}(i,\underline{w}) := \{I_{t,\underline{w}} : w_t = i \text{ and } t \in \mathbb{N} \}.
\end{equation}
And the set of all stable intervals contained in $R_i \in \mathcal{R}$ is determined by:
\begin{equation}\label{Equa: cW horizontal sub}
	\mathcal{I}(i,\mathcal{W}) := \cup_{q=1}^{Q} \mathcal{I}(i,\underline{w}^q).
\end{equation}

Finally, let $\mathcal{I}(\mathcal{W}) = \cup_{i=1}^n \mathcal{I}(i,\mathcal{W})$ be the set that contains all the stable intervals induced by the family $\mathcal{W}$.

\end{defi}

By Remark \ref{Defi: s interval of R-i induced by  t }, $\cI(i,\underline{w})$ and $\cI(\cW)$ are finite sets. The next lemma will be useful for establishing the $f$-invariance of the family $\cI(\underline{w})$.

\begin{lemm}\label{Lemm: image of stable intervals}
	For every interval $I_{i,\underline{w}}\in \cI_(i,\underline{w})$, its image under $f$ satisfies that: 
	$$
	f(I_{t,\underline{w}})\subset I_{t+1,\underline{w}}.
	$$
\end{lemm}

\begin{proof}
	The proof follows from the next computation:
	
	\begin{eqnarray*}
		f(\pi_f(\underline{I}_{t,\underline{w}}))\subset \pi_f(\sigma(\underline{I}_{t,\underline{w}}))=\\ 
		\pi_f(\underline{I}_{t+1,\underline{w}}) = I_{t+1,\underline{w}}.
	\end{eqnarray*}
	
\end{proof}

Now that we have the appropriate language we can state the theorem that reflects the desired properties of our refinement, the proof will take up the rest of this section.

\begin{theo}\label{Theo: Refinamiento s frontera fam}
	Given a geometric Markov  partition $\mathcal{R} = \{R_i\}_{i=1}^n$ of $f$, the geometric type $T$ of the pair $(f, \mathcal{R})$, and a finite family of periodic codes $\cW = \{\underline{w}^1, \cdots, \underline{w}^Q\} \subset \Sigma_{A(T)}$, there exists a finite algorithm to construct a geometric refinement of $(f, \mathcal{R})$ by horizontal sub-rectangles of $ \mathcal{R}$, known as the $s$-\emph{boundary} refinement of $(f, \mathcal{R})$ with respect to $\cW$, denoted by $(f, \textbf{S}_{\cW}(\mathcal{R}))$, with the following properties:
	\begin{itemize}
		\item Every rectangle in the Markov partition $\textbf{S}_{\cW}(\mathcal{R})$ has as stable boundary components either a stable boundary arc of a rectangle $R_i\in \mathcal{R}$ or a stable arc $I_{t, \underline{w}}$ determined by the iteration $t$ of a code $\underline{w} \in \cW$. Moreover, every arc of the form $I_{t, \underline{w}}$ for $\underline{w} \in \cW$ is the stable boundary component of some rectangle in the refinement $\textbf{S}_{\cW}(\mathcal{R})$.
		\item The periodic boundary points of $\textbf{S}_{\cW}(\mathcal{R})$ are the union of the periodic boundary points of $\mathcal{R}$ and those in the set $\{\pi_{(f,\cR)}(\sigma_A^t(\underline{w})) \mid \underline{w} \in \cW \text{ and } t \in \NN\}$.
		\item There are explicit formulas, in terms of the geometric type $T$ and the codes in $\cW$, to compute the geometric type of $\textbf{S}_{\cW}(\mathcal{R})$.
	\end{itemize} 
\end{theo}

\subsection{The geometric Markov partition \texorpdfstring{$\cR_{S(\cW)}$}{R\_S(W)}.}

\subsubsection{The rectangles in \texorpdfstring{$\cR_{S(\cW)}$}{R\_S(W)}.}

We need to establish that after cutting our rectangles along the stable intervals determined by the family $\mathcal{W}$, $\mathcal{I}(\mathcal{W})$, the resulting connected components are all rectangles. The next lemma is key to this property.

\begin{lemm}\label{Lemm: rectangles in R S(p)}
	Let $\tilde{R}$ be the closure of a connected component of $\overset{o}{R_i}\setminus \cI(\cW)$. Then $\tilde{R}$ is a horizontal sub-rectangle of $R_i$
\end{lemm}

\begin{proof}
	If $i\neq j$, then $\overset{o}{R_i}\cap \cup \cI(j,\underline{w})=\emptyset$ for all $\underline{w}\in \cW$, because for any point $x\in R_i$ and every code $\underline{v}\in \pi^{-1}_{(f,\cR)}(x)$, $v_0=i$. In contrast, if $y\in \cup\cI(j,\underline{w})$, at least one code $\underline{w}'=\sigma^t(\underline{w})$ in $\pi^{-1}(x)$ has its zero term equal to $j$, i.e., $w'_0=j$. Therefore, $x\neq y$.

	We only need to consider the case when $I_{t,\underline{w}}\subset R_i$, for $\underline{w}\in\cW$. In this situation, $I_{t,\underline{w}}$ is a horizontal interval of $R_i$. Consequently, each connected component of $\overset{o}{R_i}\setminus \cup \cI(i,\underline{w})$ is a horizontal sub-rectangle $H$ of $R_i$ minus its stable boundaries, and its closure coincides with $H$. This argument proves our lemma.
\end{proof}

\begin{defi}\label{Lemm: cR-S(cW) is markov part}
	The family, which includes all the rectangles as described in Lemma \ref{Lemm: rectangles in R S(p)}, corresponding to all the rectangles $R_i$ in $\cR$, is denoted is denoted by $\cR_{S(\cW)}$.
\end{defi}

\begin{lemm}\label{Lemm: Markov partition cW }
The family of rectangles $\mathcal{R}_{S(\mathcal{W})}$ is a Markov partition for $f$, consisting of horizontal sub-rectangles of the rectangles in $\mathcal{R}$.
\end{lemm}

\begin{proof}
	The interiors of two distinct horizontal sub-rectangles within $R_i$ do not intersect, and a horizontal sub-rectangle from $R_i$ does not overlap with the interior of a horizontal sub-rectangle from $R_j$ if $j\neq i$. Consequently, the interiors of two different rectangles in $\cR_{\cW}$ are disjoint.

	The union $\cup \cR_{S(\cW)}$ coincides with the union of the rectangles in $\cR$, but $\cR$ is  Markov partition for $f$, therefore  $\cup \cR_{S(\cW)}=\cup \cR =S$.

	The stable boundary of $\cR_{S(\cW)}$ consists of two types of stable intervals: those that are part of the stable boundary of $\cR$ and whose image is contained in $\partial^s\cR \subset \partial^s \cR_{S(\cW)}$, and those of the form $I_{t,\underline{w}^q}$ for some $\underline{w}^q\in \cW$ that, as stated in Lemma \ref{Lemm: image of stable intervals}: $f(I_{t,\underline{w}^q}) \subset I_{t+1,\underline{w}^q}\subset \partial^s\cR_{S(\cW)}$. Therefore, $\partial^s\cR_{S(\cW)}$ is  $f$-invariant.

	Since the unstable boundary of $\cR_{S(\cW)}$ coincides with the unstable boundary of $\cR$, $\partial^u\cR_{S(\cW)}$ is $f^{-1}$-invariant.
	
	This proves that $\cR_{S(\cW)}$ is a Markov partition of $f$.
\end{proof}

\begin{rema}\label{Rema: Why not s boundary code?}
If the periodic code $\underline{w} \in \mathcal{W}$ is an $s$-boundary code, Item $ii)$ of Lemma \ref{Lemm: who the intervals intersect } states that every interval in the set $\mathcal{I}(i,\underline{w})$ is contained within the stable boundary of $R_i$. Consequently, the set of rectangles obtained by cutting the rectangle $R_i \in \mathcal{R}$ through the segments in $\mathcal{I}(i,\mathcal{W})$ coincides with $R_i$. Therefore, from now on, we are going to assume that all the codes in $\mathcal{W}$ are not $s$-boundary codes.
\end{rema}

\subsubsection{The relative position of the rectangles: One code.} We must assume for a while that $\mathcal{W} := \{\underline{w}\}$ contains only one code in order to prove some technical lemmas. The generalizations are straightforward but will be introduced later, as this simpler case is sufficient to gain intuition.

In order to have a geometric refinement, we must assign a labeling to the rectangles of $\mathcal{R}_{\mathcal{S}(\mathcal{W})}$ that is coherent with the vertical orientation of the rectangles in $\mathcal{R}$. We can do this by declaring that for a pair of rectangles $R_1, R_2 \in \mathcal{R}_{\mathcal{S}(\mathcal{W})}$ contained in $R_i$, we have the labels $R_{(i, s_1)} = R_1$ and $R_{(i, s_2)} = R_2$ with $s_1 < s_2$ if and only if, with respect to the vertical direction of $R_i$, we have $R_1 <_{v(i)} R_2$. However, this procedure is not algorithmic; therefore, we need to develop a method to determine when two different stable intervals in $\mathcal{I}(i, \mathcal{W})$ are one above the other. For this reason, the first thing we are going to determine is when two horizontal intervals induced by $\underline{w} \in \mathcal{W}$ and contained in the same rectangle $R_i \in \mathcal{R}$ are one above the other with respect to the vertical orientation of $R_i$.

\begin{defi}\label{Defi: cO(i) and O(i)}
Let $\underline{w} \in \mathcal{W}$ be a non-$s$-boundary periodic code with period $P$. Then, for every $i \in \{1, \ldots, P-1\}$, we define:
\begin{equation}\label{Equa: cO(i,w)}
	\mathcal{O}(i, \underline{w}) = \{(t, \underline{w}) : t \in \{0, \ldots, P-1\} \text{ and } w_t = i\}.
\end{equation}
The cardinality of this set is denoted by $O(i, \underline{w})$. Similarly, let
\begin{equation}\label{Equa: cO(W)}
	\mathcal{O}(i, \mathcal{W}) = \bigcup_{\underline{w} \in \mathcal{W}} \mathcal{O}(i, \underline{w}).
\end{equation}
Let $O(i, \mathcal{W})$ be its cardinality.
\end{defi}

Let us recall that $\mathcal{H}(T)$ and $\mathcal{V}(T)$ are the sets of horizontal and vertical labels for the geometric type $T$, and that we are always assuming $A(T)$ is binary.

\begin{lemm}\label{Lemm: unique hriwontal for every prjection} 
Let $\underline{w} \in \mathcal{W}$ with period $P$ and take $t \in \{1, \cdots, P\}$. There is a unique index $j_{t, \underline{w}} \in \{1, \cdots, h_{w_t}\}$ that satisfies the following condition: $(w_t, j_{t, \underline{w}}) \in \mathcal{H}(T)$ is the unique pair of indices for which $\rho_T(w_t, j_{t, \underline{w}}) = (w_{t+1}, l_{t+1}) \in \mathcal{V}(T)$ (for some $l_{t+1}\in \{1,\cdots, v_{w_{t+1}}\}$). Furthermore, $l_{t+1} \in \{1, \cdots, v_{w_{t+1}}\}$ is also uniquely determined.
\end{lemm}

\begin{proof}
	Let's suppose that $w_t=i$. As the incidence matrix of $T$ is binary, there exists at most one horizontal sub-rectangle $H^i_j$ of $R_i$ such that $f(H^i_j)\subset R_{w_{t+1}}$. However, since $\underline{w}$ is an admissible code, there exists at least one horizontal sub-rectangle $H^i_j$ of $R_i$ such that $f(H^i_j)\subset R_{w_{t+1}}$. Therefore, there exists only one rectangle of $R_i$ such that $f(H^i_j)\subset R_{w_{t+1}}$, and we declare $j_{t,\underline{w}}$ as this unique number. Clearly, if $\rho_T(i,j_{t,\underline{w}})=(w_{t+1},l_{t+1})$, then $l_{t+1}\in \{1,\cdots,h_{w_{t+1}}\}$ is uniquely determined.
\end{proof}

\begin{coro}\label{Coro: projetion stable segments}
	Assume that $w_t=i$, then the stable interval $I_{t,\underline{w}}$ is contained in the horizontal sub-rectangle $H^i_{j_{t,\underline{w}}}\setminus \partial^s H^i_{j_{t,\underline{w}}}$ of $\cR$.
\end{coro}

\begin{proof}
	
	It is evident that $\pi_f(\underline{w})\in H^i_{j_{t,\underline{w}}}$. Now, consider another code $\underline{v}\in \underline{F}^s(\sigma^t(\underline{w}))$ that projects to $I_{t,\underline{w}}$. As in Lemma \ref{Lemm: unique hriwontal for every prjection}, there exists only one index $j_{t, \underline{v}}\in \{1,\ldots, h_{v_t}\}$ such that $\rho_T(i, j_{t, \underline{v}})=(v_{t+1},l')$, and therefore $\pi_f(\underline{v})\in H^i_{j_{t, \underline{v}}}$. However, due to the matching of non-negative codes between $\underline{w}$ and $\underline{v}$, we have $w_{t}=v_{t}$ and $w_{t+1}=v_{t+1}$, leading to $j_{t, \underline{v}}=j_{t, \underline{w}}$. Consequently, $\pi_f(\underline{v})\in H^i_{j_{t,\underline{w}}}$.
	
	Since $\underline{w}$ is not an $s$-boundary point, the projection of its stable leaf codes cannot intersect the stable leaves of  $s$-boundary periodic points. In particular, it does not intersect the stable boundary of the horizontal sub-rectangles of $(f,\cR)$. Hence, we deduce that $I_{t,\underline{w}} \subset H^i_{j_{t,\underline{w}}}\setminus \partial^s H^i_{j_{t,\underline{w}}}$. This confirms the claim.
	
\end{proof}

\begin{lemm}\label{Lemm: first order in O(i)}
	Let $I_{t_1,\underline{w}}$ and $I_{t_2,\underline{w}}$ be two stable segments contained within $R_i$, i.e., $w_{t_1}=w_{t_2}=i$. Suppose that $j_{t_1,\underline{w}}<_{v(i)}j_{t_2,\underline{w}}$. Then, with respect to the vertical orientation of $R_i$:
	$$
	I_{t_1,\underline{w}}<_{v(i)}I_{t_2,\underline{w}}.
	$$
\end{lemm}

\begin{proof}

	As was observed in Corollary \ref{Coro: projetion stable segments}, $I_{t_1,\underline{w}}\subset H^i_{j_{t_1,\underline{w}}}$ and $I_{t_2,\underline{w}}\subset H^i_{j_{t_2,\underline{w}}}$. With respect to the vertical orientation of $R_i$ any stable interval contained in $H^i_{j_{t_1,\underline{w}}}$ lies below any stable interval contained in $H^i_{j_{t_2,\underline{w}}}$. In particular, we have $I_{t_1,\underline{w}}<I_{t_2,\underline{w}}$, as we claimed.
\end{proof}

\begin{lemm}\label{Lemm: diferentation moment of the order}
Let $\underline{w} \in \mathcal{W}$ with period $P$. If there exist two distinct numbers $t_1$ and $t_2$ in $\{0, \cdots, P-1\}$ such that $\sigma^{t_1}(\underline{w}) \neq \sigma^{t_2}(\underline{w})$ but both have $w_{t_1} = w_{t_2} = i$, then we consider the two elements in $\mathcal{O}(i, \underline{w})$ determined by these iterations: $(t_1, \underline{w})$ and $(t_2, \underline{w})$. In this case:
	\begin{itemize}
		\item[i)]   There exists an natural number $M$ between $1$ and $P-1$ such that $w_{t_1+M} \neq w_{t_2+M}$, but for all $m$ from $0$ to $M-1$, $w_{t_1+m}=w_{t_2+m}$.
		
		\item[ii)]  If $M=1$, then $j_{t_1,\underline{w}}\neq j_{t_2,\underline{w}}$. However, if $M\geq 2$, for all $m\in \{0,\cdots,M-2\}$, the indices $j_{t_1+m ,\underline{w}}$ and $j_{t_2+m,\underline{w}}$ (introduced in in \ref{Coro: projetion stable segments}) are equal. Thus, for all $m$ in the range $0$ to $M-2$:
		$$
		\epsilon_T(w_{t_1+m},j_{t_1 +m, \underline{w}})=\epsilon_T(w_{t_2+m},j_{t_2 +m, \underline{w}}).
		$$
		
		\item[iii)]  The indices $j_{t_1+M-1,\underline{w}}, j_{t_2+M-1,\underline{w}}\in \{1,\cdots, h_{w_{t_1 +M-1}}=h_{w_{t_2 +M-1}}\}$ are different.
	\end{itemize}
	
\end{lemm}

\begin{proof}
	\textbf{Item} $i)$. Since $\sigma^{t_1}(\underline{w})\neq \sigma^{t_2}(\underline{w})$, and $\underline{w}$ has a period of $P$, there must exist an integer $M \in \{0, \cdots, P-1\}$ such that $w_{t_1+M} \neq w_{t_2+M}$, and $M$ is chosen as the smallest integer with this property. We claim that $M\neq 0$ because $w_{t_1+0}=w_{t_2+0}$. Since $M$ is the minimum integer where $w_{t_1+M} \neq w_{t_2+M}$, then for all $0\leq m<M$, it holds that $w_{t_1+m}= w_{t_2+m}$.
	
	\textbf{Item} $ii)$.If $M=1$, then $w_{t_1+1} \neq w_{t_2+1}$. Since the incidence matrix $A$ is binary, this implies that $j_{t_1,\underline{w}} \neq j_{t_2,\underline{w}}$.
	
	Now, let's consider the case when $M>1$. For each $m \in \{0, \cdots, M-2\}$, the indices $j_{t_1+m,\underline{w}}$ and $j_{t_2+m,\underline{w}}$ are the only ones that satisfy $\rho_T(w_{t_1+m},j_{t_1+m,\underline{w}})=(w_{t_1+m+1},l_m)$ and $\rho_T(w_{t_2+m},j_{t_2+m,\underline{w}})=(w_{t_2+m+1},l'_m)$. However, since $w_{t_1+m+1}=w_{t_2+m+1}$, it follows that $j_{t_1+m,\underline{w}}=j_{t_2+m,\underline{w}}$. Consequently, we have:
	$$
	\epsilon_T(w_{t_1+m},j_{t_1 +m, \underline{w}})=\epsilon_T(w_{t_2+m},j_{t_2 +m, \underline{w}}).
	$$
	
	\textbf{Item} $iii)$. Remember that $\rho_T(w_{t_1 +M-1},j_{t_1+M-1,\underline{w}})=(w_{t_1+M},l)$ and $\rho_T(w_{t_2 +M-1},j_{t_2+M-1,\underline{w}})=(w_{t_2+M},l')$. Since $w_{t_1+M}\neq w_{t_2+M}$ and the incidence matrix of $T$ is binary, it is necessarily the case that $j_{t_1+M-1,\underline{w}}\neq j_{t_2+M-1,\underline{w}}$.
\end{proof}

\begin{defi}\label{Defi: Intetchage order funtion}
Let $T \in \mathcal{G}\mathcal{T}(\textbf{p-A})^{SIM}$ and denote its incidence matrix by $A$. Suppose $\underline{w} \in \cW$ is a periodic code of period $P$ that is not an $s$-boundary code. Assume that $\sigma^{t_1}(\underline{w}) \neq \sigma^{t_2}(\underline{w})$ but $w_{t_1} = w_{t_2} = i$. Let $M$ be the number determined in Lemma \ref{Lemm: diferentation moment of the order}. We define the \emph{interchange order} between $(t_1, \underline{w})$ and $(t_2, \underline{w})$ as follows:

	\begin{itemize}
		\item If $M=1$ then $\delta((t_1,\underline{w}), (t_2,\underline{w})):=1$.
		\item If $M>1$ then
		$$
		\delta((t_1,\underline{w}), (t_2,\underline{w}))=\prod_{m=0}^{M-2}\epsilon_T(w_{t_1+m},j_{t_1 +m, \underline{w}})=\prod_{m=0}^{M-2}\epsilon_T(w_{t_2+m},j_{t_2 +m, \underline{w}}).
		$$
	\end{itemize}
\end{defi}

\begin{lemm}\label{Lemm: Comparacion de ordenes}
	Let $T$ be a geometric type in the pseudo-Anosov class with a binary incidence matrix $A$. Suppose $\underline{w}\in \Sigma_A$ is a periodic code of period $P\geq 1$ that is not an $s$-boundary code. Assume that $\sigma^{t_1}(\underline{w}) \neq \sigma^{t_2}(\underline{w})$ but $w_{t_1} = w_{t_2}=i$. Consider a realization $(f,\cR)$ of the geometric type $T$ by a pseudo-Anosov homeomorphism $f:S\to S$. Let $I_{t_1,\underline{w}}$ and $I_{t_2,\underline{w}}$ be the two stable segments determined by the numbers $t_1$ and $t_2$. Then, $I_{t_1,\underline{w}} < I_{t_2,\underline{w}}$ with respect to the vertical orientation of $R_i\in \cR$ if and only if one of the following situations occurs:
	\begin{itemize}
		\item $j_{t_1+M-1}<j_{t_2+M-1}$ and $\delta((t_1,\underline{w}), (t_2,\underline{w}))=1$, or
		\item $j_{t_1+M-1}>j_{t_2+M-1}$ and $\delta((t_1,\underline{w}), (t_2,\underline{w}))=-1$, 
	\end{itemize}
	
\end{lemm}

\begin{proof}
	The case when $M=1$ reduces to Lemma \ref{Lemm: first order in O(i)}, so let's assume that $M>1$. Assume that $I_{t_1,\underline{w}} < I_{t_2,\underline{w}}$ and call $H$ the horizontal sub-rectangle determined by such stable intervals. For all $0\leq m \leq M-2$, $j_{t_1+m}=j_{t_2+m}$ and $f^{m}(H)$ is contained in a unique horizontal sub-rectangle of $R_{w_{t_1+m}}$. In this case, after applying $f^{M-1}$ to $H$, we have two possibilities:
	
	\begin{itemize}
		\item[i)]  $j_{t_1+M-1}<j_{t_2+M-1}$ and $f^{M-1}$ preserves the vertical orientation, which happens if and only if $\delta((t_1,\underline{w}), (t_2,\underline{w}))=1$, or
		\item[ii)]  $j_{t_1+M-1}>j_{t_2+M-1}$ and $f^{M-1}$ changes the vertical orientations, which happens if and only if $\delta((t_1,\underline{w}), (t_2,\underline{w}))=-1$.
	\end{itemize}
		
	Now, assume that $j_{t_1+M-1}<j_{t_2+M-1}$ and $\delta((t_1,\underline{w}), (t_2,\underline{w}))=1$. Let  $H$ be  the horizontal sub-rectangle of $R_{w_{t_1}}$ determined by the intervals $I_{t_1,\underline{w}}$ and $I_{t_2,\underline{w}}$, we need to determine what of such intervals is the upper boundary of $H$.
	
	The number $\delta((t_1,\underline{w}), (t_2,\underline{w}))$ measures the change in the relative positions of $f^{M-1}(I_{t_1,\underline{w}})$ and $f^{M-1}(I_{t_2,\underline{w}})$ inside the rectangle $R_{w_{t_1+M-1}}$. In this case, the relative position of $f^{M-1}(I_{t_1,\underline{w}})$ and $f^{M-1}(I_{t_2,\underline{w}})$ inside $R_{w_{t_1+M-1}}$ is the same as the relative position of $I_{t_1+M-1,\underline{w}}$ and $I_{t_2+M-1,\underline{w}}$ inside $R_i$. Since $\delta((t_1,\underline{w}), (t_2,\underline{w}))=1$, the upper boundary of $f^{M-1}(H)$ coincides with the image of the upper boundary of $H$ and similarly, the inferior boundary of $H$ corresponds to the inferior boundary of $f^{M-1}(H)$.
	
	Furthermore, since $j_{t_1+M-1}<j_{t_2+M-1}$ and according to Lemma \ref{Lemm: first order in O(i)}, $I_{t_1+M-1,\underline{w}} < I_{t_2+M-1,\underline{w}}$. So, we can conclude that the inferior boundary of $f^{M-1}(H)$ is contained within $I_{t_1+M-1,\underline{w}}$ and then the inferior boundary of $H$ is $I_{t_1,\underline{w}}$, and its upper boundary is $I_{t_2,\underline{w}}$. In this manner, we have $I_{t_1,\underline{w}}<I_{t_2,\underline{w}}$ with respect to the vertical order in $R_i$.

	Let's consider a scenario where $j_{t_1+M-1}>j_{t_2+M-1}$ and $\delta((t_1,\underline{w}), (t_2,\underline{w}))=-1$. In this case, we take $H$ as the horizontal sub-rectangle of $R_{w_{t_1}}$ determined by the intervals $I_{t_1,\underline{w}}$ and $I_{t_2,\underline{w}}$.
	The negative value of $\delta((t_1,\underline{w}), (t_2,\underline{w}))$ indicates that $f^{M-1}\vert_H$ alters the relative positions of the stable boundaries of $H$ with respect to the vertical order in $R_i$. Consequently, the upper boundary of $f^{M-1}(H)$ is the image of the inferior boundary of $H$, and simultaneously, the lower boundary of $f^{M-1}(H)$ is the image  of the upper boundary of $H$.
	
	Given that $j_{t_1+M-1}>j_{t_2+M-1}$, it follows that the lower boundary of $f^{M-1}(H)$ is enclosed within $H^{w_{t_1+M-1}}_{j_{t_1+M-1,\underline{w}}}$, and the upper boundary is encompassed within $H^{w_{t_2+M-1}}_{j_{t_2+M-1,\underline{w}}}$. Therefore by  Corollary \ref{Coro: projetion stable segments}  we have that $I_{t_1+M-1,\underline{w}} > I_{t_2+M-1,\underline{w}}$, and the upper boundary of $f^{M-1}(H)$ is confined within $I_{t_1+M-1,\underline{w}}$.
	Using  Lemma \ref{Lemm: image of stable intervals}, we can deduce that   $f^{M-1}(I_{t_1,\underline{w}})\subset I_{t_1+M-1,\underline{w}}$  and then that $I_{t_1,\underline{w}}$ is the inferior boundary of $H$. Finally we get that: $I_{t_1,\underline{w}}<I_{t_2,\underline{w}}$ with respect to the vertical order in $R_i$.
	
\end{proof}

We will now establish a formal ordering for the set of stable segments within $R_i$, which constitute the stable boundaries of the rectangles contained in $\cR_{S(\underline{w})}$ that also lie within $R_i$.

\begin{defi}\label{Lemm: Boundaries code}
	Let $R_i\in\cR$ be any rectangle, denote $I_{i,-1}:=\partial^s_{-1} R_i$ as the lower boundary of $R_i$, and $I_{i,+1}:=\partial^s_{+1} R_i$ as the upper boundary of $R_i$. Define:
	$$
	\underline{s}:\cO(i,\underline{w})\cup \{(i,+1), (i,-1) \} \rightarrow \{0,1,\cdots,O(i,\underline{w}),O(i,\underline{w})+1\},
	$$
	as the unique function such that:
	\begin{itemize}
		\item $\underline{s}(t_1,\underline{w})< \underline{s}(t_1,\underline{w})$ if and only if  $I_{t_1,\underline{w}}< I_{t_1,\underline{w}}$ respect the vertical  order in $R_i$.
		\item It assigns $0$ to   $(i,-1)$ and  $O(i,\underline{w})+1$ to  $(i,+1)$.
	\end{itemize}
\end{defi}

The boundaries of any rectangle $\tilde{R}_r$ of $\cR_{S(\underline{w})}$ contained in $R_i$ are defined by two consecutive horizontal intervals in $\cO(i,\underline{w})\cup \{I_{i,-1},I_{i,+1}\}$. The previous order enables us to describe them accurately based on the relative position of $\tilde{R}_r$ inside $R_i$.

\begin{lemm}\label{Lemm: Determine boundaris of R-r}

	Let $\tilde{R}_r$ be a rectangle in the geometric Markov partition $\cR_{S(\underline{w})}$ with $r=\tilde{r}(i,s)$. Then, the stable boundary of $\tilde{R}_r$ consists of two stable segments of $R_i$ determined as follows:
	
	\begin{itemize}
		\item[i)] If $s=1$, the lower boundary of $\tilde{R}_r$ is $I_{i,-1}$. If  $O(i,\underline{w})=0$, then its upper boundary is $I_{i,+1}$. However, if $O(i,\underline{w})>0$, there exists a unique $(t,\underline{w})\in \cO(i,\underline{w})$ such that $\underline{s}(t,\underline{w})=1$, and the upper boundary of $\tilde{R}_r$ is the stable segment $I_{t,\underline{w}}$.
		
		\item[ii)] If $O(i,\underline{w})>1$ and $s \neq 1, O(i,\underline{w})+1$, there are unique $(t_1,\underline{w}), (t_2,\underline{w})\in \cO(i,\underline{w})$ such that: $\underline{s}(t_1,\underline{w})=s-1$, $\underline{s}(t_2,\underline{w})=s$. Furthermore, the lower boundary of $\tilde{R}_r$ is the stable segment $I_{t_1,\underline{w}}$, and the upper boundary of $\tilde{R}_r$ is the stable segment $I_{t_2,\underline{w}}$.
		
		\item[iii)] If $O(i,\underline{w})>0$ and $s=O(i,\underline{w})+1$, then the upper boundary of $\tilde{R}_r$ is the stable segment $I_{i,+1}$. Additionally, there exists a unique $(t,\underline{w})\in \cO(i,\underline{w})$ such that $\underline{s}(t,\underline{w})=O(i,\underline{w})$, and the inferior boundary of $\tilde{R}_r$ is the stable segment $I_{t,\underline{w}}$.
	\end{itemize}
	
\end{lemm}

\begin{proof}
	Recall that $s$ is the vertical position of $\tilde{R}_r$ inside $R_i$ but $\tilde{s}(t,\underline{w})$ is the position of the horizontal interval $I_{t,\underline{w}}$ inside $R_i$.
	
	\textbf{Item } $i)$. Since $s=1$, within the vertical ordering of $R_i$, $\tilde{R}_r$ represents the first rectangle from $\cR_{S(\underline{w})}$ contained within $R_i$. Therefore, its lower boundary coincides with the lower boundary of $R_i$, which is $I_{i,-1}$.
	
	If $O(i,\underline{w})=0$, then $\tilde{R}_r=R_i$, and it's the only rectangle from $\cR_{S(\underline{w})}$ within $R_i$. Consequently, the upper boundary of $\tilde{R}_r$ must be $I_{i,+1}$.
	
	However, if $O(i,\underline{w})>0$, the upper boundary of $\tilde{R})_r$ must correspond to a stable interval $I_{t,\underline{w}}\in \cI(i,\underline{w})$, which occupies the first position according to the order $\underline{s}$; hence, $\underline{s}(t,\underline{w})=1$.

	\textbf{Item } $ii)$. Since $\tilde{R}_r$ occupies the vertical position $1 < s < O(i,\underline{w})+1$ among all the rectangles in $\cR_{S(\underline{w})}$ within $R_i$, its stable boundaries cannot coincide with $I_{i,-1}$ or $I_{i,+1}$. Therefore, its stable boundary is defined by two consecutive stable segments in $\cI(i,\underline{w})$, denoted as $I_{t_1,\underline{w}} < I_{t_2,\underline{w}}$. Furthermore, the interval $I_{t_1,\underline{w}}$ holds the position $s-1$ relative to the elements in $\cI(i,\underline{w})\cup {I_{i,+1}, I_{i,-1}}$, and $I_{t_2,\underline{w}}$ is positioned at $s$. This leads to the conclusion that $\underline{s}(t_1,\underline{w})=s-1$ and $\underline{s}(t_2,\underline{w})=s$.

	\textbf{Item }  $iii)$.Since $s=O(i,\underline{w})+1$, $\tilde{R}_r$ is the upper rectangle of $\cR_{S(\underline{w})}$ contained in $R_i$. Consequently, its upper boundary is $I_{i,+1}$, and its lower boundary is the segment $I_{t,\underline{w}}$ located in position $O(i,\underline{w})$ among the elements of $\cI(i,\underline{w})\cup \{I_{i,+1}, I_{i,-1}\}$. Therefore, $\underline{s}(t,\underline{w})=O(i,\underline{w})$.
	
\end{proof}

Let's determine the image of the stable boundaries of a rectangle $R_r$. We are going to use Lemma \ref{Lemm: image of stable intervals} as our main tool.

\begin{lemm}\label{Lemm: imge of boundaries}
	
	Let $\tilde{R}_r$, where $r=\tilde{r}(i,s)$, be a rectangle in the geometric Markov partition $\cR_{S(\underline{w})}$. We can determine the image of the horizontal boundary of $\tilde{R}_r$ under $f$ in the following manner:
	
	\begin{itemize}
		\item[i)] If the lower boundary of $\tilde{R}_r$ is $I_{i,-1}$ and $\Phi_T(i,1)=(k,l,\epsilon)$, then $f(I_{i,-1})$ is contained within $I_{k,-1}$ if $\epsilon=1$, and it is contained within $I_{k,+1}$ otherwise.

		\item[ii)] If one boundary component of $\tilde{R}_r$ is the stable segment $I_{t,\underline{w}}$, then $f(I_{t,\underline{w}})\subset I_{t+1,\underline{w}}$.
		
		\item[iii)]  If $I_{i,+1}$ is the upper boundary of $\tilde{R}_r$ and $\Phi_T(i,h_i)=(k,l,\epsilon)$, then $f(I_{i,+1})$ is contained in $I_{k,+1}$ if $\epsilon=1$ and is contained in $I_{k,-1}$ otherwise.
		
	\end{itemize}
		
\end{lemm}

\begin{proof}
	\textbf{Item} $i)$. It's evident that the image of $I_{i,-1}$ under $f$ will be contained either in the upper or lower boundary of $R_k$, depending on how $f$ alters the vertical orientation. Specifically, if $\epsilon=1$, then $f(I_{i,-1})$ is within the lower boundary of $R_k$, whereas if $\epsilon=-1$, it will be situated in the upper boundary of $R_k$.
	
	\textbf{Item}  $ii)$. This is Lemma \ref{Lemm: image of stable intervals}.
	
	\textbf{Item} $iii)$. The argument is the same as in Item $i)$.	
\end{proof}

\subsubsection{\texorpdfstring{Generalizations for the family $\mathcal{W}$}{Generalizations for the family W}}\label{Subsec: Generalizations for the family}

Almost any result in this sub-section can be proved in the same manner as the respective result was proved for only one code. For example, the next lemma is proved with the same strategy developed in Lemma  \ref{Lemm: diferentation moment of the order}. The number $M$ is the minimum natural number such that $w^1_{t+1+M}\neq w^2_{t_2+M}$, and all the other properties follow from this number.

\begin{lemm}\label{Lemm: diferentation moment of the order,cW}

	Suppose that $\underline{w}^1$ and $\underline{w}^2$ are two codes in $\cW$ with periods $P_1$ and $P_2$, respectively. Let $t_1$ be a number in $\{0,\cdots, P_1-1\}$ and $t_2$ be a number in $\{0,\cdots, P_2-1\}$ such that $\sigma^{t_1}(\underline{w}^1) \neq \sigma^{t_2}(\underline{w}^2)$, but both $w^1_{t_1}$ and $w^2_{t_2}$ are equal to  $i\in\{1,\cdots,n\}$. Consider the two elements in $\cO(i,\cW)$ determined by these iterations: $(t_1,\underline{w}^1)$ and $(t_2,\underline{w}^2).$ In this way:

	\begin{itemize}
		\item[i)]  There exists an integer $M\in\{1,\cdots, P-1\}$ such that $w^1_{t_1+M} \neq w^2_{t_2+M}$, but for all $0 \leq m \leq M-1$, the numbers $w^1_{t_1+m}$ and $w^2_{t_2+m}$ are equal.
		
		\item[ii)] 	If $M=1$, the indexes $j_{t_1,\underline{w}^1}$ and $j_{t_2,\underline{w}^2}$ (Definition \ref{Lemm: rectangles in R S(p)}) are distinct. However, when $M\geq 2$, for every $m$ in the set $\{0,\cdots,M-2\}$, the indices $j_{t_1+m,\underline{w}^1}$ and $j_{t_2+m,\underline{w}^1}$ (see Definition \ref{Lemm: rectangles in R S(p)}) are equal. Thus, for all $m\in\{0,\cdots,M-2\}$:
		$$
		\epsilon_T(w^1_{t_1+m},j_{t_1 +m, \underline{w}^1})=\epsilon_T(w^2_{t_2+m},j_{t_2 +m, \underline{w}^2}).
		$$
		
		\item[iii)]  The  indices  $j_{t_1+M-1,\underline{w}^1}$ and $j_{t_2+M-1,\underline{w}^2}$ in $\{1,\cdots, h_{w^1_{t_1 +M-1}}=h_{w^2_{t_2 +M-1}}\}$  are different.
	\end{itemize}
	
\end{lemm}

This permits us to define the interchange function as  in Definition \ref{Defi: Intetchage order funtion} but this time respect two non necessarily equal  codes in $\cW$.

\begin{defi}\label{Defi: Intetchage order funtion,cW}
	
	Suppose that $\underline{w}^1$ and $\underline{w}^2$ are two codes in $\cW$ with periods $P_1$ and $P_2$, respectively. Let $t_1$ be a number in $\{0,\cdots, P_1-1\}$ and $t_2$ be a number in $\{0,\cdots, P_2-1\}$ such that $\sigma^{t_1}(\underline{w}^1) \neq \sigma^{t_2}(\underline{w}^2)$ but $w^1_{t_1} = w^2_{t_2}=i$. Let $M$ be the number determined in Lemma \ref{Lemm: diferentation moment of the order,cW}. We define the \emph{interchange order} between $(t_1, \underline{w}^1)$ and $(t_2, \underline{w}^2)$ as follows:
	
	\begin{itemize}
		\item If $M=1$ then $\delta((t_1,\underline{w}^1),(t_2,\underline{w}^2))=1$.
		\item If $M>1$ then
		$$
		\delta((t_1,\underline{w}^1),(t_2,\underline{w}^2))=\prod_{m=0}^{M-2}\epsilon_T(w^1_{t_1+m},j_{t_1 +m, \underline{w}^1})=\prod_{m=0}^{M-2}\epsilon_T(w^2_{t_2+m},j_{t_2 +m, \underline{w}^2}).
		$$
	\end{itemize}
\end{defi}

Lemma \ref{Lemm: Comparacion de ordenes} admits the following generalization whose proof is essentially the same.

\begin{lemm}\label{Lemm: Comparacion de ordenes, cW}
	Let $I_{t_1,\underline{w}^1}$ and $I_{t_2,\underline{w}^2}$ be two stable segments such that $w^1_{t_1}=w^2_{t_2}=i$. Then, $I_{t_1,\underline{w}^1} <_{v(i)} I_{t_2,\underline{w}^2}$ with respect to the vertical orientation of $R_i$ if and only if one of the following situations occurs:
	
	\begin{itemize}
		\item $j_{t_1+M-1,\underline{w}^1}<j_{t_2+M-1,\underline{w}^2}$ and $\delta((t_1,\underline{w}^1)(t_2,\underline{w}^2))=1$, or
		\item $j_{t_1+M-1,\underline{w}^1}>j_{t_2+M-1,\underline{w}^2}$ and $\delta((t_1,\underline{w}^1)(t_2,\underline{w}^2))=-1$
	\end{itemize}
\end{lemm}

Now, we can establish a formal order for the set of stable segments in $R_i$ that are part of the stable boundary of the rectangles in $\cR_{S(\cW)}$ contained within $R_i$. 

\begin{defi}\label{Lemm: Boundaries code,CW}
	Let $R_i\in\cR$ be any rectangle, denote $I_{i,-1}:=\partial^s_{-1} R_i$ as the lower boundary of $R_i$, and $I_{i,+1}:=\partial^s_{+1} R_i$ as the upper boundary. Define
	$$
	\underline{s}:\cO(i,\cW)\cup \{(i,+1), (i,-1)\} \rightarrow \{0,1,\cdots,O(i,\cW),O(i,\cW)+1\}.
	$$
	as the unique function such that:
	\begin{itemize}
		\item $\underline{s}(t_1,\underline{w}^1)< \underline{s}(t_2,\underline{w}^2)$ if and only if  $I_{t_1,\underline{w}^1}< I_{t_1,\underline{w}^2}$. 
		\item It assigns $0$ to $(i,-1)$ and  $O(i,\cW)+1$ to  $(i,+1)$.
	\end{itemize}
	
\end{defi}

As every rectangle in the partition $\mathcal{R}_{\mathcal{S}(\mathcal{W})}$ that is contained in $R_i$ is bounded by two consecutive stable segments in $\mathcal{I}(i, \mathcal{W})$, the order function uniquely determines a label for the rectangles in the Markov partition in the following manner:

\begin{enumerate}
	\item It is not difficult to see that the number of rectangles in $\mathcal{R}_{S(\mathcal{W})}$ is equal to $O(i, \mathcal{W}) + 1$.
	\item Consider all the stable segments contained in $R_i$, including their boundaries, ordered vertically in $R_i$ using the criterion developed in \ref{Lemm: Comparacion de ordenes, cW} as $\{I_{(i,s)}\}_{s=0}^{O(i, \mathcal{W}) + 1}$.
	\item We assign to the rectangle $R' \in \mathcal{R}_{S(\mathcal{W})}$ contained in $R_i$, and bounded by the consecutive segments $I_{(i,s)}$ and $I_{(i,s+1)}$, with $s \in \{0, \cdots, O(i, \mathcal{W})\}$, the label $R' := R_{(i,s+1)}$ for $s \in \{0, \cdots, O(i, \mathcal{W})\}$.
\end{enumerate}

\begin{coro}\label{Coro: Rs(w) es particion geometrica}
With the labeling of the rectangles in $\mathcal{R}_{S(\mathcal{W})}$ described above, the Markov partition $(f, \mathcal{R}_{S(\mathcal{W})})$ is a geometric refinement of $(f, \mathcal{R})$ by horizontal sub-rectangles.
\end{coro}

\begin{defi}\label{Defi: Geometric s-boundary refinament cW}
	The Markov partition $\cR_{S(\cW)}:=\{\tilde{R}_r\}_{r=1}^N$ with the canonical  geometrization is called the $s$-\emph{boundary refinement} of $\cR$ with respect to the family $\cW$. Its geometric type is denoted by:
		\begin{equation}
		T_{S(\cW)} := \{N,\{H_r,V_r\}_{r=1}^N,\Phi_{S(\cW)}:=(\rho_{S(\cW)},\epsilon_{S(\cW)})\}.
	\end{equation}
	We label the rectangles in $\mathcal{R}_{S(\cW)}$ using the \emph{lexicographic order} \ref{Defi: Orden Lexicodromico} as
$$
\tilde{R}_{(i,s)}:=\tilde{R}_{\tilde{r}(i,s)}.
$$
	\end{defi}

In the following sub-sections, we are going to consider a rectangle $R_r$ and we are going to assume that we know the terms $(i, s)$ such that $r=\tilde{r}(i,s)$. The next lemma indicates how to determine such parameters and justifies our future hypotheses regarding them.

\begin{lemm}\label{lemm: determinating r=(i,s)}
	Let $\tilde{R}_r$ a rectangle in $\cR_{S(\cW)}$ then we can determine two unique numbers $(i_r,s_r)$ such that $r=\tilde{r}(i_r,s_r)$ in following manner:
	\begin{itemize}
		\item[i)] There exist unique $i_r\in \{1,\cdots,n\}$ such that:
		$$
		\sum_{i<i_r} O(i,\cW) +1 < r \leq \sum_{i\leq i_r} O(i,\cW) +1 
		$$
		\item[ii)] Take $s_r:= r - \sum_{i<i_r} O(i,\cW) +1$.
	\end{itemize}
\end{lemm}

\begin{figure}[h]
	\centering
	\includegraphics[width=0.4\textwidth]{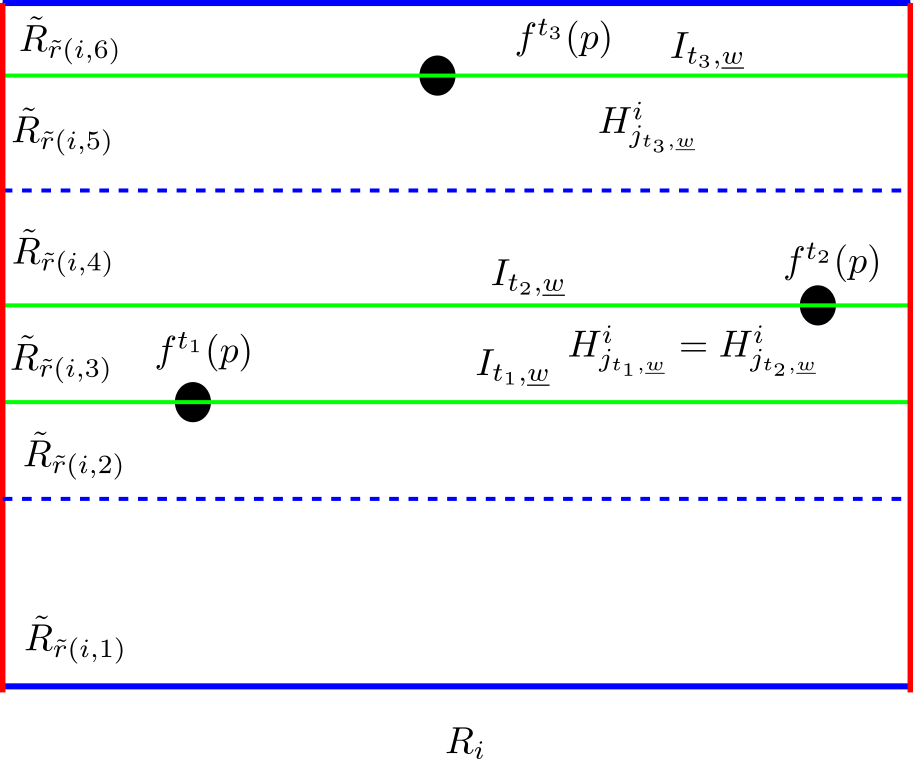}
	\caption{The set $\cI(i,\underline{w})$ and the rectangles $H^i_{ j_{t,\underline{w}}}$ }
	\label{Fig: Inva boundary}
\end{figure}

\begin{proof}
	The only point that we need to argue is the first item, as the second one is totally determined by such a number $i_r$. But since $r \in \{1, \cdots, \sum_{i=1}^n O(i, \cW) + 1\}$, in fact, there exists a maximum $i_0\in\{1,\cdots, n\}$ such that $\sum_{i<i_0} O(i, \cW) + 1 < r$. Clearly $i_r = i + 1$ satisfies the second equality because if this is not the case:  $\sum_{i<i_0+1} O(i, \cW) + 1 < r$ contradicting that $i_0$ is the maximum with this property.
\end{proof}

\subsubsection{The boundary of the refinement $\mathcal{R}_{S(\mathcal{W})}$}

Lemma \ref{Lemm: Determine boundaris of R-r, cW} is a generalization of \ref{Lemm: Determine boundaris of R-r} , and Lemma \ref{Lemm: imge of boundaries,cW} is a generalization of  \ref{Lemm: imge of boundaries}. Their proofs are exactly the same as the results they generalize.

\begin{lemm}\label{Lemm: Determine boundaris of R-r, cW}
	Let $\tilde{R}_r$ be a rectangle in the geometric Markov partition $\cR_{S(\cW)}$ with $r=\tilde{r}(i,s)$. Then, the stable boundary of $\tilde{R}_r$ consists of two stable segments of $R_i$ determined as follows:
	
	\begin{itemize}
		\item[i)] If $s=1$, the lower boundary of $\tilde{R}_r$ is $I_{i,-1}$ and if  $O(i,\cW)=0$,  its upper boundary is $I_{i,+1}$. However, if $O(i,\cW)>0$, there exists a unique $(t,\underline{w}^{q})\in \cO(i,\cW)$ such that $\underline{s}(t,\underline{w}^q)=1$, and the upper boundary of $\tilde{R}_r$ is the stable segment $I_{t,\underline{w}^q}$.
		
		\item[ii)] If $O(i,\cW)>1$, $s \neq 1$ and $s\neq O(i,\cW)+1$, there are unique $(t_1,\underline{w}^1), (t_2,\underline{w}^2)\in \cO(i,\cW)$ such that: $\underline{s}(t_1,\underline{w}^1)=s-1$, $\underline{s}(t_2,\underline{w}^2)=s$. Furthermore, the lower boundary of $\tilde{R}_r$ is the stable segment $I_{t_1,\underline{w}^1}$, and the upper boundary of $\tilde{R}_r$ is the stable segment $I_{t_2,\underline{w}^2}$.
		
		\item[iii)] If $O(i,\cW)>0$ and $s=O(i,\cW)+1$, then the upper boundary of $\tilde{R}_r$ is the stable segment $I_{i,+1}$. Additionally, there exists a unique $(t,\underline{w}^q)\in \cO(i,\cW)$ such that $\underline{s}(t,\underline{w}^q)= O(i,\cW)$, and the inferior boundary of $\tilde{R}_r$ is the stable segment $I_{t,\underline{w}^q}$.
	\end{itemize}
	
\end{lemm}

\begin{lemm}\label{Lemm: imge of boundaries,cW}
	
	Let $\tilde{R}_r$ with $r=\tilde{r}(i,s)$ be a rectangle in the geometric Markov partition $\cR_{S(\cW)}$. Then, the stable boundary components of $\tilde{R}_r$ have  images under $f$ determined in the following manner:
	
	\begin{itemize}
		\item[i)] If the lower boundary of $\tilde{R}_r$ is $I_{i,-1}$ and $\Phi_T(i,1)=(k,l,\epsilon)$, then: if $\epsilon=1$,   $f(I_{i,-1}) \subset I_{k,-1}$, but if $\epsilon=-1$, $f(I_{i,-1}) \subset I_{k,+1}$.

		\item[ii)] If one boundary component of $\tilde{R}_r$ is the stable segment $I_{t,\underline{w}}$, then $f(I_{t,\underline{w}})\subset I_{t+1,\underline{w}}$.
		
		\item[iii)]  If $I_{i,+1}$ is the upper boundary of $\tilde{R}_r$ and $\Phi_T(i,h_i)=(k,l,\epsilon)$, then:  if $\epsilon=1$,   $f(I_{i,+1}) \subset I_{k,+1}$, but if $\epsilon=-1$, $f(I_{i,+1}) \subset I_{k,-1}$.
	\end{itemize}
\end{lemm}

\subsection{The geometric type of \protect$\mathcal{R}_{\mathcal{S}(\mathcal{W})}\protect$}

\subsubsection{The number of rectangles in \texorpdfstring{$\mathcal{R}_{\mathcal{S}(\mathcal{W})}$}{R_S(W)}}

The next lemma is immediate.

\begin{lemm}\label{Lemm: N en TcW}
	For all $i \in \{1, \ldots, n\}$, the number of rectangles in the Markov partition $\cR_{S(\cW)}$ that are contained in $R_i$ is equal to $O(i,\cW)  + 1$ and the number $N$ in the geometric type $T_{S(\cW)}$, equal to the number of rectangles in the $s$-boundary refinement respect to $\cW$, is determined by the following formula:
	\begin{equation}\label{Equa: Number N,cW }
		N=\sum_{i=1}^{n} O(i,\cW)+1.
	\end{equation}
\end{lemm}

\subsubsection{\texorpdfstring{The vertical sub-rectangles in $T_{\mathcal{S}(\mathcal{W})}$}{The vertical sub-rectangles in $T_{S(W)}$}}\label{Subsubsec: Vertical sub-rectangles}

\begin{defi}\label{Defi: Vertical sub R-S(cW)}
	Let $\tilde{R}_r$ be a rectangle in the geometric Markov partition $\cR_{S(\cW)}$ and suppose that $r = \tilde{r}(i, s)$. Let $V^i_l$ be a vertical sub-rectangle of the Markov partition $(f, \cR)$ contained in $R_i$. Define:
	$$
	\tilde{V}^{r}_l:= \overline{\overset{o}{\tilde{R}_r} \cap \overset{o}{V^i_l}}.
	$$
	as the closure of the intersection between the interior of  $\tilde{R}_r$ and the interior of  $V^i_l$.
	
\end{defi}

\begin{lemm}\label{Lemm: Caracterization of vertica sub of R-S(w)}
	The set $\tilde{V}^{r}_l$ is a vertical sub-rectangle of $\tilde{R}_r$
\end{lemm}

\begin{proof}
The intersection of the interiors of a horizontal and a vertical sub-rectangle of $R_i$ has a unique connected component. In particular, $\overset{o}{V^i_l} \cap \overset{o}{\tilde{R}_r}$ has a unique connected component. Clearly, $V^i_l$ crosses $\tilde{R}_r$ from the bottom to the top, and therefore, $\tilde{V}^r_l$ is a vertical sub-rectangle of $\tilde{R}_r$.
\end{proof}

The next is just a technical lemme to be used more late.

\begin{lemm}\label{Lemm: unique sub (i,j) for a r,cW}
Let $\tilde{R}_r \in \mathcal{R}_{S(\mathcal{W})}$ with $r = \tilde{r}(i, s)$. Let $\tilde{H} \subset \tilde{R}_r$ be a horizontal sub-rectangle of $(f, \mathcal{R}_{S(\mathcal{W})})$. Then there is a unique pair $(i, j) \in \mathcal{H}(T)$ such that $\overset{o}{\tilde{H}} \subset \overset{o}{H^i_j}$, where $H^i_j$ is a horizontal sub-rectangle of $(f, \mathcal{R})$.
\end{lemm}

\begin{proof}
We claim that $\tilde{H}$ intersects a unique horizontal sub-rectangle $H^i_j$ in its interior. If this is not the case, there are two adjacent sub-rectangles of $R_i$, $H^i_{j_1}$ and $H^i_{j_2}$, sharing a stable boundary $I$ that lies inside the interior of $\tilde{H}$, i.e., $\overset{o}{I} \subset \overset{o}{\tilde{H}}$. Then, $f(\overset{o}{\tilde{H}})$ will intersect the stable boundary of $\mathcal{R}$ in the stable interval $f(\overset{o}{I})$. However, since the stable boundary of $\mathcal{R}$ is contained in the stable boundary of $\mathcal{R}_{S(\mathcal{W})}$, the interval $f(\overset{o}{I})$ is contained in the stable boundary of $\mathcal{R}_{S(\mathcal{W})}$, leading to a contradiction. Therefore, $\overset{o}{\tilde{H}}$ is contained in the interior of a unique horizontal sub-rectangle $H^i_j$.
\end{proof}

\begin{lemm}\label{Lemma: Vertical sub R-ps, cW}
The rectangles $\{\tilde{V}^r_l\}_{l=1}^{v_i}$ introduced in Definition \ref{Defi: Vertical sub R-S(cW)} are the vertical sub-rectangles of the Markov partition $(f, \mathcal{R}_{S(\mathcal{W})})$ that are contained in $\tilde{R}_r$.
\end{lemm}

\begin{proof}
	
Consider a single rectangle $\tilde{V}^r_l$. We can observe that $f^{-1}(\tilde{V}^r_l)$ is a subset of $f^{-1}(V^k_l)$, and by Lemma \ref{Lemm: unique sub (i,j) for a r,cW}, it is contained in some horizontal sub-rectangle of $\mathcal{R}$, $H^i_j$. The preimage $f^{-1}(\tilde{V}^r_l)$ is a horizontal sub-rectangle of $H^i_j$ whose interior does not intersect $\partial^s \mathcal{R}_{S(\mathcal{W})}$, and it is exclusively contained within a unique rectangle $\tilde{R}_r$. Moreover, the stable boundary of $f^{-1}(\tilde{V}^r_l)$ has an image contained within the stable boundary of the Markov partition $\mathcal{R}_{S(\mathcal{W})}$. Therefore, $f^{-1}(\tilde{V}^r_l)$ is a horizontal sub-rectangle of $\mathcal{R}_{S(\mathcal{W})}$, and hence $\tilde{V}^r_l$ is a vertical sub-rectangle of the Markov partition $\mathcal{R}_{S(\mathcal{W})}$.

	Conversely, if $\tilde{V}$ is a vertical sub-rectangle of $\cR_{S(\cW)}$, then $f^{-1}(\tilde{V})=\tilde{H}$ is a horizontal sub-rectangle of $\cR_{S(\cW)}$. Let $(i,j_{\tilde{H}})$ be the indexes give in Definition \ref{Lemm: unique sub (i,j) for a r,cW}. Thus, $\tilde{V}$ is a horizontal sub-rectangle of $V^k_l=f(H^i_{j_{\tilde{H}}})$, and its interior does not intersect the stable boundaries of $\cR_{S(\cW)}$. However, its stable boundaries are contained in $\partial^s \cR_{S(\cW)}$. This implies that $\tilde{V}\subset V^k_l\cap \tilde{R}_{r}$ for some $r\in \{1,\cdots, n\}$, and it turns out to be one of the rectangles $\tilde{V}^{r}_l$ given by $\tilde{V}^{r}_l:= \overline{\overset{o}{\tilde{R}_r} \cap \overset{o}{V^k_l}}$.
	
\end{proof}

This result  have the next immediate corollary.

\begin{coro}\label{Coro: number of vertical sub rec, cW}
	If $r=\tilde{r}(i,s)$  then $V_r=v_i$.
\end{coro}

\begin{coro}\label{Coro: vertical possition is delimitate,cW}
	Let $\tilde{H}\subset \tilde{R}_r$ any horizontal sub-rectangle of $\cR_{S(\cW}$,  assume that $\tilde{H}\subset H^i_j$ and $f(H^i_j)=V^k_l$. If $f(H)=\tilde{V}^{r'}_{l'}$ then $l=l'$.
\end{coro}

\begin{proof}
	As $\tilde{H}\subset H^i_j$, then $f(\tilde{H})\subset f( H^i_j)=V^k_l$. Simultaneously, $f(H)=\tilde{V}^{r'}_{l'}=\overline{ \overset{o}{\tilde{R}_{r'}}\cap \overset{o}{V^k_{l}} }$. Given that the intersection $\overset{o}{\tilde{R}_{r'}}\cap \overset{o}{V^k_{l}}$ has only one connected component, it is necessary that $\overset{o}{V^k_{l}}=\overset{o}{V^k_{l'}}$ and thus $l=l'$.
\end{proof}

\subsubsection{The relative position of $H^i_j$ with respect to $\tilde{R}_r$} We are going to determine the set of horizontal sub-rectangles of $\mathcal{R}$ that intersect a certain rectangle $R_r$ from $\mathcal{R}_{S(\mathcal{W})}$. For this reason, let us introduce the set that indexes such sub-rectangles.

\begin{defi}\label{Defi cH(r) set, cW}
	Let $\tilde{R}_r\in \cR_{S(\cW)}$ with $r=\tilde{r}(i,s)$. We define:
	$$
	\cH(r):=\{(i,j)\in \cH(T): \overset{o}{H^i_{j}}\cap \overset{o}{\tilde{R}_r}\neq \emptyset\}.
	$$
\end{defi}

If we know the stable segments that bound $\tilde{R}_r$, we can determine the indices of the horizontal sub-rectangles within $\cR$ that contain the upper and lower boundaries of $\tilde{R}_r$. Let's assume  these horizontal sub-rectangles are in positions $j{-} < j{+}$. Thus, we can define $\cH(r)$ as the set of indices $(i, j)$ where $j$ ranges from $j_{-}$ to $j_{+}$. The following lemma provides a more detailed explanation of this observation.

\begin{figure}[h]
	\centering
	\includegraphics[width=0.6\textwidth]{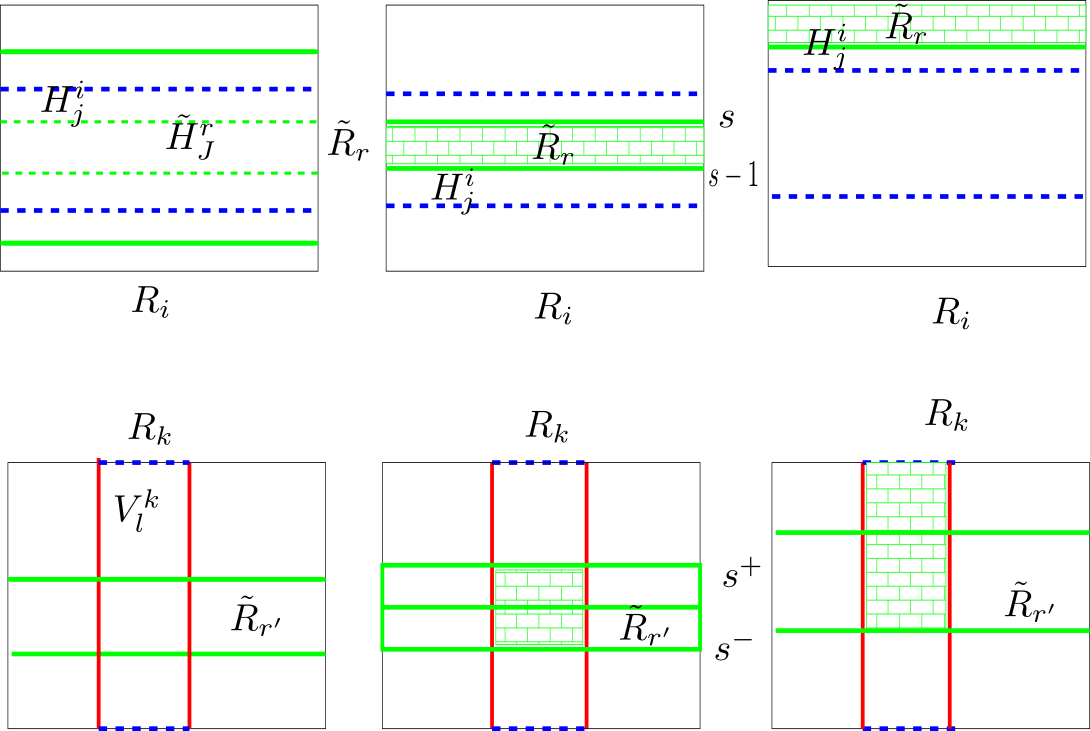}
	\caption{The relative position of the horizontal sub-rectangles inside $\tilde{R}_{r}$}
	\label{Fig: 3 cases}
\end{figure}

\begin{lemm}\label{Lemm: determination cH(r), cW}
	Let $\tilde{R}_r$ be a rectangle in the geometric Markov partition $\cR_{S(\underline{w})}$ with $r=\tilde{r}(i,s)$ and $s\in {0,\cdots, O(i,\cW)}$. Then the set $\cH(r)$ is determined by one of the three following formulas:
	
	\begin{itemize}
		\item[i)] If $O(i,\cW)=0$, then:  $\cH(r)=\{(i,j)\in \cH(T)\}$.

		\item[ii)]  If $O(i,\cW)>0$, $s \neq 0$ and $s \neq O(i,\cW)$, let $I_{t_1,\underline{w}^1}$ be the inferior boundary of $\tilde{R}_r$, and let $I_{t_2,\underline{w}^2}$ be the upper boundary of $\tilde{R}_r$. Let $j_{t_1,\underline{w}^1}$ and $j_{t_1,\underline{w}^2}$ be the indices determined in Lemma \ref{Lemm: unique hriwontal for every prjection}. In this situation:
		$$
		\cH(r)=\{(i,j)\in \cH(T): j_{t_1,\underline{w}^1} \leq j \leq j_{t_2,\underline{w}^2} \}.
		$$
		
		\item[iii)] If $O(i,\cW)>0$ and $s=O(i,\cW)$, let $I_{t_1,\underline{w}^1}$ be the inferior boundary of $\tilde{R}_r$ and  $I_{i,+1}$  the upper boundary of $\tilde{R}_r$. Let $j_{t_1,\underline{w}^1}$ be the index determined in \ref{Lemm: unique hriwontal for every prjection}. In this manner:
		$$
		\cH(r)=\{(i,j)\in \cH(T): j_{t_1,\underline{w}^1} \leq j \leq h_i \}.
		$$
		\item[iv)] If $O(i,\cW)>0$ and $s=0$, let $I_{t_1,\underline{w}^1}$ be the upper boundary of $\tilde{R}_r$, let $I_{i,-1}$ be its inferior boundary and let $j_{t,\underline{w}^1}$ be the index determined in \ref{Lemm: unique hriwontal for every prjection}. In this manner:
		$$
		\cH(r)=\{(i,j)\in \cH(T): 1 \leq j \leq j_{t_1,\underline{w}^1} \}.
		$$
		
	\end{itemize}
	
\end{lemm}

\begin{proof}
	\textbf{Item} $i)$. If $O(i,\cW)=0$, then $\tilde{R}_r=R_i$, indicating that $\cH(r)$ must include the labels of all the horizontal sub-rectangles within the Markov partition $\cR$ that are contained in $R_i$. Thus, $\cH(r)=\{(i,j)\in \cH(T)\}$.

\textbf{Item} $ii)$. The lower boundary of $\tilde{R}_r$ lies within $H^i_{j_{t_1, \underline{w}^1}}$, while the upper boundary is contained in $H^i_{j_{t_2, \underline{w}^2}}$. It is evident that all the horizontal sub-rectangles of the form $H^i_j$ with $j_{t_1, \underline{w}} \leq j \leq j_{t_2, \underline{w}}$ are the only ones that satisfy $\overset{\circ}{H^i_j} \cap \overset{\circ}{\tilde{R}_r} \neq \emptyset$. This implies that $\mathcal{H}(r) = \{ (i, j) \in \mathcal{H}(T) \mid j_{t_1, \underline{w}} \leq j \leq j_{t_2, \underline{w}} \}$.

\textbf{Item} $iii)$. Every horizontal sub-rectangle $H^i_j$ with $j_{t_1,\underline{w}}\leq  j \leq \leq h_i$ satisfies that $\overset{o}{H^i_{j}}\cap \overset{o}{\tilde{R}_r}\neq \emptyset$, and these are the only horizontal sub-rectangles within $\cR$ with this property. Thus, $\cH(r)=\{(i,j)\in \cH(T): j_{t,\underline{w}} \leq j \leq h_i \}$.

\textbf{Item} $iv)$. The argument is the same as in Item $iii)$.
\end{proof}

\subsubsection{The number of horizontal sub-rectangles of $\cR_{S(\cW)}$}
We are going to count how many rectangles in the Markov partition $\cR_{S(\cW)}$ are contained in the intersection $\tilde{R}_r\cap H^i_j$ and describe the way that $f$ send such horizontal sub-rectangles into vertical sub-rectangles. 

\begin{center}
\textbf{Fist case:} $\tilde{R}_r \subset H^i_{j}\setminus\partial^s H^i_{j}$ 
\end{center}

\begin{figure}[h]
	\centering
	\includegraphics[width=0.5\textwidth]{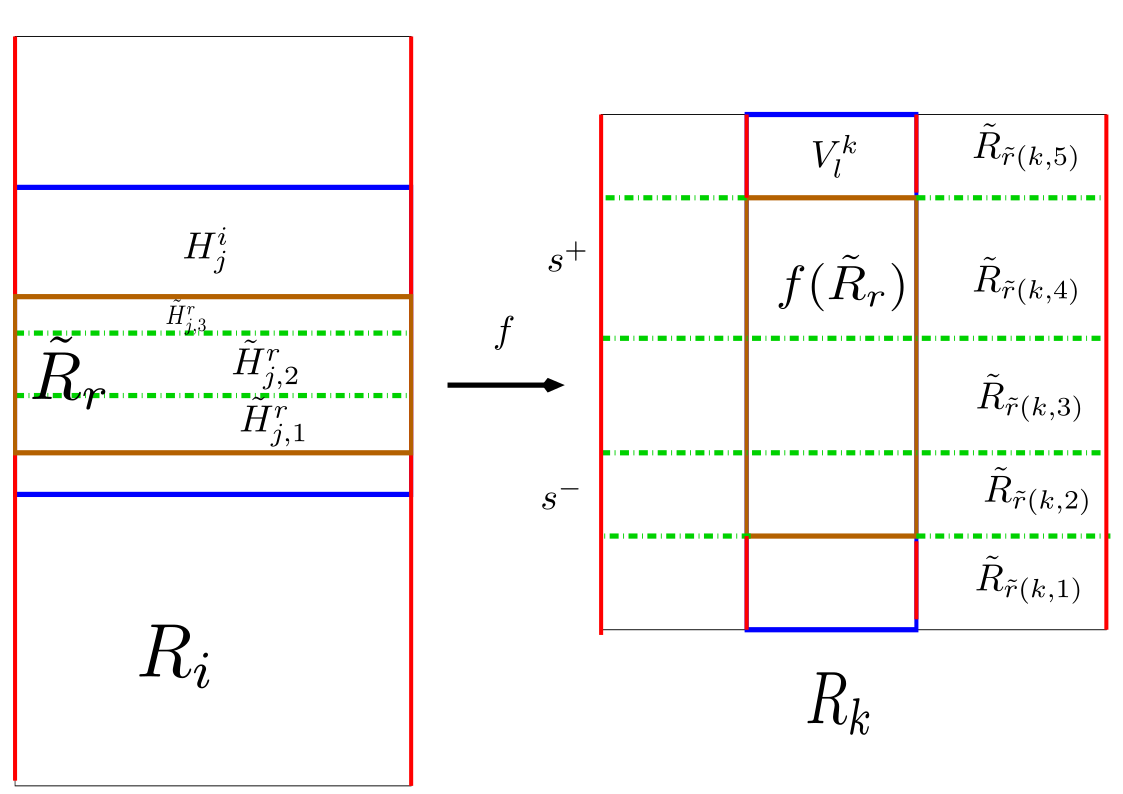}
	\caption{First case: $\tilde{R}_r \subset H^i_{j}\setminus\partial^s H^i_{j}$ }
	\label{Fig: Caso I}
\end{figure}

\begin{lemm}\label{Lemm: number horizontal sub case I, cW}
	Let $\tilde{R}_r\in \cR_{S(\cW)}$, with $r=\tilde{r}(i,s)$. Suppose we have a pair of indices $(i,j)\in \cH(r)$ such that $\tilde{R}_r \subset H^i_j\setminus\partial^s H^i_j$. In this case $s\neq $ and $s\neq O(i,\cW)$ and we take $I_{t_1,\underline{w}^1}$ as the lower boundary of $\tilde{R}_r$ and $I_{t_2,\underline{w}^2}$ as the upper boundary of $\tilde{R}_r$.
	
	Assume that $\rho_T(i,j)=(k,l)$. Let $s^-=\underline{s}(t_1+1,\underline{w}^1)$ and $s^+=\underline{s}(t_2+1,\underline{w}^2)$. Then,  $w^1_{t_1+1}=w^2_{t_2+1}=k$ and the number of horizontal sub-rectangles of $\cR_{S(\cW)}$ contained in the intersection $\tilde{R}_r\cap H^i_j$ is given by the formula:
	
	\begin{equation}\label{Equa: h (i,j,r) caso I, cW}
		\underline{h}(r,j)= \vert s^- - s^+\vert.
	\end{equation} 
	
\end{lemm}

\begin{proof}
Look at Figure \ref{Fig: Caso I}. Since $\tilde{R}_r \subset H^i_j \setminus \partial^s H^i_j$, $\tilde{R}_r$ doesn't share a stable boundary component with $R_i$. This means $s \neq 0$ and $s \neq O(i, \mathcal{W})$, and we can take the intervals $I_{t_1,\underline{w}^1}$ and $I_{t_2,\underline{w}^2}$ as described in the lemma.

Since $f(\tilde{R}_r) \subset f(H^i_j) \subset R_k$, then $w_{t_1+1}^1 = w_{t_2+1}^2 = k$, as we claimed. The number of horizontal sub-rectangles of $\tilde{R}_r$ is equal to the number of all the rectangles in $\mathcal{R}_{S(\mathcal{W})}$ contained in $R_k$ and located between the intervals $I_{t_1+1,\underline{w}^1}$ and $I_{t_2+2,\underline{w}^2}$. If their relative positions are given by the numbers $s^-$ and $s^+$, then the number of such rectangles is given by $\vert s^- - s^+ \vert$, and we have proved our lemma.
\end{proof}

\begin{defi}\label{Defi: order rectangles case I, cW}
	The horizontal sub-rectangles of $\tilde{R}_r$ as described in Lemma \ref{Lemm: number horizontal sub case I, cW} are labeled from the bottom to the top with respect to the orientation of $\tilde{R}_r$ as:
	$$
	\{\tilde{H}^r_{j,J}\}_{J=1}^{\underline{h}(r,j)}.
	$$
\end{defi}

\begin{lemm} \label{Lemm: Permutation rho caso I, cW}
	With the hypothesis and notations of Lemma \ref{Lemm: number horizontal sub case I, cW}, assume that $(\rho,\epsilon)(i,j)=(k,l,\epsilon(i,j))$ and $f(\tilde{H}^r_{j,J})=\tilde{V}^{r'}_{l'}$. Then:
	\begin{itemize}
		\item[i)] 	If $\epsilon=\epsilon(i,j) = 1$, then $s^- < s^+$, $r'=\tilde{r}(k, s^- + J)$ and $l=l'$.
		\item[ii)] If $\epsilon=\epsilon(i,j) = -1$, then $s^- > s^+$, $r'=\tilde{r}(k, s^- - J+1)$ and $l=l'$.
	\end{itemize}
where  $\underline{s}:\cO(i,\cW)\cup \{(i,+1), (i,-1)\} \rightarrow \{0,1,\cdots,O(i,\cW),O(i,\cW)+1\}.$ is the order function that was introduced in Lemma \ref{Lemm: Boundaries code,CW}.
\end{lemm}

\begin{proof}
	\textbf{Item} $i)$. If $\epsilon(i,j) = 1$ then $f$ preserves the vertical orientation of $\tilde{R}_r$, and then $I_{t_1+1,\underline{w}^1} < I_{t_2+1,\underline{w}^2}$, so  $s^- < s^+$.
	
The rectangle $\tilde{H}^r_{j,J}$ is in vertical position $J$ inside $\tilde{R}_r \cap H^i_j$ and is mapped by $f$ to the rectangle whose vertical position inside $R_k$ is $J$ positions above the rectangle in position $s^-$ within $R_k$. Therefore, $f(\tilde{H}^r_{j,J})$ is inside a rectangle $\tilde{R}_{r'}\subset R_k$ whose position  is equal to $s^- + J$. In other words, $r'=\tilde{r}(k,s^-+J)$.
	
	\textbf{Item} $ii)$. If $\epsilon(i,j) = -1$ then $f$ inverts the vertical orientation of $\tilde{R}_r$. Then, $I_{t_1+1,\underline{w}^1} > I_{t_2+1,\underline{w}^2}$ and $s^- > s^+$.
	
	The rectangle $\tilde{H}^r_{j,1}$ that is positioned at place $1$ inside $\tilde{R}_r\cap H^i_j$ is sent into to the rectangle $\tilde{R}_{r'}$, whose upper boundary is at position $s^-$, therefore $r'=s^-=s^- -(J-1)$. In general the upper boundary of the rectangles $\tilde{R}\subset R_k$ determine its position inside $R_k$. That means that  $f(\tilde{H}^r_{j,J})$ is inside a rectangle $\tilde{R}_{r'}$ withing  $R_k$ whose position  is given by $s^- - J+1$. In other words, $r'=\tilde{r}(k,s^- -J+1)$.
	
	Finally in bot cases,  Corollary \ref{Coro: vertical possition is delimitate,cW} implies that $l'=l$.
\end{proof}

This lets us to introduce the following definition.

\begin{defi}\label{Defi: permutation caso I, cW}
	With the hypothesis of Lemma \ref{Lemm: number horizontal sub case I, cW}, assume that $\Phi_T(i,j)=(k,l,\epsilon)$. Then: 
	
	If $\epsilon(i,j) = 1$ we define:
	
	\begin{equation}\label{Equ: rho caso 1, +,cW}
		\underline{\rho}_{(r,j)}(r,J)=(\tilde{r}(k, s^-+J), l  ).
	\end{equation}
	
	and if $\epsilon(i,j) = -1$:
	\begin{equation}\label{Equ: rho caso 1, -, cW}
		\underline{\rho}_{(r,j)}(r,J)=(\tilde{r}(k, s^+  -J+1) , l    ).
	\end{equation}
\end{defi}

\begin{center}
\textbf{{Second case:} $H^i_{j}\subset \tilde{R}_r$ }
\end{center}
		\begin{figure}[h]
	\centering
	\includegraphics[width=0.5\textwidth]{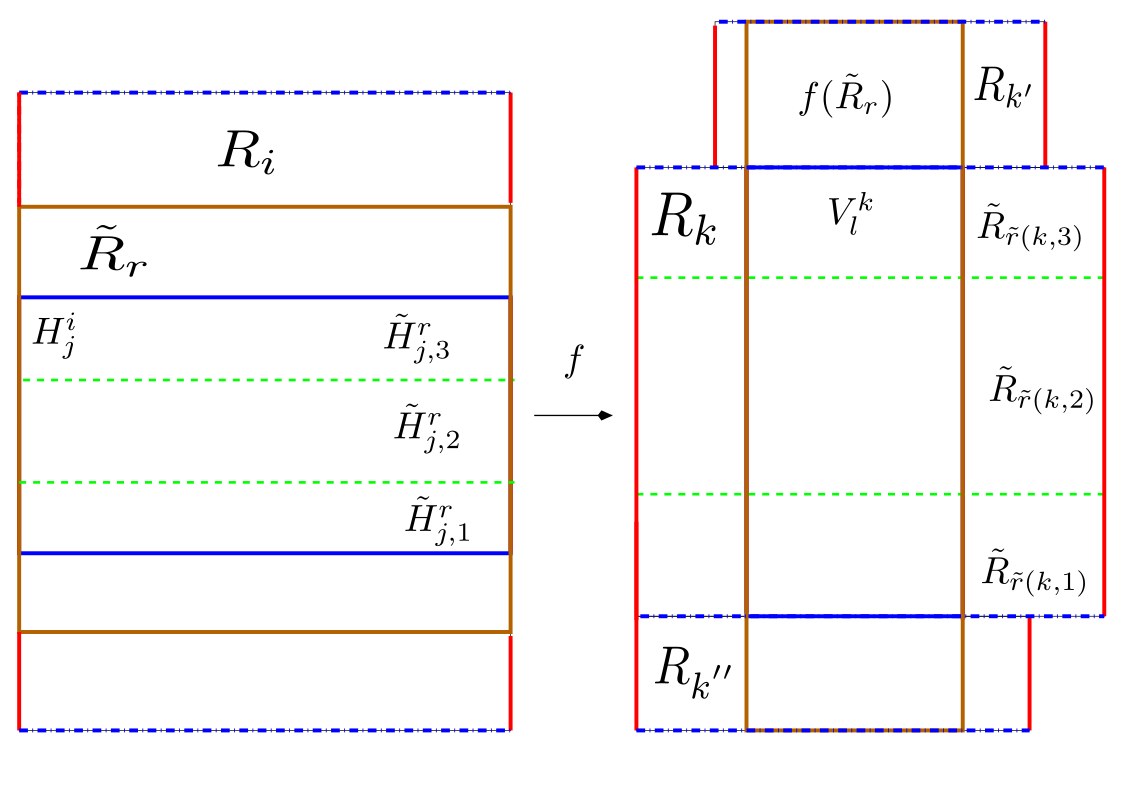}
	\caption{Second case: $H^i_{j}\subset \tilde{R}_r$  }
	\label{Fig: Caso II}
\end{figure}

\begin{lemm}\label{Lemm: number horizontal sub case II, cW}
	
	Let $\tilde{R}_r\in \cR_{S(\cW)}$, with $r=\tilde{r}(i,s)$. Let $(i,j)\in \cH(r)$ be a pair of indices such that $H^i_j\subset \tilde{R}_r$, assume that $\rho_T(i,j)=(k,l)$, then the number of horizontal sub-rectangles of $\cR_{S(\cW)}$ that are contained in the intersection $\tilde{R}_r\cap H^i_j$ is given by the formula:
	
	\begin{equation}\label{Equa: h (i,j,r) caso I,cW}
		\underline{h}(r,j)=O(k,\cW)+1,
	\end{equation} 
	
\end{lemm}

\begin{proof}
	The image of $f^(H^i_j\cap \tilde{R}_r)=f(H^i_j)=V^k_l$ and $V^k_l$ intersect all the rectangles of $\cR_{S(\cW)}$ that are contained in $R_k$, the total amount of them is $O(k,\cW)+1,$ and this is the number of horizontal sub-rectangles of $\cR_{S(\cW)}$ that are contained in $\tilde{R}_r\cap H^i_j$.
\end{proof}

\begin{defi}\label{Defi: order rectangles case II,cW}
The horizontal sub-rectangles of $\tilde{R}_r$ described in Lemma \ref{Lemm: number horizontal sub case II, cW} are labeled from bottom to top with respect to the vertical orientation of $\tilde{R}_r$ as:
$$
\{\tilde{H}^r_{j,J}\}_{J=1}^{\underline{h}(r,j)}.
$$
\end{defi}

Now we can determine the image of any horizontal sub-rectangle.

\begin{lemm}\label{Lemm: Permutation rho caso I,cW}
	Let $\tilde{R}_r\in \cR_{S(\cW)}$, with $r=\tilde{r}(i,s)$. Let $(i,j)\in \cH(r)$ be a pair of indices such that $H^i_j\subset \tilde{R}_r$, and let $\tilde{H}^r_{(j,J)}$ be  an horizontal sub-rectangle of $\tilde{R}_r\cap H^i_j$. Assume that $\rho_T(i,j)=(k,l)$,  and that $f(\tilde{H}^r_{j,J})=\tilde{V}^{r'}_{l'}$, then:	
	\begin{itemize}
		\item[i)] If $\epsilon(i,j)=1$, then $r'=\tilde{r}(k,J)$ and $l=l'$.
		
		\item[ii)]  If  $\epsilon_T(i,j)=-1$, then $r'=\tilde{r}(k,O(k,\cW)+1 -J+1)$ and $l=l'$.
	\end{itemize}
	
\end{lemm}

		\begin{proof}
	\textbf{Item} $i)$. If $\epsilon(i,j)=1$, the vertical orientation is preserved by $f$ restricted to $\tilde{R}_r$, and the horizontal sub-rectangle $\tilde{H}^r_{j,J}$ that is in position $J$ inside $\tilde{R}_r\cap H^i_j$ is sent to the rectangle $\tilde{R}_{r'}$ located in position $J$ inside $R_k$. This means that $r'=\tilde{r}(k,J)$ and in view of Corollary \ref{Coro: vertical possition is delimitate,cW},  $l'=l$.
	
	\textbf{Item} $ii)$. If $\epsilon(i,j)=-1$, the vertical orientation is invert by $f$ restricted to $\tilde{R}_r\cap H^i_j$, and the horizontal sub-rectangle $\tilde{H}^r_{j,J}$  that is in position $J$ inside $\tilde{R}_r\cap H^i_j$ is sent to the rectangle $\tilde{R}_{r'}$ located in position $J-1$ below the top rectangle of $\cR_{S(\underline{w})}$ (for $J=1$ the position is $O(k,\underline{w}))+1)$, and this position is equal to $O(k,\underline{w})+1-J+1$ inside $R_k$. This means that $r'=\tilde{r}(k,O(k,\underline{w})+1-J+1)$. Finally,  by  Corollary \ref{Coro: vertical possition is delimitate,cW},  $l'=l$.

\end{proof}

\begin{defi}\label{Defi: permutation caso II, cW}
	Let $\tilde{R}_r\in \cR_{S(\cW)}$, with $r=\tilde{r}(i,s)$. Let $(i,j)\in \cH(r)$ be a pair of indices such that $H^i_j\subset \tilde{R}_r$, and let $\tilde{H}^r_{j,J}$ be  an horizontal sub-rectangle of $\tilde{R}_r\cap H^i_j$. Assume that $\rho_T(i,j)=(k,l)$, then:  
	
	If  $\epsilon(i,j)=1$ we define:
	\begin{equation}\label{Equ: rho caso 2, +,cW}
		\underline{\rho}_{(r,j)}(r,J)=(\tilde{r}(k,J),l),
	\end{equation}
	
	and  if $\epsilon(i,j)=-1$ 
	
	\begin{equation} \label{Equ: rho caso 2, -, cW}
		\underline{\rho}_{(r,j)}(r,J)=(\tilde{r}(k,O(k,\cW)+1 -J+1),l).
	\end{equation}
\end{defi}

\begin{center}
\textbf{Third case:} $H^i_{j}$ contains only one horizontal boundary component of $\tilde{R}_r$.
\end{center}

\begin{figure}[h]
	\centering
	\includegraphics[width=0.8\textwidth]{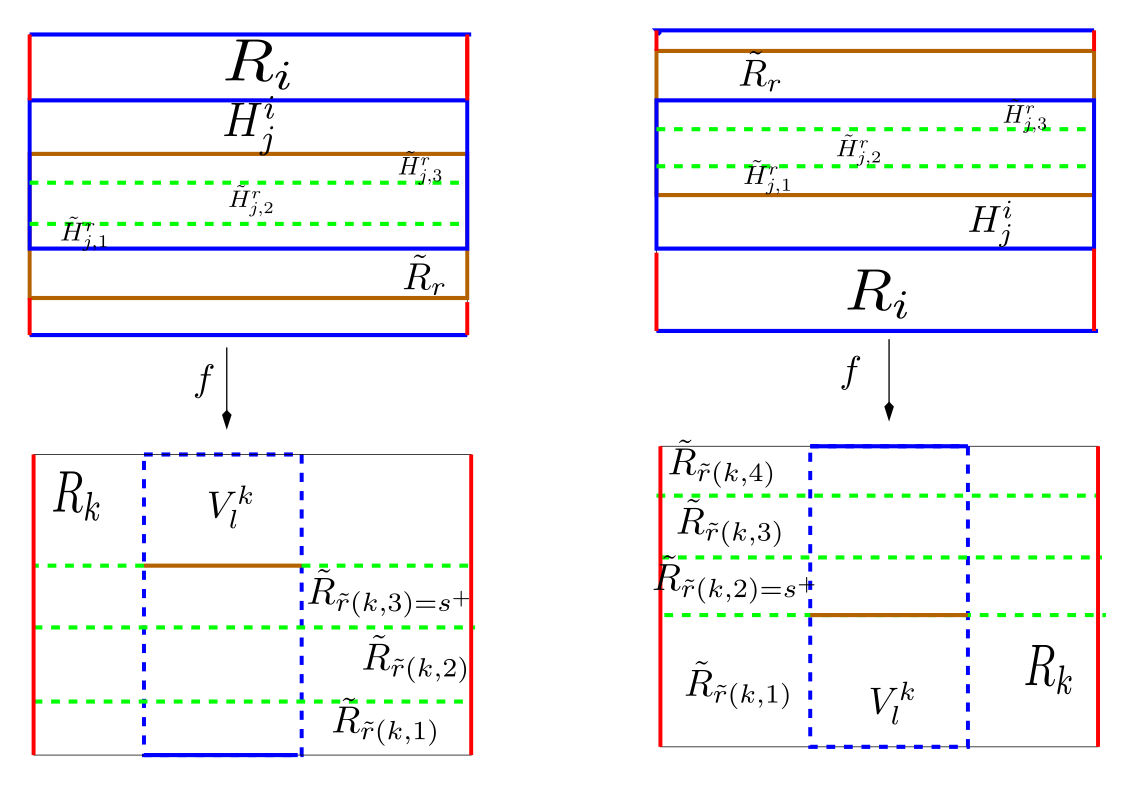}
	\caption{$H^i_{j}$ contains only one horizontal boundary component of $\tilde{R}_r$  }
	\label{Fig: Caso III}
\end{figure}

\begin{lemm}\label{Lemm: number horizontal sub case III, upper, cW}
	
	Consider $\tilde{R}_r\in \cR_{S(\cW)}$, where $r=\tilde{r}(i,s)$, and assume that $O(i,\cW)>0$. 	Take a pair of indexes $(i,j)\in \cH(r)$ such that $H^i_j$ exclusively contains the upper boundary of $\tilde{R}_r$. Let $I_{t,\underline{w}^1}$ (with $\underline{w}^1\in \cW$) be the upper boundary of $\tilde{R}_r$. Let $s^+=\underline{s}(t+1,\underline{w}^1)$ and assume that $\rho_T(i,j)=(k,l)$. Then, the number of horizontal sub-rectangles of  $\cR_{S(\cW)}$ that are contained in  $\tilde{R}_r\cap H^i_j$ can be computed using one of the following formulas:
	
	\begin{itemize}
		\item If $\epsilon_T(i,j)=1$:
		
		\begin{equation}\label{Equ: h(i,j,r), caso 3, +, up,cW}
			\underline{h}(r,j)=s^+
		\end{equation} 
		
		\item If $\epsilon_T(i,j)=-1$
		\begin{equation}\label{Equ: h(i,j,r), caso 3, -,up,cW}
			\underline{h}(r,j)=O(k,\cW)+1-s^+ .
		\end{equation} 
	\end{itemize}
	
\end{lemm}

 	Look the  Figure  \ref{Fig: Caso III} to get some pictorial intuition. about our arguments
\begin{proof}
		
	\textbf{Case} $\epsilon_T(i,j)=1$. The upper boundary of $\tilde{R}_r$ is mapped to the stable segment $I_{t+1,\underline{w}^1}$, positioned at place $s^+$ inside $R_k$ and the inferior  boundary of $\tilde{R}_r\cap H^i_j$ coincides with the inferior  boundary of $H^i_j$ and since $\epsilon_T(i,j)=1$, $f$ preserves the vertical orientation of $\tilde{R}_r\cap H^i_j$. Consequently, the image of the inferior  boundary of $\tilde{R}_r\cap H^i_j$ is contained in the inferior boundary of $R_k$. The number of rectangles in $\cR_{S(\cW)}$ between the segment $I_{t+1,\underline{w}^1}$ and the inferior boundary of $R_k$ is given by $s^+$ and this number is equal to the number of horizontal sub-rectangles of $\tilde{R}_r \cap H^i_j$.
	
	\textbf{Case} $\epsilon_T(i,j)=-1$. The upper boundary of $\tilde{R}_r\cap H^i_j$ is equal to the upper boundary of $\tilde{R}_r$ and  is mapped by $f$ to the stable segment $I_{t+1,\underline{w}^1}$, positioned at place $s^+$ inside $R_k$. The inferior boundary of $\tilde{R}_r\cap H^i_j$ coincides with the inferior boundary of $H^i_j$ , but since $\epsilon_T(i,j)=-1$, $f$ inverts the vertical orientation of $\tilde{R}_r\cap H^i_j$. Consequently, the image of the inferior  boundary of $\tilde{R}_r\cap H^i_j$ is contained in the upper boundary of $R_k$. The number of rectangles in $\cR_{S(\cW)}$ between the segment $I_{t+1,\underline{w}}$ and the upper boundary of $R_k$ is given by $O(k,\cW)+1-s^+$ and this number is equal to the number of horizontal sub-rectangles of $\tilde{R}_r \cap H^i_j$.
	
\end{proof}

	 The next lemma  addresses the situation in which the rectangle $H^i_j$ just contains the inferior boundary of $\tilde{R}_r$. Its proof follows a similar approach to that of Lemma \ref{Lemm: number horizontal sub case III, upper, cW}, so we won't delve into further details.

\begin{lemm}\label{Lemm: number horizontal sub case III, lower, cW}
	
	Consider $\tilde{R}_r\in \cR_{S(\cW)}$, where $r=\tilde{r}(i,s)$, and assume that $O(i,\cW)>0$.	Take a pair of indexes $(i,j)\in \cH(r)$ such that  $H^i_j$ just contains the lower boundary of $\tilde{R}_r$. Let $I_{t,\underline{w}^1}$ (with $\underline{w}^1\in \cW$ ) be the lower boundary of $\tilde{R}_r$, and denote $s^+=\underline{s}(t+1,\underline{w}^1)$. In this scenario, the number of horizontal sub-rectangles of  $\cR_{S(\cW)}$ that are contained in  $\tilde{R}_r\cap H^i_j$ can be determined using one of the following formulas:
	
	\begin{itemize}
		\item If $\epsilon_T(i,j)=1$:
		
		\begin{equation}\label{Equ: h(i,j,r), caso 3, 1, low cW}
			\underline{h}(r,j)=O(k,\cW)+1-s^+
		\end{equation} 
		
		\item If $\epsilon_T(i,j)=-1$
		\begin{equation}\label{Equ: h(i,j,r), caso 3, -1, low cW}
			\underline{h}(r,j)= s^+ .
		\end{equation} 
	\end{itemize}
\end{lemm}

\begin{defi}\label{Defi: order rectangles case III,cW}
	The horizontal sub-rectangles of $\tilde{R}_r$ described in Lemma \ref{Lemm: number horizontal sub case III, upper, cW}  and Lemma \ref{Lemm: number horizontal sub case III, lower, cW}  are labeled from the bottom to the top with the vertical orientation of $\tilde{R}_r$ as:
	$$
	\{\tilde{H}^r_{j,J}\}_{J=1}^{\underline{h}(r,j)}.
	$$
\end{defi}

\begin{lemm} \label{Lemm: Permutation rho caso III, upper cW}
	With the hypothesis of Lemma \ref{Lemm: number horizontal sub case III, upper, cW}, suppose that  $H^i_j$  just contains the upper boundary of $\tilde{R}_r$ and assume that $\Phi_T(i,j)=(k,l,\epsilon_T(i,j))$ and $f(\tilde{H}^r_{j,J})=\tilde{V}^{r'}_{l'}$.	Then we have the following situations: 
	\begin{itemize}
		\item[i)] 	If $\epsilon_T(i,j)=1$, then $r'=\tilde{r}(k,  J)$ and $l=l'$

		\item[ii)] If $\epsilon_T(i,j)=-1$, then $r'=\tilde{r}(k, O(k,\cW)+1-J+1)$ and $l=l'$.
	\end{itemize}	
\end{lemm}

 \begin{proof}
	\textbf{Item} $i)$. When $\epsilon_T(i,j) = 1$, $f$ preserves the vertical orientation  restricted to $\tilde{R}_r\cap H^i_j$. Consequently, the horizontal sub-rectangle $\tilde{H}^r_{j,J}$ that is located at position $J$ within $\tilde{R}_r\cap H^i_j$, is send to the rectangle $\tilde{R}_{r'}$ that is located at position $J$ over the rectangle that is in the position $s^+$ within $R_k$, therefore $r'=\tilde{r}(k,s^+ + J)$.
	
	\textbf{Item}  $ii)$. When $\epsilon_T(i,j) = -1$,  $f$ inverts the vertical orientation  of $\tilde{R}_r\cap H^i_j$. Consequently, the horizontal sub-rectangle $\tilde{H}^r_{j,J}$ that is located at position $J$ within $\tilde{R}_r\cap H^i_j$, is send to the rectangle located at position $J-1$ below the upper rectangle in $R_k$, this is given by  $O(k,\underline{w})+1-(J-1)$ withing $R_k$. Therefore, $r'=\tilde{r}(k,O(k,\underline{w})+1-J+1)$. 
	
	In each case Corollary \ref{Coro: vertical possition is delimitate,cW} implies that $l'=l$.
	
\end{proof}

A similar statement is true when the rectangle $H^i_j$ contains only the lower boundary of $\tilde{R}_r$.

\begin{lemm}\label{Lemm: Permutation rho caso III, lower cW}
	With the hypothesis of Lemma \ref{Lemm: number horizontal sub case III, lower, cW}, suppose that  $H^i_j$  just contains the lower boundary of $\tilde{R}_r$ and assume that $\Phi_T(i,j)=(k,l,\epsilon_T(i,j))$ and $f(\tilde{H}^r_{j,J})=\tilde{V}^{r'}_{l'}$.	Then we have the following situations: 
	
	\begin{itemize}
		\item[i)] 	If $\epsilon_T(i,j)=1$, then $r'=\tilde{r}(k, s^+ + J)$ and $l=l'$

		\item[ii)] If $\epsilon_T(i,j)=-1$, then $r'=\tilde{r}(k, s^+-J+1)$ and $l=l'$.
	\end{itemize} 	
\end{lemm}

\subsubsection{Collecting the information}

Now, we can collect the information from the three cases to compute the number of horizontal sub-rectangles in $\tilde{R}_r$.

\begin{coro}\label{Coro: Number of horizontal sub, cW}
	Let $\tilde{R}_r\in \cR_{S(\cW)}$ with $r=\tilde{r}(i,s)$, then the number $H_r$ of horizontal sub-rectangles of the Markov partition $(f,\cR_{S(\cW)})$ that are contained in $\tilde{R}_r$ is given by:
	\begin{equation}
		H_r=\sum_{\{j:(i,j)\in \cH(r)\}}\underline{h}(r,j).
	\end{equation}
\end{coro}

\begin{center}
	\textbf{The function $\rho_{S(\cW)}$ in the geometric type $T_{S(\cW)}$}
\end{center}

\begin{defi}\label{Defi: horizontal sub in R-r, cW}
	The sub-rectangles of the geometric Markov partition $\cR_{S(\cW)}$ contained in $\tilde{R}_r$ are labeled as:
	$$
	\{H^r_{\underline{J}}\}_{\underline{J}=1}^{H_r}
	$$
	from the bottom to the top with respect to the vertical orientation of  $\tilde{R}_r$.
\end{defi}

The next lemma is immediate.

\begin{lemm}\label{Lemm: parameter J, cW}
	For each $\underline{J}\in\{1,\cdots, H_r\}$, there exists a unique index $j_{(\underline{J})}$ such that $(i,j)\in\cH(r)$  and  $\tilde{H}^r_{\underline{J}}\subset H^i_j\cap \tilde{R}_r$. The index $j_{(\underline{J})}$  is determined by satisfy the next inequalities:
	
	\begin{equation}\label{Equa: Cut the index, cW}
		\sum_{\{j': (i,j')\in \cH(r)  \text{ and } j'<j_{(\underline{J})}\} }\underline{h}(j',r)< \underline{J} \leq  \sum_{\{j': (i,j')\in \cH(r)  \text{ and } j'\leq j_{(\underline{J})}\} }\underline{h}(j',r).
	\end{equation} 
	
\end{lemm}

\begin{defi}\label{Defi: Horizontal parameter, cW}
	Given $\underline{J}\in \{1,\cdots, H_r\}$ we call the index $j_{(\underline{J})}$ that was determined in Lemma \ref{Lemm: parameter J, cW} the \emph{horizontal parameter} of $\underline{J}$.
\end{defi}

\begin{defi}\label{Defi: relative index, cW }
	Let $r=\tilde{r}(i,s)$, $\underline{J}\in \{1,\cdots, H_r\}$, and let $j_{(\underline{J})}$ be the horizontal parameter of $\underline{J}$. The \emph{relative index} of $H^r_J$ inside of $\tilde{R}_r\cap H^i_{j_{(\underline{J}})}$ is given by:
	\begin{equation}\label{Equa: relative index, cW}
		J_{(\underline{J})}:=\underline{J}-\sum_{j'<j_{({\underline{J})}}}\underline{h}(r,j').
	\end{equation}
	
\end{defi}

\begin{defi}\label{Defi: relative position, cW }
	Let $\underline{J}\in \{1,\cdots, H_r\}$ the \emph{relative position} of $\underline{J}$ inside $\tilde{R}_r$ is the pair $(j_{\underline{J}},J_{(\underline{J})})$ of the horizontal parameter of $\underline{J}$ and the relative index of $H^r_J$ inside of $\tilde{R}_r\cap H^i_{j_{\underline{J}}}$
\end{defi}

\begin{lemm}\label{Lemm: relative position relation, cW}
	The rectangle $\tilde{H}^r_{j_{(\underline{J})},J_{(\underline{J})}}$ and the rectangle $\tilde{H}^r_{\underline{J}}$ are the same.
\end{lemm}


\begin{coro}\label{Lemm: The permutation, cW}
	Let $r=\tilde{r}(i,s)$ and $\underline{J}\in \{1,\cdots, H_r\}$. Let $(j_{(\underline{J})},J_{(\underline{J})})$ the relative position of $\underline{J}$ as was given in Definition \ref{Defi: relative position, cW }. Then the function $\rho_{S(\cW)}$ is given by the formula:
	\begin{equation}\label{Equa: The permutation, cW}
		\rho_{S(\cW)}(r,\underline{J}):=\underline{\rho}_{(r,j_{(\underline{J})})}(r,J_{(\underline{J})}).
	\end{equation}

\end{coro}

\textbf{The orientation $\epsilon_{S(\cW)}$}

\begin{lemm}\label{Lemm: The orientation, cW}
	Let $r=\tilde{r}(i,s)$ and $\underline{J}\in \{1,\cdots, H_r\}$. Let $j_{(\underline{J})}$ the horizontal parameter of $\underline{J}$ (Definition  \ref{Defi: Horizontal parameter, cW}). Then:
	\begin{equation}\label{Equa: permtation in the s refinament, cW}
		\epsilon_{S(\underline{w})}(r,\underline{J}):=\epsilon_T(i,j_{(\underline{J})}).
	\end{equation}
\end{lemm}

We can summarize our discussion in the following corollary.

\begin{coro}\label{Coro: Algoritm computation TS(w), cW}
	Let $T$ be a geometric type in the pseudo-Anosov class with a binary incidence matrix denoted as $A:=A(T)$. Let $\cW$ a family of periodic codes contained in $\Sigma_A$, there is algorithm to compute the   geometric type:
	
	\begin{equation*}
		T_{S(\cW)}:=\{N,\{H_r,V_r\}_{r=1}^N,\Phi_{S(\cW)}:=(\rho_{S(\cW)},\epsilon_{S(\cW)})\}.
	\end{equation*}
	
	using  the given geometric type $T$ and the codes in  $\cW$. 
\end{coro}

We claim that Theorem \ref{Theo: Refinamiento s frontera fam} has been essentially proved, so let us recapitulate our results.

\begin{theo*}[\ref{Theo: Refinamiento s frontera fam}]
	Given a geometric partition $\mathcal{R} = \{R_i\}_{i=1}^n$ of $f$, the geometric type $T$ of the pair $(f, \mathcal{R})$, and a finite family of periodic codes $\cW = \{\underline{w}^1, \cdots, \underline{w}^Q\} \subset \Sigma_{A(T)}$, there exists a finite algorithm to construct a geometric refinement of $\cR$ by horizontal sub-rectangles of the pair $(f, \mathcal{R})$, known as the $s$-\emph{boundary} refinement of $(f, \mathcal{R})$ with respect to $\cW$, denoted by $(f, \textbf{S}_{\cW}(\mathcal{R}))$, with the following properties:
	\begin{itemize}
		\item Every rectangle in the Markov partition $\textbf{S}_{\cW}(\mathcal{R})$ has as stable boundary components either a stable boundary arc of a rectangle $R_i$ in $\mathcal{R}$ or a stable arc $I_{t, \underline{w}}$ determined by the iteration $t$ of a code $\underline{w} \in \cW$. Moreover, every arc of the form $I_{t, \underline{w}}$ for $\underline{w} \in \cW$ is the stable boundary component of some rectangle in the refinement $\textbf{S}_{\cW}(\mathcal{R})$.
		\item The periodic boundary points of $\textbf{S}_{\cW}(\mathcal{R})$ are the union of the periodic boundary points of $\mathcal{R}$ and those in the set $\{\pi_{(f,\cR)}(\sigma_A^t(\underline{w})) \mid \underline{w} \in \cW \text{ and } t \in \NN\}$.
		\item There are explicit formulas, in terms of the geometric type $T$ and the codes in $\cW$, to compute the geometric type of $\textbf{S}_{\cW}(\mathcal{R})$.
	\end{itemize} 
\end{theo*}

\begin{proof}
	We have constructed $\mathcal{R}_{\mathcal{S}(\mathcal{W})}$ in Lemma \ref{Lemm: cR-S(cW) is markov part}, and it was done in such a manner that the rectangles obtained satisfy the first item of our proposition. Since the projection of the codes is the boundary of the rectangles in $\mathcal{R}_{\mathcal{S}(\mathcal{W})}$, it is necessary that, apart from the boundary periodic codes of $\mathcal{R}$, the boundary refinement also includes the projection of all the codes in $\mathcal{W}$ and, consequently, the whole orbits of such points.
	
	Finally, Corollary \ref{Coro: Algoritm computation TS(w), cW} establishes the existence of such an algorithm. Thus, we have completed the proof of our theorem.
\end{proof}

\subsection{\texorpdfstring{The $u$-boundary refinement.}{The u-boundary refinement.}}\label{Subsec: U-boundary refinement}

In this part, we are going to define the $u$-boundary refinement of a Markov partition along a family of periodic codes. This procedure involves cutting $\cR$ along the unstable segments that correspond to the orbit of any of the codes in $\cW$. In fact, we can define it using our previous $s$-boundary refinement applied to $T^{-1}$ and $f^{-1}$.

\begin{defi}\label{Defi: U boundary refinament}
	
	Let $T$ be a geometric type in the pseudo-Anosov class with an incidence matrix $A := A(T)$ that is binary. Consider a generalized pseudo-Anosov homeomorphism $f: S \rightarrow S$ with a geometric Markov partition $\cR$ of geometric type $T$.
	Let 
	$$
	\cW = \{\underline{w}^1, \cdots, \underline{w}^Q\},
	$$
	be a family of periodic codes that are non-$u$-boundary. 
	\begin{itemize}
		\item Let $\cR_{U(\cW)}$  the $s$-boundary refinement of $\cR$ when $\cR$ is viewed as a Markov partition of $f^{-1}$. i.e $(\cR,f^{-1})$. Denote by $T^{-1}_{S(\cW)}$ the  geometric type of $(\cR_{U(\cW)},f^{-1})$.

		\item Let $T_{U(\cW)}$ the geometric type of $(\cR_{U(\cW)},f)$.
		
	\end{itemize}
	
	In this case the geometry Markov partition $\cR_{U(\cW)}$ for $f$ called the $u$-boundary refinement of $\cR$ respect the family $\cW$.
\end{defi}

\begin{rema}\label{Rema: some properties of U-boundary refinament}
	
	We can make the following direct observations:
	
	\begin{enumerate}
		\item If $\underline{w}$ is not a $u$-boundary point for $(\cR,f)$, the $\underline{w}$ it is not a $s$-boundary point for $(\cR,f^{-1})$, which justifies our assumption.
		
		\item The geometric Markov partition $(\cR_{U(\cW)},f)$ has a geometric type $T_{U(\cW)}$, and this geometric type is the inverse of the geometric type of $(\cR_{U(\cW)},f^{-1})$, therefore:
		$$
		T_{U(\cW)} := (T^{-1}_{S(\cW)})^{-1}.
		$$
		and we can determine $T_{U(\cW)}$ in an algorithmic manner.
		
		\item For every $\underline{w}\in \cW$, the point $\pi_f(\underline{w})$ (where $\pi_f$ is the projection with respect to $(\cR,f)$) is a $u$-boundary point of $\cR_{U(\cW)}$.
		
		\item If $\underline{w}$ is a $u$-boundary point, clearly $\cR_{U(\underline{w})}=\cR$ and $T_{U(\underline{w})}=T$. This justifies taking any code in $\cW$ as a non $u$-boundary code.
	\end{enumerate}
\end{rema}

%% file: Corner/Corner.tex
\section{The corner refinement}\label{Section: Corner ref}

Given a couple of homomorphism/geometric Markov partition $(f, \cR)$ and a family of periodic codes $\cW$, we are interested in concatenating the $u$ and $s$ boundary refinements to produce a refinement of $(f, \cR)$, namely $(f, \cR_{\cC})$, with the corner property and such that the only periodic codes are those on the boundary of $\cR$. The following theorem describes this situation.

\begin{theo}\label{Theo: The corner refinamiento}
	Let $T \in \mathcal{T}(\textbf{p-A})^{SIM}$ and let $(f, \cR)$ be a pair that represents $T$. Then there exists a geometric Markov partition of $f$, denoted $\cR_{\cC}$, called the corner refinement of $(f, \cR)$ with the following properties:
	\begin{enumerate}
		\item The geometric Markov partition $\cR_{\cC}$ has the corner property.
		\item The periodic boundary codes of $(f, \cR_{\cC}$ are the boundary periodic points of $(f, \cR)$.
		\item There exists an algorithm and explicit formulas to compute the geometric type $T_{\cC}$ of $(f, \cR_{\cC}$.
	\end{enumerate}
\end{theo}

 In the next sub-section, we will develop the theory to determine the boundary codes of every geometric type, and then proceed to construct and compute the geometric type of the corner refinement. This intermediate step is necessary in order to obtain an algoritmic description od the geometric $T_{\cC}$

\subsection{Boundary codes and \texorpdfstring{$s,u$}{s,u}-generating functions}\label{Subsec: SU generating}

Let's proceed with the construction of the codes that are projected onto the boundary of the Markov partition, $\partial^{s,u}\cR$. We assume that $\cR=\{R_i\}_{i=1}^n$ is a geometric Markov partition of $f$ with geometric type $T$. For each $i\in \{1,\cdots,n\}$, we label the boundary components of $R_i$ as follows:

\begin{itemize}
	\item $\partial^s_{+1}R_i$ denotes the upper stable boundary of the rectangle $R_i$.
	\item $\partial^s_{-1}R_i$ denotes the lower stable boundary of $R_i$.
	\item $\partial^u_{-1}R_i$ denotes the left unstable boundary of $R_i$.
	\item $\partial^u_{+1}R_i$ denotes the right unstable boundary of $R_i$.
\end{itemize}

By using these labeling conventions, we can uniquely identify the boundary components of $\cR$ based on the geometric type $T$.

\begin{defi}\label{Defi; s,u boundary labels of T}
	Let $T=\{n,\{(h_i,v_i)\}_{i=1}^n,\Phi_T\}$  be an abstract geometric type. The $s$-\emph{boundary labels} of $T$ are defined as the formal set:
	$$
	\cS(T):=\{(i,\epsilon): i\in \{1,\cdots,n\} \text{ and } \epsilon\in \{1,-1\} \}, 
	$$
	Similarly, the $u$-\emph{boundary labels} of $T$ are defined as the formal set:
	$$
	\cU(T):=\{(k,\epsilon): k\in \{1,\cdots,n\} \text{ and } \epsilon\in \{1,-1\} \}
	$$
\end{defi}

In a while, we will justify such names. It's important to note that this definition was made using only the value of $n$ given by the geometric type $T$, so it does not depend on the specific realization. With these labels, we can formulate in terms of the geometric type the codes that $\pi_f$ projects to $\partial^{u,s}_{\pm 1 }R_i$. The first step is to introduce a \emph{generating function}

We begin by relabeling the stable boundary component of $\mathcal{R}$ using the next  function:
$$
\theta_T:\{1,\cdots,n\}\times \{1,-1\}\rightarrow \{1,\cdots,n\} \times\cup_{i=1}^n \{1,h_i\}_{i=1}^n \subset \cH(T)
$$ 
that is defined as:

\begin{equation}\label{Equa: theta T relabel}
	\theta_T(i,-1)=1 \text{ and } \theta(i,1)=h_i.
\end{equation}

The effect of $\theta_T$ is to choose the sub-rectangle of $R_i$ that contains the stable boundary component in question. Specifically, the boundary $\partial^s_{-1}R_1$ is contained in $\partial^s_{-1}H^i_1$, and the boundary $\partial^s_{+1}R_1$ is contained in $\partial^s_{+1}H^i_{h_i}$. This allows us to track the image of a stable boundary component, as the image of $\partial^s_{+1}R_i$ under $f$ should be contained in the stable boundary of $f(H^i_{h_i})$. But remember, the image of $\partial^s_{-1}R_i$ is a boundary component of $f(H^i_1) = V^k_l$, and the pair $(k,l)$ is uniquely determined by $\Phi_T$. By considering the value of $\epsilon_T(i,h_i)$, we can trace the image of the upper boundary component of $H^i_{h_i}$ in a more precise manner. If $\epsilon_T(i,h_i) = 1$, it indicates that the map $f$ does not alter the vertical orientations. As a result, the image of the upper boundary component of $H^i_{h_i}$ will remain on the upper boundary component of $f(H^i_{h_i})$. i.e. in this example:
$$
f(\partial^s_{+1}R_i)\subset f(\partial^s_{+1}H^i_{h_i}) \subset \partial^s_{+1}R_k.
$$ 
On the contrary, if $\epsilon_T(i,h_i)=-1$, it means that the map $f$ changes the vertical orientation. This has the following implication:
$$
f(\partial^s_{+1}R_i)\subset f(\partial^s_{+1}H^i_{h_i})\subset \partial_{-1}R_k.
$$

We don't really care about the index $l$ in $(k,l) \in \mathcal{V}(T)$, so it's convenient to decompose $\rho_T$ into two parts: $\rho_T:=(\xi_T,\nu_T)$. In this decomposition
$$
\xi_T:\cH(T)\rightarrow \{1,\cdots, n\},
$$
is defined as $\xi_T(i,j)=k$ if and only if $\rho_T(i,j)=(k,l)$.

Let's continue with our example. In order to determine where $f$ sends the upper boundary component $\partial^s_{+1}R_i$, we need to identify the rectangle in $\mathcal{R}$ to which $f$ maps $H^i_j$. This can be determined using the following relation:
$$
\xi_T(i,\theta_T(1))=\xi_T(i,h_i)=k.
$$ 
We now incorporate the change of orientation to determine the boundary component of $R_k$ that contains $f(\partial^s_{+1}H^i_{h_i})$. If we assume that $f(\partial^s_{+1}R_i) \subset \partial^s_{\epsilon}R_k$, then the value of $\epsilon$ can be determined using the following formula:
$$
\epsilon=+1 \cdot\epsilon_T(i,h_i) = +1 \cdot \epsilon_T(i,\theta_T(i,1)).
$$   
This procedure is illustrated in Figure  \ref{Fig: theta T}.
\begin{figure}[h]
	\centering
	\includegraphics[width=0.6\textwidth]{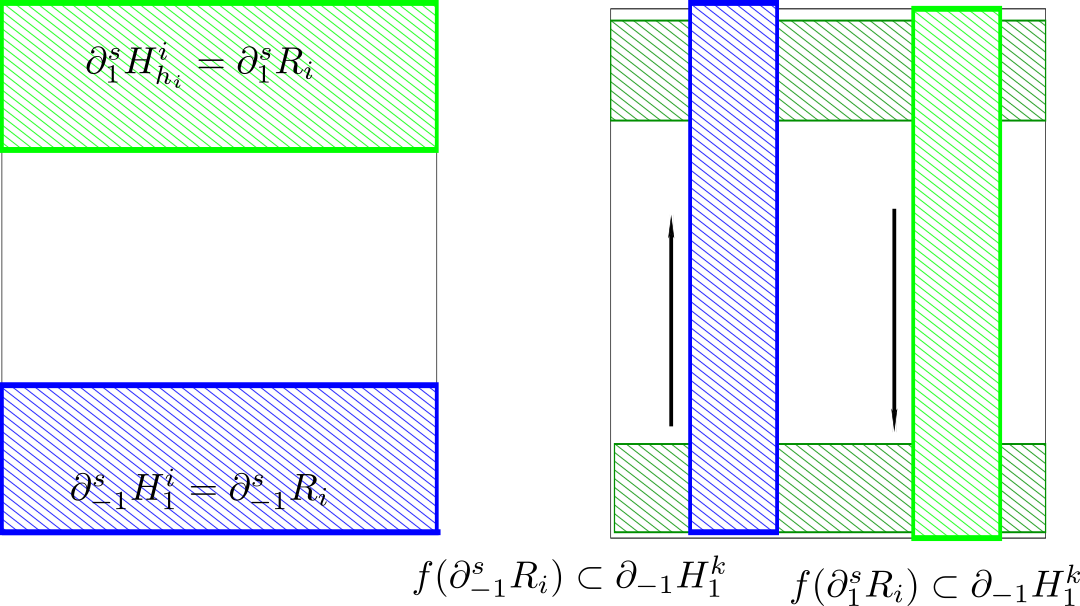}
	\caption{The effect of $\theta_T$}
	\label{Fig: theta T}
\end{figure}
The $S$-generating function utilizes $\theta_T$ and $\xi_T$ to generalize this idea.

\begin{defi}\label{Defi: s-boundary generating funtion}
	The $s$-generating function of $T$ is the function 
	$$
	\Gamma(T):\cS(T) \rightarrow \cS(T),
	$$
	defined for every $s$-boundary label $(i_0,\epsilon_0) \in \mathcal{S}(T)$ by the following formula:

	\begin{equation}\label{Equ: Gamma generation funtion}
		\Gamma(T)(i_0,\epsilon_0)=(\xi_T(i_0,\theta_T(i_0,\epsilon_0)),
		\epsilon_0 \cdot \epsilon_T(i_0,\theta_T(i_0,\epsilon_0)) ).
	\end{equation}
	
\end{defi}

The $u$-boundary generating function can be defined as either the $s$-boundary generating function of $T^{-1}$ (associated with $f^{-1}$), or it can be defined directly as follows.

\begin{defi}\label{Defi: u-boundary generating funtion}
	The $u$-\emph{generating function} of $T$ is
	$$
	\Upsilon(T):\cU(T)\rightarrow  \cU(T).
	$$
	defined as the $s$-\emph{generating function} of $T^{-1}$, $\Gamma(T^{-1})$.
\end{defi}

The orbit of a label $(i_0,\epsilon_0)\in \cS(T)$ under $\Gamma(T)$ is defined as:
$$
\{(i_m,\epsilon_m)\}_{m\in \NN}:=\{\Gamma(T)^m(i_0,\epsilon_0) : m\in\NN\}. 
$$
We define $\varsigma_T:\cS(T)\rightarrow {1,\dots,n}$ as the projection onto the first component of the funtion $\Gamma_T$, i.e $\varsigma_T(i,\epsilon)=k$ if and only if $\Gamma_T(i,j)=(k,\epsilon)$. We will focus on the $s$-generating function. The ideas for the unstable case naturally extend, but sometimes we are going to recall them too.

The generating functions allow us to create a positive (or negative) code in $\Sigma^+$ for every $s$-boundary label of $T$ by composing $\Gamma(T)$ ($\Upsilon(T)$) with $\varsigma_T$ to obtain:

\begin{eqnarray}\label{Equa: $s$-boundary code +}
	\underline{I}^{+}(i_0,\epsilon_0):=\{ \varsigma_T\circ\Gamma(T)^m(i_0,\epsilon_0)\}_{m\in\NN}=\{i_m\}_{m\in \NN}
\end{eqnarray}

Similarly, for $(k_0,\epsilon_0)\in \cU(T)$, we can construct the negative code in $\Sigma^{-}$ given by:

\begin{eqnarray}\label{Equa: $u$-boundary code -}
	\underline{J}^{-}(k_0,\epsilon_0):=\{ \varsigma_T\circ\Upsilon(T)^m(k_0,\epsilon_0)\}_{m\in\NN}=\{k_{-m}\}_{m\in \NN}
\end{eqnarray}

These positive and negative codes are of great importance and deserve a name.

\begin{defi}\label{Defi: positive negative boundary codes }
	
	Let $(i_0,\epsilon_0)\in \cS(T)$. The $s$-\emph{boundary positive code} of $(i_0,\epsilon_0)$ is denoted as $\underline{I}^+(i_0,\epsilon_0)\in \Sigma^+$, and it is determined by Equation \ref{Equa: $s$-boundary code +}. The set of $s$-boundary positive codes of $T$ is denoted as:
	$$
	\underline{\cS}^{+}(T)=\{\underline{I}^{+}(i,\epsilon): (i,\epsilon)\in \cS(T)\} \subset \Sigma^+.
	$$
	Let $(k_0,\epsilon_0)\in \cU(T)$. The $u$-\emph{boundary negative code }of $(k_0,\epsilon_0)$ is denoted as $\underline{J}^{-}(k_0,\epsilon_0)=\{k_{-m}\}_{m=0}^{\infty}$ and is determined by Equation \ref{Equa: $u$-boundary code -}. The set of $u$-boundary negative codes of $T$ is denoted as
	
	$$
	\underline{\cJ}^{-}(T)=\{\underline{J}^{-}(k,\epsilon): (k,\epsilon)\in \cU(T)\}\subset \Sigma^{-}
	$$
\end{defi}

Lemma \ref{Lemm: positive code admisible} below asserts that every $s$-boundary positive code of $T$ corresponds to the positive part of at least one element of $\Sigma_A$. The analogous result for $u$-boundary negative codes is true using a fully symmetric approach with the $u$-generating function. Therefore, we will only provide the proof for the stable case.

\begin{lemm}\label{Lemm: positive code admisible}
	Every $s$-boundary positive code $\underline{I}^{+}(i_0,\epsilon_0)$ belongs to $\Sigma_A^+$. Similarly, every $u$-boundary negative code $\underline{J}^{-}(k_0,\epsilon_0)$ belongs to $\Sigma_A^-$.
\end{lemm}

\begin{proof}
	The incidence matrix $A$ of $T$ has a relation to any realization $(f,\cR)$ of $T$, which we can exploit to visualize the situation, anyway observe that our arguments only use the properties of $A$ or $T$. 
	In the construction of the positive code $\underline{I}^{+}(i_0,\epsilon_0)$, the term $i_1$ corresponds to the index of the unique rectangle in $\cR$ such that $f(H^{i_0}_{\theta_T(1,\epsilon_0)})=V^{i_1}_{l_1}$ (where $\theta_T(i_0,\epsilon)=1$ or $h_i$). In either case, we have:
	
	$$
	f^{-1}(\overset{o}{R{i_1}})\cap \overset{o}{R_{i_0}}=\overset{o}{H^{i_0}_{\theta_T(1,\epsilon_0)}}\neq \emptyset,
	$$
	
	which implies $a_{i_0,i_1}=1$. By induction, we can indeed extend this result to the entire positive code $\underline{I}^{+}(i_0,\epsilon_0)$.
	
	Similarly, for the case of $u$-boundary negative codes, we can use the symmetric approach by considering $T^{-1}$ and the inverse of the incidence matrix $A^{-1}$. Additionally, we can consider the realization $(f^{-1},\cR)$ to obtain a visualization. 
\end{proof}

Proposition \ref{Prop: s code Inyective} is the first step in proving that each $s$-boundary label uniquely determines a set of codes that projects to a single boundary component of the partition.

\begin{prop}\label{Prop: s code Inyective}	
	The map $I:\cS(T)\rightarrow \Sigma_A^+$ defined by $I(i,\epsilon):=\underline{I}^+(i,\epsilon)$ is  injective map. Similarly, the map $J:\cU(T)\rightarrow \Sigma_A^-$ defined by $J(i,\epsilon):=\underline{J}^-(i,\epsilon)$ is also injective.
\end{prop}

\begin{proof}
	
	Clearly $I$ is a  well defined map as $\Gamma^n(T)$ is itself a well defined map with the same domain. 	If $i_0 \neq i'_0$, then the sequences $\underline{I}^+(i_0,\epsilon_0)$ and $\underline{I}^+(i'_0,\epsilon'_0)$ will differ in their first term, making them distinct. The remaining case involves considering $(i_0,1)$ and $(i_0,-1)$.
	
	Let's denote $\underline{I}^+(i_0,1) = {i_m}$ and $\underline{I}^+(i_0,-1) = {i'_m}$ as their positive codes. We will analyze the sequences ${\Gamma(T)^m(i_0,1)}$ and ${\Gamma(T)^m(i_0,-1)}$ and show that there exists an $m \in \NN$ such that $i_m \neq i'_m$. Our approach starts with the following technical lemma:
	
	\begin{lemm}\label{Lemm: positive h-i}
		If $T$ is in the pseudo-Anosov class and $A$ is a binary matrix, then there exists $M \in \mathbb{N}$ such that if  $\underline{I}^+(i_0,1) = \{i_m\}_{m\in \NN}$, then  $h_{i_M} > 1$.
	\end{lemm}
	
	\begin{proof}
		Since $T$ is in the pseudo-Anosov class, there exists a realization $(f,\mathcal{R})$. The infinite sequence ${i_m}$ takes a finite number of values (at most $n$), so there exist natural numbers $m_1 < m_2$ such that $R_{i_{m_1}} = R_{i_{m_2}}$. If such an $M$ does not exist, then for all $m_1 \leq m \leq m_2$, $h_{i_m} = 1$. This implies that $f^{m_2-m_1}(R_{i_{m_1}}) \subset R_{i_{m_2}} = R_{i_{m_1}}$, which is a vertical sub-rectangle of $R_{i_{m_1}}$.
		
		By uniform expansion in the vertical direction, the rectangle $R_{i_{m_1}}$ reduces to a stable interval, and this is not permitted in a Markov partition, leading us to a contradiction. 	
	\end{proof}
	
	In view of Lemma $M:=\min\{m\in \NN: h_{i_m}>1\}$ exists. If there is $0\leq m\leq M$ such that $i_m\neq i'_m$, our proof is complete. If not, the following Lemma addresses the remaining situation.
	
	\begin{lemm}\label{Lemm: diferen positive code}
		Suppose that for all $0\leq m\leq M$, $i_m=i_m'$, then $i_{M+1}\neq i'_{M+1}$.
	\end{lemm}
	
	\begin{proof}
		
		Observe that for all $0\leq m\leq M$, if $\Gamma(T)^m(i_0,+1)=(i_m,\epsilon_m)$ then $$\Gamma(T)^m(i_0,-1)=(i_m,-\epsilon_m)$$.
		
		They have the same index $i_m=i'_m$ but still have inverse $\epsilon$ part, i.e., $\epsilon_m=-\epsilon'_m$. In fact, the first term $\epsilon_0=-\epsilon'_0$, without lost of generality $\epsilon_0=1$  and $\epsilon'_0=-1$,   by hypothesis and considering that  $1=h_{i_0}=h_{i'_0}$ we infer that:
		$$
		\theta_T(i_0,\epsilon_0)=(i_0,h_{i_m})=(i'_0,1)=\theta_T(i'_0,\epsilon_0'),
		$$
		Therefore:	
		$$
		\epsilon_1=\epsilon_0\cdot \epsilon_T(i_0,h_{i_0})=-\epsilon'_0\epsilon_T(i'_0,1)=-\epsilon'_1.
		$$
		then continue the argument by induction. In particular: $\Gamma(T)^M(i_0,1)=(i_M,\epsilon_M)$ and  $\Gamma(T)^M(i'_0,-1)=(i_M,-\epsilon_M)$.
		
		The incidence matrix of $T$ have $\{0,1\}$ has coefficients ${0,1}$, and since that $1\neq h_{i_M}$,  if $\rho_T(i_M,1)=(k,l)$ then $\rho_T(i_M,h_{i_M})=(k',l')$ where $k\neq k'$.
		
		Consider the  case when, $\theta_T(i_M,\epsilon_M)=(i_M,1)$ and  $\theta_T(i_M,\epsilon'_M)=(i_M, h_{i_M})$. Lets apply the formula of $\Gamma(T)$:
		$$
		\Gamma(T)^{M+1}(i_0,1)=\Gamma(T)(i_M,\epsilon_M)=(\xi_T(i_M,1),\epsilon_M\cdot \epsilon_T(i_M,1))=(k,\epsilon_{M+1})
		$$
		and 				
		$$
		\Gamma(T)^{M+1}(i_0,-1)=\Gamma(T)(i_M,-\epsilon_M)=(\xi_T(i_M,h_{i_M}),-\epsilon_M\cdot \epsilon_T(i_M,h_{i_M}))=(k',\epsilon'_{M+1})
		$$	
		therefore $i_{M+1}=k\neq k'=i'_{M+1}$.
		
		The situation $\theta_T(i_M,\epsilon_M)=(i_M,h_{i_M})$ and  $\theta_T(i_M,\epsilon'_M)=(i_M, 1)$ is treated similarly. The has been lemma  proved.
	\end{proof}
	
	The proposition follows from the previous lemma. The result for negative codes associated with $u$-boundary labels is proven using a fully symmetric approach with the $u$-generating function.
\end{proof}

If $\Gamma(T)(i,\epsilon)=(i_1,\epsilon_1)$, applying the shift to this code gives another $s$-boundary positive code. That is clearly given by 
$$
\sigma(\underline{I}^+(i,\epsilon))=\underline{I}^+(i_1,\epsilon_1),
$$
where $\Gamma(T)(i_0,\epsilon_0)=(i_1,\epsilon_1)$. Since there are $2n$ different $s$-boundary positive codes, there exist natural numbers $k_1\neq k_2$ with $k_1,k_2\leq 2n$ such that $\sigma^{k_1}(\underline{I}^+(i,\epsilon))=\sigma^{k_2}(\underline{I}^+(i,\epsilon))$, and the code $\underline{I}^+(i,\epsilon)$ is pre-periodic. This implies the following corollary.

\begin{coro}\label{Coro: preperiodic finite s,u boundary codes}
	There are exactly $2n$ different $s$-boundary positive codes and $2n$ different $u$-boundary negative codes. Furthermore, every $s$-boundary positive code and every $u$-boundary negative code is pre-periodic under the action of the shift $\sigma$. 
	
	Moreover, for every $s$-boundary positive code $\underline{I}^+(i,\epsilon)$, there exists $k\leq 2n$ such that $\sigma^k(\underline{I}^+(i,\epsilon))$ is periodic. Similarly, for every $u$-boundary negative code $\underline{J}^-(i,\epsilon)$, there exists $k\leq 2n$ such that $\sigma^{-k}(\underline{J}^-(i,\epsilon))$ is periodic.
\end{coro}

Now we are ready to define a family of admissible codes that project onto the stable and unstable leaves of periodic boundary points of $\cR$.

\begin{defi}\label{Defi: s,u-boundary codes}
	The set of $s$-\emph{boundary codes} of $T$ is:
	\begin{equation}
		\underline{\cS}(T):=\{\underline{w}\in \Sigma_A: \underline{w}_+\in \underline{\cS}^{+}(T)\}.
	\end{equation}
	The set of  $u$-\emph{boundary codes} of $T$ is 
	\begin{equation}
		\underline{\cU}(T):=\{\underline{w}\in \Sigma_A: \underline{w}_-\in \underline{\cU(T)}^{-}(T)\}.
	\end{equation}
\end{defi}

Next, Proposition \ref{Prop: positive codes are boundary}  states that the projection of $s$-boundary codes of $T$ through $\pi^f$ is always contained in the stable boundary of the Markov partition $(f,\cR)$. Similarly, Proposition \ref{Prop: boundary points have boundary codes} ensures that these are the only codes in $\Sigma_A$ that project to the stable boundary. This provides us with a symbolic characterization of the boundary of the Markov partition. As we have done so far, we will provide a detailed proof for the stable case, noting that the unstable version follows by symmetry.

\begin{prop}\label{Prop: positive codes are boundary}
	Let $(f,\cR)$ be a pair consisting of a homeomorphism and a partition that realizes $T$. Suppose $(i,\epsilon)\in \cS(T)$ is an $s$-boundary label of $T$, and let $\underline{w}\in \underline{\cS}(T)$ be a code such that $\underline{I}^+(i,\epsilon)=\underline{w}_+$. Then, $\pi_f(\underline{w})\in \partial^s_{\epsilon} R_{i}$.
	
	Similarly, if $(i,\epsilon)\in \cU(T)$ is a $u$-boundary label of $T$, and $\underline{w}\in \underline{\cU}(T)$ is a code such that $\underline{J}^-(i,\epsilon)=\underline{w}_-$, then $\pi_f(\underline{w})\in \partial^u_{\epsilon}R_i$.
	
\end{prop}

\begin{proof}
	Make $(i,\epsilon)=(i_0,\epsilon_0)$ to make coherent the following notation and $\underline{I}^+(i,\epsilon)=\{i_m\}_{m\in \NN}$. 	For every $s\in \NN$ define the rectangles $\overset{o}{H_s}=\cap_{m=0}^s f^{-1}(\overset{o}{R_{i_m}})$. The limit of the closures of such rectangles when $s$ converge to infinity is a unique stable segment of $R_{i_0}$. The proof is achieved  if $\partial^s_{\epsilon_0}R_{i_0}\subset H_s:=\overline{ \overset{o}{H_s} }$ for all $s\in \NN$. In this order of ideas is enough to argument that for all $s\in \NN$:
	$$
	f^s(\partial^s_{\epsilon_0}R_{i_0}) \subset R_{i_{s}}.
	$$
	We are going to probe  by induction over $s$ something more specific:
	$$
	f^s(\partial_{\epsilon_0} R_{i_0})\subset \partial_{\epsilon_s}R_{i_s}.
	$$  
	
	\emph{Base of induction}: For $s=0$. This is the case as $f^0(\partial^s_{\epsilon_0}R_{i_0})\subset \partial^s_{\epsilon_0}R_{i_0}$.
	
	\emph{Hypothesis of induction }: Assume  that $f^s(\partial_{\epsilon_0} R_{i_0})\subset \partial_{\epsilon_s}R_{i_s}$ and  $f^s(\partial^s_{\epsilon_0}R_{i_0})\subset R_{i_s}$.
	
	\emph{Induction step}: We are going to prove that
	$$
	f^{s+1}(\partial_{\epsilon_0} R_{i_0})\subset \partial_{\epsilon_{s+1}}R_{i_{s+1}}
	$$ 
	
	and then  that $f^{s+1}(\partial^s_{\epsilon_0}R_{i_0})\subset R_{i_{s+1}}$. For that reason consider the two following cases:
	
	\begin{itemize}
		\item $\epsilon_{s}=1$. In this situation $f^s(\partial^s_{\epsilon_0}R_{i_0})\subset H^{i_s}_{h_{i_s}}$.  Hence  $f^{s+1}(\partial^s_{i_0}R_{i_0})\subset f(H^{i_s}_{h_{i_s}}) \subset R_{i'_{s+1}}$. Where $R_{i'_{s+1}}$ is the only rectangle such that $\xi_T(i_s,h_{i_s})=(i'_{s+1})$.		
		Even more $f^{s+1}(\partial^s_{\epsilon_0} R_{i_0})\subset \partial^s_{\epsilon'_{s+1}} R_{i'_{s+1}}$, where $\epsilon'_{s+1}$ obey to the formula $\epsilon'_{s+1}= \epsilon_T(i_s,h_{i_s})=\epsilon_s \cdot \epsilon_T(i_s,h_{i_s})$.
		
		\item  $\epsilon_{s}=-1$. In this situation  $f^s(\partial^s_{\epsilon_0}R_{i_0})\subset H^{i_s}_1$. Hence $f^{s+1}(\partial^s_{i_0}R_{i_0})\subset f(H^{i_s}_{1}) \subset R_{i'_{s+1}}$, where $R_{i_{s+1}}$ is the only rectangle such that $\xi_T(i_s,1)=(i'_{s+1}$. 
		Even more, $f^{s+1}(\partial^s_{\epsilon_0} R_{i_0})\subset \partial^s_{\epsilon'_{s+1}} R_{i'_{s+1}}$, where  $\epsilon'_{s+1}$ obey to the formula $\epsilon'_{s+1}= -\epsilon_T(i_s,1)=\epsilon_s \cdot \epsilon_T(i_s,1)$.
	\end{itemize}
	In bot situations:
	$$
	f^{s+1}(\partial^s_{\epsilon_0})R_{i_0})\subset \partial^s_{\epsilon'_{s+1}}R_{i'_{s+1}}
	$$
	and they follows the rule:
	$$
	(i'_{s+1},\epsilon'_{s+1})=(\xi_T(i_s,\theta_T(i_s,\epsilon_s)),\epsilon_s\cdot \epsilon_T(i_s,\theta_T(\epsilon_s)))=\Gamma(T)^{s+1}(i_0,\epsilon_0)=(i_{s+1},\epsilon_{s+1}).
	$$
	Therefore $f^{s+1}(\partial_{\epsilon_0} R_{i_0})\subset \partial_{\epsilon_{s+1}}R_{i_{s+1}}$, as we claimed and the result is proved.
	The unstable case is totally symmetric.
	
\end{proof}

Proposition \ref{Prop: positive codes are boundary} provides justification for naming the $s$-boundary and $u$-boundary labels of $T$ as such, as they generate codes that are projected to the boundary of the Markov partition. The next proposition asserts that these codes are the only ones that have such a property.

\begin{prop}\label{Prop: boundary points have boundary codes}
	If $\underline{w}\in \Sigma_A$ projects to the stable boundary of the Markov partition $(f,\cR)$ under $\pi_f$, i.e., $\pi_f(\underline{w})\in \partial^s \cR$, then it follows that $\underline{w}\in \underline{\cS}(T)$. Similarly, if $\pi_f(\underline{w})\in \partial^u\cR$, then $\underline{w}\in \underline{\cU}(T)$.
	
\end{prop}

\begin{proof}
	
	Like $A(T)$ has coefficients ${0,1}$, the sequence $\underline{w}_+$ determines, for all $m\in \mathbb{N}$, a pair $(w_m,j_m)\in \mathcal{H}(T)$ such that $\xi_T(w_m, j_m)=w_{m+1}$ and a number $\epsilon_T(w_m,j_m)=\epsilon_{m+1} \in {1,-1}$. It is important to note that, like  $f^m(x)\in \partial^s\cR$, in fact,  there are only two  cases (unless $h_{w_m}=1$ where they are the same):
	$$
	w_{m+1}=\xi_T(w_m,1) \text{ or well } w_{m+1}=\xi_T(w_m,h_{w_m}).
	$$
	
	This permit to define $\epsilon_m\in \{-1,+1\}$ as the only number such that:
	\begin{equation}\label{Equa: determine epsilon m}
		w_{m+1}=\xi_T(w_m,\theta_T(w_m,\epsilon_m)).
	\end{equation}
	Even more $\epsilon_m$ determine $\epsilon_{m+1}$ by the formula:
	$$
	\epsilon_{m+1}=\epsilon_m\cdot \epsilon_T(w_m,\theta_T(w_m,\epsilon_m)).
	$$
	In resume:
	\begin{equation}\label{Equa: w determine by gamma}
		\Gamma(T)(w_m,\epsilon_m)=(w_{m+1},\epsilon_{m+1})
	\end{equation}
	follow the rule dictated by the $s$-generating function. So if we know $\epsilon_M$ for certain $M\in \NN$ we can determine  $\sigma^{M}(\underline{w})_+=\underline{I}^+(w_M,\epsilon_M)$, so it rest to determine at least one $\epsilon_M$. 
	
	Lemma \ref{Lemm: positive h-i} can be adapted to this context to prove the existence of a minimal $M\in \mathbb{N}$ such that $h_{w_M}>1$. Then, equation \ref{Equa: determine epsilon m} determines $\epsilon_M$. Now, we need to recover $\epsilon_0$, but we proceed backwards. Since $h_{w_m}=1$ for all $m<1$, we have:

	$$
	\epsilon_M=\epsilon_{M-1}\cdot \epsilon_T(w_{M-1},1)
	$$
	and then $\epsilon_{M-1}=\epsilon_M \cdot \epsilon_T(w_{M-1},1)$. By applied this procedure we can determine $\epsilon_0$ and then using \ref{Equa: w determine by gamma} to get that
	$$
	\Gamma(T)^m(w_0,\epsilon_0)=(w_m,\epsilon_m).
	$$
	The conclusion is that $\underline{w}_+=\underline{I}^{+}(w_0,\epsilon_0)$. 
	
	The unstable boundary situation is analogous.
\end{proof}

Therefore, $\underline{\cS}(T)$ and $\underline{\cU}(T)$ are the only admissible codes that project to the boundary of a Markov partition. With this in mind, we can distinguish the periodic boundary codes from the non-periodic ones. It is important to note that the cardinality of each of these sets is less than or equal to $2n$.

\begin{defi}\label{Defi: s,u-boundary periodic codes}
	The set of $s$-\emph{boundary periodic codes} of $T$ is:
	\begin{equation}
		\text{ Per }(\underline{\cS(T)}):=\{\underline{w}\in \underline{\cS}(T): \underline{w} \text{ is periodic }\}.
	\end{equation}
	The set of  $u$-\emph{boundary periodic codes} of $T$ is 
	\begin{equation}
		\text{ Per }(\underline{\cU(T)}):=\{\underline{w}\in \underline{\cU}(T): \underline{w} \text{ is periodic }\}.
	\end{equation}
\end{defi}

\subsubsection{\texorpdfstring{Decomposition of $\Sigma_{A}$}{Decomposition of Sigma_{A}}}

Now we can describe the stable and unstable leaves of $f$ corresponding to the periodic points of the boundary.

\begin{defi}\label{Defi: stratification Sigma A}
	We define the $s$-\emph{boundary leaves codes} of $T$:
	\begin{equation}
		\Sigma_{\cS(T)}=\{\underline{w}\in \Sigma_A:   \exists k\in\NN \text{ such that } \sigma^k(\underline{w})\in \underline{\cS(T)}\}.
	\end{equation}
	We define the $u$-\emph{boundary leaves codes} of $T$
	\begin{equation}
		\Sigma_{\cU(T)}=\{\underline{w}\in \Sigma_A:  \exists k\in\NN \text{ such that } \sigma^{-k}(\underline{w})\in \underline{\cU(T)}\}.
	\end{equation}
	Finally, the \emph{totally interior codes} of $T$ are
	$$
	\Sigma_{\cI nt(T)}=\Sigma_A\setminus(\Sigma_{\cS(T)} \cup \Sigma_{\cU(T)}).
	$$
\end{defi}

The importance of such a division of $\Sigma_A$ is that its projections are well determined, as indicated by the following lemma.

\begin{lemm}\label{Lemm: Projection Sigma S,U,I}
	
	A code $\underline{w}\in \Sigma_A$ belongs to $\Sigma_{\cS(T)}$ if and only if its projection $\pi_f(\underline{w})$ is within the stable leaf of a boundary periodic point of $\cR$.
	
	A code $\underline{w}\in \Sigma_A$ belongs to $\Sigma_{U(T)}$ if and only if its projection $\pi_f(\underline{w})$ lies within the unstable leaf of a boundary periodic point of $\cR$.

	A code $\underline{w}\in \Sigma_A$ belongs to $\Sigma_{\cI nt(T)}$ if and only if its projection $\pi_f(\underline{w}$ is contained within Int$(f,\cR)$>
\end{lemm}

\begin{proof}
	The $s$-boundary leaf codes satisfy Definition  \ref{Defi: s,u-leafs}, and according to Proposition \ref{Prop: Projection foliations}, if $\underline{w}\in \Sigma_{\cS(T)}$, their projection lies on the same stable manifold as an $s$-boundary component of $\cR$. Stable boundary components of a Markov partition represent the stable manifold of an $s$-boundary periodic point, and for each $\underline{w}\in \Sigma_{\cS(T)}$, there exists $k=k(\underline{w})\in \mathbb{N}$ such that $\sigma^k(\underline{w})_+$ is a periodic positive code (Corollary \ref{Coro: preperiodic finite s,u boundary codes}). This positive code corresponds to a periodic point on the boundary of $\cR$ within whose stable manifold $\pi_f(\underline{w})$ is contained. This proves  one direction in the first assertion  of the lemma.

	If $\underline{v}$ is a periodic boundary code ,$\pi_f(\underline{v})\in \partial^s\cR$, by definition $\underline{v}\in \underline{\cS(T)}$.  Suppose that $\pi_f(\underline{w})$ is on the stable leaf of $\underline{v}$, then there exist $k \in \NN$ such that $\pi_f(\sigma^k(\underline{w}))$ is on the same stable boundary of $\cR$ as $\pi_f(\underline{v})$.  The proposition\ref{Prop: boundary points have boundary codes} implies that $\sigma^k(\underline{w})\in \underline{\cS(T)}$, hence  $\underline{w}\in \Sigma_{\cS(T)}$ by definition. This completes the proof of the first assertion of the lemma. A similar argument proves the unstable case.	
	
	The conclusion of these items is that  $\underline{w}\in \Sigma_{\cS(T)}\cap \Sigma_{\cU(T)}$ if and only if $\pi_f(\underline{w})\in \cF^{s,u}(\text{ Per }^{s,u}(\cR))$. 
	
	As proved in the Lemma \ref{Lemm: Caraterization unique codes} totally interior points are disjoint from the stable and unstable lamination generated by boundary periodic points. If $\underline{w}\in  \Sigma_{\cI nt(T)}$ and $\pi_f(\underline{w})$ is on the stable or unstable leaf of a $s,u$-boundary periodic point, we have seen that $\underline{w}\in \Sigma_{\cS(T),\cU(T)}$ which is not possible, therefore $\pi_f(\underline{w})\in $Int$(f,\cR)$. In the conversely direction, if  $\pi_f(\underline{w})$ is not in the stable or unstable lamination of $s,u$-boundary points, $\underline{w}\notin (\Sigma_{\cS(T)})\cup \Sigma_{\cU(T)}$, so $\underline{w}\in\Sigma_{\cI nt(T)}$. This ends the proof.
\end{proof}
We have obtained the decomposition 
$$
\Sigma_{A(T)}=\Sigma_{\cI nt(T)} \cup \Sigma_{\cS(T)} \cup \Sigma_{\cU(T)}
$$ 
and have characterized the image under $\pi_f$ of each of these sets. In the next subsection we use this decomposition to define relations on each of these parts and then extend them to an equivalence relation in $\Sigma_A$.

\subsection{The corner refinement}

Let $T$ be a geometric type within the pseudo-Anosov class with binary incidence matrix $A(T)$. Consider a generalized pseudo-Anosov homeomorphism $f: S \to S$ with a geometric Markov partition $\cR$ of geometric type $T$. The projection $\pi_{(f,\cR)}: \Sigma_{A(T)} \to S$, as defined in \ref{Defi: projection pi}, depends on the specific geometric type $T$ of the Markov partition. We will relabel it as $\pi_T: \Sigma_{A(T)} \to S$ to enable easier comparison with another projection $\pi_{T'}$ associated with a different geometric type $T'$

For a given geometric type $T$, we have defined its $s$ and $u$ generating functions in \ref{Defi: s-boundary generating funtion}  and \ref{Defi: u-boundary generating funtion}, respectively. Iterating these functions generates codes in $\Sigma_{A(T)}$, which project to the stable and unstable boundary components of $\cR$. Furthermore, as described in sub-section \ref{Subsec: SU generating}, using $T$,  it is feasible to determine the sets of periodic $s$-boundary codes of $T$:
$$
\textit{Per}(\underline{S(T)}):=\{\underline{w}\in \Sigma_{A(T)} : \underline{w} \text{ is periodic for $\sigma$ and } \pi_T(\underline{w})\in \partial^s \cR \}.
$$
and the sets of periodic $u$-boundary codes:
$$
\textit{Per}(\underline{(T)}):=\{\underline{w}\in \Sigma_{A(T)} : \underline{w} \text{ is periodic for $\sigma$ and } \pi_T(\underline{w})\in \partial^u \cR \}.
$$
Their union is the set of boundary periodic codes:
$$
\textit{Per}(\underline{B(T)}):=  \textit{Per}(S(T)) \cup  \textit{Per}(U(T)).
$$

In order to avoid to heavy notations we  relabel these sets:
$$
\underline{S(T)}:=\textit{Per}(\underline{S(T)}), \, \, \underline{U(T)}:=\textit{Per}(\underline{U(T)}) \text{ and } \underline{B(T)}:=\textit{Per}(\underline{B(T)})  
$$

Finally a corner periodic code  is any code in the  intersection:
$$
\underline{C(T)}:=\underline{S(T)} \cap \underline{U(T)}.
$$

Definition \ref{Defi: s,u-boundary periodic codes} reveals that $ \underline{B(T)}$ can be  described in a combinatorial manner using $T$, avoiding  the use of the Markov partition or the homeomorphism $f$. Based on the theory developed in Section \ref{Subsec: SU generating}, we can draw the following conclusion, which we present in the form of a corollary:

\begin{coro}\label{Coro: algoritmic per codes}
	Given a specific geometric type $T$, we can compute the sets: $\underline{(S(T)})$, $\underline{U(T)}$, $\underline{B(T)}$, and $\underline{C(T)}$, in terms of the geometric type $T$.
\end{coro}

\begin{defi}\label{Defi: Corner type}
	A geometric type $T$ have the \emph{corner property} if every periodic boundary code is a  corner peridic code, i.e $\underline{B(T)}=\underline{S(T)}\cap \underline{U(T)}$.
\end{defi}

We fix the following notation:
\begin{itemize}
	\item $P_s(f)$ is the set of $s$-boundary periodic points of $(f,\cR)$.
	\item $P_u(f)$ is the set of $u$-boundary periodic points of $(f,\cR)$
	\item  $P_u(f)=P_s(f)\cup P_u(f)$  is the set of periodic boundary  codes of $(f,\cR)$, and
	\item $P_c(f)=P_s(f)\cap P_u(f)$ is the set of corner points of $\cR$.
\end{itemize}

\begin{defi}\label{Defi: Corner type partition }
	A geometric Markov partition $\cR$ of $f$ have the \emph{corner property} if every rectangle $R\in \cR$ that contains a boundary point $p\in P_b(f,\cR)$ have $p$ as a corner point.
\end{defi}

A corner point $p$ of a rectangle $R \in \cR$ is on the boundary of another rectangle $R' \in \cR$ but it is not necessarily a corner point of $R'$; it can be on the interior of the unstable boundary of $R'$, for example. In this situation, $\pi_T^{-1}(p) \subset \Sigma_{A(T)}$ contains a corner code and a $u$-boundary code that is not an $s$-boundary code. Therefore, $\underline{C(T)} \neq \underline{B(T)}$ and $T$  don't have the corner property. This discrepancy makes it necessary to introduce the following lemma

\begin{lemm}\label{Lemm: Corner prop equivalence}
	The geometric Markov partition $\cR$ has the corner property if and only if its geometric type $T$ has the corner property.
\end{lemm}

\begin{proof}

	If $\cR$ satisfies the corner property, and we have an $s$-boundary code $\underline{w} \in \underline{S(T)}$, then by  Lemma  \ref{Lemm: Projection Sigma S,U,I}, $\pi_T(\underline{w})$ is a stable boundary periodic point contained in the rectangle $R_{w_0}$, where $w_0=\underline{w}_0$, is $0$ therm of the code $\underline{w}$. However, by hypothesis, if $p$ is  a $s$-boundary point  of $R_{w_0}$, it is a corner point of $R_{w_0}$,   then, $p$  is a $u$-boundary point. We can use Lemma  \ref{Lemm: Projection Sigma S,U,I} to deduce that the code $\underline{w}$ is also a $u$-boundary code. So, $\underline{w} \in \underline{C(T)}$. The case when the code $\underline{w}$ is a $u$-boundary code, i.e., $\underline{w} \in \underline{U(T)}$, is similarly proven.
	
	On the other hand, assume  $T$ satisfies the corner property. Let $p\in R_{w_0}\in \cR$ a $s$-boundary point of $R$. There exist a $s$-boundary code $\underline{w}\in \pi^{-1}_T(p)$ such that $\underline{w}_0=w_0$. Like $T$ have the corner property, any $s$-boundary periodic code is $u$-boundary, which in turn means $\underline{w}$ is a $u$-boundary code. As a consequence of Lemma  \ref{Lemm: Projection Sigma S,U,I}, $p=\pi_T(\underline{w})\in R_{w_0}$ is a $u$-boundary point contained in $R_{w_0}$, therefore $p$ is a corner point of such rectangle, $R_{w_0}$. The situation when $p$ is $s$-boundary point of $R_{w_0}$ is proved in the same manner.
	
	On the other hand, assume that $T$ satisfies the corner property. Let $p \in  \in \partial^s R_{w_0}$ be an $s$-boundary point of $R_{w_0}$. There exists an $s$-boundary code $\underline{w} \in \pi^{-1}_T(p)$ such that $\underline{w}_0=w_0$. Since $T$ has the corner property, any $s$-boundary periodic code is also a $u$-boundary code, which, in turn, means that $\underline{w}$ is a $u$-boundary code. As a consequence of \ref{Lemm: Projection Sigma S,U,I}, $p=\pi_T(\underline{w})\in \partial^s R{w_0}$ is a $u$-boundary point contained in $R{w_0}$; therefore, $p$ is a corner point of such a rectangle, $R_{w_0}$. The situation when $p$ is an $s$-boundary point of $R_{w_0}$ is proved in the same manner.
	
\end{proof}

\subsubsection{The construction of the corner refinement}

The first goal of this section is to provide a method for creating a geometric Markov partition $\cR_{\cC}$ for the homeomorphism $f$, along with its corresponding geometric type $T_{\cC}$, in such a way that both $\cR_{\cC}$ and $T_{\cC}$ satisfy the corner property.

Let $T$ be a geometric type in the pseudo-Anosov class with a binary incidence matrix $A(T)$. Consider a generalized pseudo-Anosov homeomorphism $f: S \to S$ with a geometric Markov partition $\cR$ of geometric type $T$. The procedure is as describe in the following paragraph:

\begin{enumerate}
	\item Get the $s$-boundary refinement of $\cR$ along the family of boundary periodic  of codes $\underline{B(T)}$ to obtain:
	$$
	\cR_{\cS(\underline{B(T)})}.
	$$
	
	and its geometric type: $T_{\cS(\underline{B(T)})}$.
	
	\item Compute the set of boundary periodic codes of the geometric type $T_{\cS(\underline{B(T)})}$ (Corollary \ref{Coro: algoritmic per codes}):
	$$
	\underline{S(T_{\cS(\underline{\cB(T)})})} \subset \Sigma_{A(T_{\cS(\underline{B(T)})})}.
	$$
	
	\item Obtain the $u$-boundary refinement of $\cR_{\cS(\underline{B(T)})}$ along the family of boundary periodic codes $\underline{S(T_{\cS(\underline{B(T)})})}$ to obtain the corner refinement of $\cR$:
	
	\begin{equation}\label{Equa: Corner refinament}
		\cR_{\cC}:=[\cR_{\cS(\underline{B(T)})}]_{\cU(\underline{S(T_{S(\underline{B(T)})})})}.
	\end{equation}
	
	whose geometric type is $T_{\cC}$.
\end{enumerate}

\begin{defi}\label{Defi: Corner refi}
	The geometric Markov partition for $f$, $\cR_{\cC}$, described by equation (\ref{Equa: Corner refinament}), is the \emph{Corner refinement} of the geometric Markov partition $(f,\cR)$, and its geometric type $T_{\cC}$ is the \emph{Corner refinement} of $T$.
\end{defi}

Proposition \ref{Prop: Corner ref periodic boundary points} implies Theorem \ref{Theo: The corner refinamiento}:

\begin{prop}\label{Prop: Corner ref periodic boundary points}
Let $T\in \cG\cT(\textbf{p-A})^{SIM}$. Consider a \textbf{p-A} homeomorphism $f: S \to S$ with a geometric Markov partition $\cR$ of geometric type $T$. Let $(f, \cR_{\cC})$ be the corner refinement of $(f, \cR)$, and let $T_{\cC}$ be the geometric type of $(f, \cR_{\cC})$. Then:
\begin{itemize}
	\item[i)] The boundary periodic points of $(f, \cR_{\cC})$ and the boundary periodic points of $(f, \cR)$ are exactly the same. In simpler terms: $P_b(f, \cR_{\cC}=P_b(f, \cR)$.
	
	\item[ii)] The Markov partition $(f, \cR_{\cC})$ exhibits the corner property.
	
	\item[iii)] There exists an algorithm to compute the geometric type $T_{\cC}$, which consists of computing the geometric type $T_{\cS(B(T))}$ of the $s$-boundary refinement of $(f, \cR)$ along the family of periodic boundary codes of $\cR$, and then computing the geometric type of the $u$-boundary refinement of $(f, \cR_{\cS(\underline{B(T)})})$ along the family of periodic boundary codes of $\cR_{\cS(\underline{B(T)})}$.
	\end{itemize}

\end{prop}

\begin{proof}
	
	\textbf{Item} $i)$. For the $s$-boundary refinement, the following equality holds:
	$$
	P_b(f,\cR_{S(\underline{B(T)}))}) : = P_b(f,\cR)\cup \pi_T(\underline{B(T)})=P_b(f,\cR).
	$$
	As a consequence, the boundary points of $(f, \cR_{\cC})$ coincide with the boundary points of $(f, \cR).$

	\textbf{Item} $ii)$. Let $p$ be a boundary point of $\cR_{S(\underline{B(T)})}$, and let $R \in \cR_{S(\underline{B(T)})}$ be a rectangle that contains $p$. Let $R_{0} \in \cR$ be a rectangle of $\cR$ that also contains $R$. After the $s$-boundary refinement:
	
	\begin{enumerate}
		\item If $p$ belonged to both $\partial^s R_{0}$ and $\partial^u R_0$, then $p$ becomes a corner point of the rectangle $R$.
		\item If $p$ were  in $\partial^s R_{0}$ but not in $\partial^u R_0$, it remains in $\partial^s R$ but not in $\partial^u R$. There is no change.
		\item If $p$ was in $\partial^u R_{0}$ but not in $\partial^s R_0$, then it becomes a corner point of $R$. 
	\end{enumerate}

	Let $p$ be a periodic boundary point in $\cR_{\cC}$. Consider a rectangle $R$ in $$
	\cR_{\cC}=[\cR_{S(\underline{B(T)}}]_{U(\underline{B(T_{S(\underline{B(T)})})})},
	$$
	that contains $p$, along with a rectangle $R_0$ in $R_{0} \in \cR_{S(\underline{B(T)})}$  that contains $R$. After the $u$-boundary refinement, we can have two possible outcomes
	
	\begin{enumerate}
		\item 	If $p\in \partial^s R_0\cap \partial^u R_0$  was a corner point of $R_{w_0}$, then the rectangle $R$ has $p$ as one of its corner points.
		
		\item 	If $p\in \partial^u R_0 \setminus \partial^s R_0$, after the $u$-boundary refinement, $p$ is in $\partial^s R \cap \partial^u R$ and is a corner point of $R$.
	\end{enumerate}

	In conclusion, the geometric Markov partition $(f, \cR_{C(T)})$ exhibits the corner property, as asserted in Item $ii$.
	
	\textbf{Item} $iii)$. This is a consequence of the algorithmic computation of the geometric types in the $s$ and $u$ boundary refinements, and the algorithmic computation of the family of boundary periodic codes for any geometric type, as explained in Subsection \ref{Subsec: SU generating}.

\end{proof}

\subsection{The corner refinement along a family of periodic codes.}
 
Let $\cR$ be a geometric Markov partition of $f: S \to S$ with geometric type $T$ that has the corner property. Take a family of periodic codes $\cW\subset \Sigma_{A(T)}$ and the family $P_{\cW} = \pi_T(\cW)$ of periodic points for $f$. The $s$-boundary refinement of $\cR$ concerning this family yields $\cR_{\cS(\cW)}$ and its geometric type $T_{\cS(\cW)}$. The boundary points of $(f,\cR_{\cS(\cW}$)  are given by the union of the boundary points of $\cR$ and the periodic points determined by the projections of the codes in $\cW$:
$$
P_b(f,\cR_{S(\cW)})=P_b(f,\cR)\cup P_{\cW}
$$ 

Let  $\cR_{\cC_{\cW})}$ be the corner refinement of  $\cR_{S(\cW)}$ with geometric type $T_{\cC(\cW)}$. A corollary of Proposition \ref{Prop: Corner ref periodic boundary points}, when applied to the geometric type $T_{S(\cW)}$, is as follows,and in particular it implies Theorem \ref{Theo: The corner refinamiento}.

\begin{coro}\label{Coro: Boundary refinament}
	Let $\cR_{ \cC(\cW) }$  be the corner refinement of $(f,\cR_{S(\cW)})$ with geometric type $T_{\cC(\cW)})$. Then:
	
	\begin{itemize}
		\item[i)] The boundary periodic points of $(f, \cR_{\cC(\cW))})$ and the boundary periodic points of $(f, \cR_{S(\cW)})$ are exactly the same. So they are the union of the boundary periodic points of $(f,\cR)$ such that are the protection of the family $\cW$ by $\pi_{(f, \cR)}$.
		
		\item[ii)] The Markov partition $(f, \cR_{\cC(\cW))})$ exhibits the corner property.
		
		\item[iii)] There are formulas and an algorithm to compute the geometric type $T_{\cC(\cW)}$.
	\end{itemize}
\end{coro}

\subsubsection{The corner refinement in  codes of bounded period} 

Finally, let us make a construction that will help us during the discussion of the Béguin algorithm (\cite{beguin2002classification}) in the article 

\begin{defi}\label{Defi: boudary of period P }
	
	Let $T$ be a geometric type in the pseudo-Anosov class with a binary incidence matrix $A(T)$, and let $f: S \to S$ be a generalized pseudo-Anosov homeomorphism with a geometric Markov partition $\cR$ of geometric type $T$. Assume that $(f, \cR)$ (and $T$) has the corner property. Now, consider the following definitions:
	
	\begin{itemize}
		\item Let  $P_{\underline{B(T)}}$ be  the maximum period among the boundary periodic codes of $T$:
		$$
		P_{\underline{B(T)}}:=\max\{\textbf{Per}(\underline{w},\sigma_{T}): \underline{w}\in \underline{B(T)}\}
		$$
		\item For any integer $P \geq P_{\underline{B(T)}}$, define $\cW_P$ as the set of admissible codes with periods less than or equal to $P$:
		$$
		\cW_{P}:=\{\underline{w}\in \Sigma_{A(T)}: \textbf{Per}(\underline{w},\sigma_{T})\leq P\}.
		$$
		\item The corner refinement of $\cR_{S(\cW_P)}$ along the family $\cW_P$  is  given by  $\cR_{\cC(\cW_P)}$ and its  geometric type is denoted as $T_{\cC(P)}$.
	\end{itemize}
	
\end{defi}

\begin{lemm}\label{Lemm: relation periods projection refinament C}
	Let $P\geq P_{\underline{B(T)}}$. Suppose $p\in S$ is a periodic point of $f$ with a period less than or equal to $P$; i.e., $\textbf{Per}(p,f)\leq P$. Consider $\underline{w}\in \Sigma_{A(T)}$ as a code that projects to $p$: $\pi_T(\underline{w})=p$. Then, the period of $\underline{w}$ is less than or equal to $P$, and $\underline{w}\in \cW_{P}$.
	
	Conversely, if $\underline{w}\in \cW_{P}$, then the point $p:=\pi_T(\underline{w})\in S$ is a periodic point of $f$ with a period $\textbf{Per}(p,f)$ less than or equal to $P$. In this manner:
	
	\begin{equation}\label{Equa: Projection codes an point period P}
		\pi_T(\cW_{P})=\{p\in S : \textit{Per}(p,f)\leq P\} \text{ and } \pi_T^{-1}(\{p\in S : \textit{Per}(p,f)\leq P\})=\cW_P
	\end{equation}

\end{lemm}

\begin{proof}
	Let $p\in S$ be a periodic point of $f$ with period less or equal than $P$. We have two possible situations: 
	\begin{itemize}
		\item The periodic point $p$ is a boundary point. In this case, $\underline{w}\in \underline{B(T)}\subset \cW_{P}$ just by definition of $\cW_P$. Let's recall that any code in $\underline{B(T)}$ has a period less than or equal to $P_{\underline{B(T)}}\leq P$.
		
		\item The periodic point $p$ is interior. In this case, $\pi^{-1}(p)$ contains a unique code $\underline{w}$ whose period coincides with the period of $p$, i.e., $\textbf{Per}(\underline{w},\sigma_{T})=\textbf{Per}(p,f)\leq P$. Therefore, $\underline{w}\in \cW_{P}$.
		
	\end{itemize}
	
	This implies that:
	\begin{equation}\label{Equa: period codes and points, I}
		\pi_T^{-1}(\{p\in S : \textit{Per}(p,f)\leq P\}) \subset \cW_P
	\end{equation}
	and 
	
	\begin{equation}\label{Equa: period codes and points, II}
		\{p\in S : \textit{Per}(p,f)\leq P\}= \pi_T(\pi_T^{-1}(\{p\in S : \textit{Per}(p,f)\leq P\}))  \subset \pi_T(\cW_P).
	\end{equation}
	
	For any code $\underline{w}\in \cW_{P}$ with a period $Q$ no greater than $P$, we can observe the following relationship:
	
	$$
	p=\pi_T(\underline{w})=\pi_T(\sigma^Q(\underline{w}))=f^Q(\pi_T(\underline{w}))=f^Q(p)
	$$
	
	This implies that $P\geq Q \geq\textbf{Per}(p,f)$, i.e., the period of $p$ with respect to $f$ is less than or equal to $P$. Then:
	
	\begin{equation}\label{Equa: period codes and points, II, b}
		\pi_{T}(\cW_P)  \subset \{p\in S : \textit{Per}(p,f)\leq P\}.
	\end{equation}
	
	This inclusion, along with the information provided in \ref{Equa: period codes and points, II}, implies that:
	$$
	\pi_{T}(\cW_P)  = \{p\in S : \textit{Per}(p,f)\leq P\}.
	$$

	Furthermore, if $\underline{w}\in \cW_P$, trivially $\underline{w}\in \pi_T^{-1}(\pi_T(\underline{w}))$ and $p:=\pi_T(\underline{w})$ has a period less than or equal to $P$. In this manner:
	
	\begin{equation}\label{Equa: period codes and points, I, b}
		\cW_P \subset \pi_T^{-1}(\{p\in S : \textit{Per}(p,f)\leq P\}).
	\end{equation}
	Using this information and the contention provided in \ref{Equa: period codes and points, I}, we establish the second equality of our lemma:
	$$
	\pi_T^{-1}(\{p\in S : \textit{Per}(p,f)\leq P\})=\cW_P.
	$$
	
	This ends our proof.
\end{proof}

\begin{coro}\label{Coro: P boundary refinament}
	When $P\geq P_{B(T)}$, the periodic boundary points of the corner refinement $(f, \cR_{C(T_{S(\cW_P)})})$ coincides with the set of all periodic points of $f$ having periods less than or equal to $P$ i.e.,
	$$
	P_b(f,\cR_{C(T_{S(\cW_P)})} )=\{p\in S : \textit{Per}(p,f)\leq P\}. 
	$$
	Furthermore, $(f,\cR_{C(T_{S(\cW_P)})} )$ also possesses the corner property.
\end{coro}

\begin{proof}
	From Corollary \ref{Coro: Boundary refinament}: 
	$$
	P_b(f, \cR_{C(T_{S(\cW_P)})}) =P_b(f, \cR_{S(\cW_P)}).
	$$
	But, we recall that:
	$$
	P_b(f, \cR_{S(\cW_P)})= P_b(f,\cR) \cup \pi_T(\cW_{P}).
	$$
	In view of Lemma \ref{Lemm: relation periods projection refinament C} 
	$$
	\pi_T(\cW_P) = \{p\in S : \textit{Per}(p,f)\leq P\}
	$$ 
	and clearly $P_b(f,\cR)\subset \{p\in S : \textit{Per}(p,f)\leq P\}$. Therefore: 
	$$
	P_b(f, \cR_{C(T_{S(\cW_P)})}) =P_b(f, \cR_{S(\cW_P)})=P_b(f,\cR) \cup \pi_T(\cW_{P})= \{p\in S : \textit{Per}(p,f)\leq P\}.
	$$
\end{proof}